\documentclass[a4paper,dvipsnames]{article}
\usepackage[utf8]{inputenc}
\usepackage[a4paper,top=3cm,bottom=2cm,left=3cm,right=3cm,marginparwidth=1.75cm]{geometry}

\usepackage{booktabs} 
\usepackage{setspace} 
\usepackage{cancel}   
\usepackage{enumitem} 
\usepackage{lipsum}   
\usepackage{etoolbox} 
\usepackage[margin=2cm]{caption} 

\usepackage[backend=biber, hyperref=true, style=numeric, sorting=nyt, doi=true, isbn=false, url=false, maxbibnames=9, sortcites=true]{biblatex}
\renewenvironment{abstract}
  {\quotation
  {\bfseries\noindent{\abstractname:}}}
  {\endquotation}

\usepackage{authblk} 

\usepackage{bbm}
\usepackage{amssymb,amsmath,amsthm}
\allowdisplaybreaks 
\usepackage{verbatim}
\usepackage[notref, notcite, final]{showkeys}
\usepackage{mathtools}
\usepackage{enumitem}
\usepackage[colorlinks, bookmarksnumbered, bookmarks]{hyperref} 
\usepackage{xcolor}
\usepackage{txfonts}
\usepackage{lscape}
\usepackage{arydshln}
\usepackage{mathabx}

\theoremstyle{definition} 
\newtheorem{theorem}{Theorem}[section] 
\newtheorem{definition}[theorem]{Definition}
\newtheorem{lemma}[theorem]{Lemma}
\newtheorem{proposition}[theorem]{Proposition}
\newtheorem{corollary}[theorem]{Corollary}
\newtheorem{assumption}[theorem]{Assumption}
\newtheorem{rmk_temp}[theorem]{Remark}
\numberwithin{equation}{section} 

\newenvironment{remark}
  {\pushQED{\qed}\begin{rmk_temp}}
  {\popQED\end{rmk_temp}}

\newbool{draftversion}
\NewDocumentEnvironment{remarkdraft}{+b}
  {\ifbool{draftversion}{\begin{remark}{\color{red}#1}\end{remark}}{}}{}
\boolfalse{draftversion}   

\newcommand{\ysrmk}[2][red]{{\color{#1}#2}}

\newcommand{\N}{\mathbb{N}}
\newcommand{\Z}{\mathbb{Z}}

\newcommand{\R}{\mathbb{R}}
\newcommand{\C}{\mathbb{C}}
\newcommand{\A}{\mathbb{A}}

\newcommand{\simgrad}{\sym\nabla}

\newcommand{\eps}{{\varepsilon}}

\DeclareMathOperator{\sym}{sym}

\newcommand{\vect}[1]{\boldsymbol #1}

\let\oldsqrt\sqrt
\def\sqrt{\mathpalette\DHLhksqrt}
\def\DHLhksqrt#1#2{%
\setbox0=\hbox{$#1\oldsqrt{#2\,}$}\dimen0=\ht0
\advance\dimen0-0.2\ht0
\setbox2=\hbox{\vrule height\ht0 depth -\dimen0}%
{\box0\lower0.4pt\box2}}

\AtBeginDocument{
  \let\div\relax
  \DeclareMathOperator{\div}{div}
}

\title{\LARGE\MakeUppercase{\textbf{An operator-asymptotic approach to periodic homogenization for \\ equations of linearized elasticity}}}

\def\correspondingauthor{\footnote{Corresponding author:
josip.zubrinic@fer.hr}}

\author[1]{Yi-Sheng Lim}
\author[2]{Josip Žubrinić\correspondingauthor{}}

\affil[1]{Department of Mathematical Sciences, University of Bath, Claverton Down, Bath BA2 7AY, \newline United Kingdom (Email: ysl64@bath.ac.uk)}

\affil[2]{Faculty of Electrical Engineering and Computing, University of Zagreb, Unska 3, 10000 Zagreb, Croatia (Email: josip.zubrinic@fer.hr)}

\date{}

\begin{document}

\maketitle

\vspace{-0.8cm}

\begin{abstract}
    We present an operator-asymptotic approach to the problem of homogenization of periodic composite media in the setting of three-dimensional linearized elasticity. This is based on a uniform approximation with respect to the inverse wavelength $|\chi|$ for the solution to the resolvent problem when written as a superposition of elementary plane waves with wave vector (``quasimomentum") $\chi$. We develop an asymptotic procedure in powers of $|\chi|$, combined with a new uniform version of the classical Korn inequality. As a consequence, we obtain $L^2\to L^2$, $L^2\to H^1$, and higher-order $L^2\to L^2$ norm-resolvent estimates in $\mathbb{R}^3$. The $L^2 \to H^1$ and higher-order $L^2 \to L^2$ correctors emerge naturally from the asymptotic procedure, and the former is shown to coincide with the classical formulae.

    \vskip 0.5cm
    

    \vskip 0.5cm

\end{abstract}

\onehalfspacing
\section{Introduction}

This paper lies in the subject of homogenization, which is the study of approximating a highly heterogeneous medium with a homogeneous one. Physically, one is motivated by the desire to understand properties of composite materials, and this translates into a general mathematical problem of finding the asymptotic behaviour of solutions $\vect u_\eps$ to partial differential equations (or of minimizers to a functional), depending on a parameter $\eps>0$ which encodes the heterogenity of the medium.

In the present work, we look at the effective elastic properties of a $\eps\Z^3$-periodic composite material. We further assume that the composite fills the whole space $\R^3$, and that the material is subjected to small strain. This brings us to consider the equations of linearized elasticity on $\R^3$:
\begin{align}\label{eqn:intro_linearelas}
     (\simgrad)^* \left( \mathbb{A}\left(\tfrac{\cdot}{\eps}\right) \simgrad \vect u_\eps \right) + \vect u_\eps = \vect f, \quad \vect f \in L^2(\R^3;\R^3), \quad \eps > 0.
\end{align}
Here, the unknown $\vect u_\eps$ denotes the displacement of the medium from some reference position. $\simgrad \vect u = \frac{1}{2}(\nabla \vect u + \nabla \vect u^\top)$ is the linearized strain tensor, $(\simgrad)^\ast$ denotes the $L^2-$adjoint of the operator $\simgrad$, and $\mathbb{A}$ is a $\Z^3-$periodic fourth order tensor-valued function encoding the properties of the medium. Upon further assumptions on $\mathbb{A}$ (see Section \ref{sect:defn_key_ops}), \eqref{eqn:intro_linearelas} has a unique solution $\vect u_\eps \in L^2(\R^3;\R^3)$. 

In its most basic form, the problem of homogenization asks one to find the limit $\vect u_\text{hom}$ of $\vect u_\eps$ (in an appropriate topology), and to characterize $\vect u_\text{hom}$ as the solution to an equation of the form
\begin{align}\label{eqn:intro_linearelas_homo}
    (\simgrad)^* \left( \mathbb{A}^\text{hom} \simgrad \vect u_\text{hom} \right) + \vect u_\text{hom} = \vect f,
\end{align}
where the tensor $\mathbb{A}^{\text{hom}}$ is constant in space, representing an effective homogeneous medium. There are numerous ways to tackle this, for instance, by a two-scale expansion \cite[Chapter II]{Oleinik2012MathematicalPI}, by Tartar's oscillating test functions method \cite[Chapter~10]{cioranescu_donato}, or by an abstract framework of G-convergence \cite[Chapter 12]{zhikov}.

\subsection{Goal of the paper}
We present an approach to homogenization in the setting of three-dimensional linearized elasticity \eqref{eqn:intro_linearelas}. This method, which we will henceforth refer to as an ``operator-asymptotic" approach, is based on an asymptotic expansion of the solution $\vect u_\eps$ to the resolvent equation \eqref{eqn:intro_linearelas}, in powers of the inverse wavelength $|\chi|$, when $\vect u_\eps$ is written as a superposition of elementary plane waves with wave vector (``quasimomentum") $\chi$. This expansion is justified with operator norm error estimates, which is obtained using new $\chi$-dependent versions of the Korn's inequality on the torus. We show that this procedure gives us a natural and explicit way to define $\A^\text{hom}$ of \eqref{eqn:intro_linearelas_homo} and the correctors, which makes it possible to compare with the classical two-scale expansion.

\subsection{The operator-asymptotic approach}

We coin the name ``operator-asymptotic" approach for the following reasons: First, the approach gives error estimates in the operator norm (Theorem \ref{thm:main_thm}), such as
\begin{align}
    \| \vect u_\eps - \vect u_\text{hom} \|_{L^2} \leq C \eps \| \vect f \|_{L^2}, \quad \text{where $C>0$ is independent of $\eps$ and $\vect f$.}
\end{align}
This is a norm-resolvent estimate: If we write $\mathcal{A}_\eps = (\simgrad)^* \mathbb{A}_\eps \left(\simgrad \right)$ where $\mathbb{A}_\eps = \mathbb{A}\left(\tfrac{\cdot}{\eps}\right)$ and $\mathcal{A}^{\rm hom}  \equiv (\simgrad)^* \mathbb{A}^{\rm hom} \left(\simgrad \right)$, then the above estimate can be written in the form
\begin{align}\label{eqn:intro_operator_est}
    \| \left( \mathcal{A}_\eps + I \right)^{-1} - \left( \mathcal{A}^{\text{hom}} + I \right)^{-1} \|_{L^2 \rightarrow L^2} \leq C \eps, \quad \text{where $C>0$ is independent of $\eps$.}
\end{align}
(We will define $\mathcal{A}_\eps$ and $\mathcal{A}^{\text{hom}}$ through its form in Section \ref{sect:defn_key_ops}.)

Second, the operator-asymptotic method is a ``spectral method", in the language of \cite{zhikov_pastukhova_2016_opsurvey}. Such methods use the Floquet-Bloch-Gelfand transform $\mathcal{G}_\eps$ (Section \ref{sect:scaled_gelfand}) to decompose $\mathcal{A}_\eps$ into a family of operators $\{ \mathcal{A}_\chi \}_{\chi \in [-\pi,\pi)^3}$ on $L^2([0,1)^3;\C^3)$ (Proposition \ref{prop:pass_to_unitcell}). One then performs a spectral analysis on $\mathcal{A}_\chi$, with the goal of identifying the behaviour of $\mathcal{A}_\eps$ near the bottom of the spectrum $\sigma(\mathcal{A}_\eps)$ (Section \ref{sect:spectral_analysis}). This behaviour is then captured by the homogenized operator $\mathcal{A}^\text{hom}$. The construction of $\mathcal{A}^\text{hom}$ together with the error estimates depends on the method at hand, and the operator-asymptotic method achieves this by developing an asymptotic expansion for the resolvent of $|\chi|^{-\alpha} \mathcal{A}_\chi$, with an appropriate choice of $\alpha>0$ (Section \ref{sect:the_asymp_method}).

We mention here the key advantages of the operator-asymptotic method: (i) It is applicable to systems of PDEs. (ii) It is applicable to problems such as thin plates \cite{kirill_igor_plates} and rods \cite{kirill_igor_josip_rods}, where the leading order approximation of $\mathcal{A}_\eps$ is a differential operator of mixed order. (iii) The asymptotic expansion of $|\chi|^{-\alpha} \mathcal{A}_\chi$ provides a constructive way to define $\mathcal{A}^\text{hom}$ and associated corrector objects, which in turn allows us to compare with classical two-scale expansion formulae (Sections \ref{sect:smoothingdrop} and \ref{sect:first_corr_agree}). An extended summary of the method can be found in Section \ref{sect:method_summary}.

\subsection{Existing literature on operator norm estimates in homogenization}

Let us now provide a brief overview on existing approaches to homogenization, restricting our discussion to those that are capable of producing operator norm estimates such as \eqref{eqn:intro_operator_est}. The first operator norm estimates were obtained by Birman and Suslina in \cite{birman_suslina_2004}. This is a spectral method, relying on the Floquet-Bloch decomposition of $\mathcal{A}_\eps$. It is here where the phrase ``threshold effects" was coined, to emphasize that homogenization of periodic differential operators can be viewed as a task of producing an approximate spectral expansion for the operator, near the bottom of the spectrum. This idea is made precise with a key auxiliary operator called the ``spectral germ" \cite[Sect~1.6]{birman_suslina_2004}. Their approach has since been extended in many ways, for instance, to obtain estimates for the parabolic semigroup \cite{suslina_2010_parabolic} and hyperbolic semigroup \cite{dorodnyi_suslina_2018_hyperbolic}, and to the setting of bounded domains \cite{suslina_2013_dirichlet,suslina_2013_neumann} and perforated domains \cite{suslina_2018_perforated}.

The use of Floquet-Bloch analysis in the context of homogenization can be traced back to Zhikov \cite{zhikov1989} and Conca and Vanninathan \cite{conca_vanninathan1997}, but the authors did not pursue the goal of obtaining operator norm estimates. We mention also Ganesh and Vanninathan \cite{gv} in which Floquet-Bloch analysis was used for the homogenization of three dimensional linear elasticity system. It should be noted however, that the result \cite[Theorem 1]{zhikov1989} is a pointwise estimate of the form
\begin{align}\label{eqn:intro_pointwise}
    |K(x,y,t) - K_{\text{hom}}(x,y,t)| \leq \frac{C}{t^{(d+1)/2}},  \quad x, y \in \R^d,
\end{align}
where $K$ and $K_{\text{hom}}$ are the fundamental solutions corresponding to the operators $\mathcal{A}_{\eps=1}$ and $\mathcal{A}^{\text{hom}}$ on $L^2(\R^d;\R)$. That \eqref{eqn:intro_pointwise} implies the norm-resolvent estimate \eqref{eqn:intro_operator_est} is explained in the survey paper by Zhikov and Pastukhova \cite{zhikov_pastukhova_2016_opsurvey}. With this in mind, we will refer to \cite{zhikov1989} as ``Zhikov's spectral approach".

Other approaches that appeared thereafter include: The periodic unfolding method, introduced by Griso in \cite{griso_2004, griso_2006}. The shift method, introduced by Zhikov and Pastukhova in \cite{zhikov_pastukhova_2005_operator} (see also their survey paper \cite{zhikov_pastukhova_2016_opsurvey}). A refinement of the two-scale expansion method by Kenig, Lin, and Shen \cite{kenig_lin_shen_2012}, which directly dealt with the case of bounded domains (see also the recent book by Shen \cite{shen_book}). A recent work by Cooper and Waurick \cite{cooper_waurick_2019}, under which uniform in $\chi$ norm-resolvent estimates for the family $\{ \mathcal{A}_\chi \}_{\chi}$ can be achieved. A recent work by Cherednichenko and D’Onofrio applicable to periodic thin structures \cite{kirill_serena_thinstruct_scalar, kirill_serena_thinstruct_maxwell}. Finally, we mention a recent work by Cooper, Kamotski, and Smyshlyaev \cite{cooper_kamotski_smyshlyaev_2023} that develops a general framework applicable to many classes of periodic homogenization problems.

The operator-asymptotic approach adds to the above list. This approach first appeared in the work of Cherednichenko and Velčić for periodic thin elastic plates \cite{kirill_igor_plates}. The problem under study in \cite{kirill_igor_plates} was one of simultaneous dimension-reduction and homogenization, where the thickness of the plate $h>0$ goes to zero together with the periodicity $\eps>0$ of the medium.

For historical reasons, we mention the work of Cherednichenko and Cooper \cite{kirill_cooper_highcontrast_2016}, in which norm-resolvent estimates were obtained for the case of high-contrast homogenization by means of operator asymptotics, which led to the development of the Cherednichenko-D’Onofrio and Cherednichenko-Velčić operator-asymptotic approach.

The above list is non-exhaustive, and is growing at the point of writing. With so many approaches to choose from, it is natural to wonder how the operator-asymptotic approach fits in. This is the content of the next section.

\subsection{Comparing the operator-asymptotic approach to existing spectral methods}

Of the above list, we will further restrict our discussion to these spectral methods: (a) Birman-Suslina's spectral germ approach, (b) Zhikov's spectral approach, (c) Cooper-Waurick approach, (d) Cherednichenko-D’Onofrio approach, and (e) Cherednichenko-Velčić operator-asymptotic approach.

The first point of comparison is the method's applicability to vector problems, say, a system of $n$ PDEs on $\R^d$ where $n>1$: (b) only applies to scalar problems, whereas (a), (c), (d), and (e) applies to vector problems. In moving from scalar to vector problems, the lowest eigenvalue of $\mathcal{A}_{\chi=0}$ is no longer simple, and this complicates the use of analytic perturbation theory. (a) addresses the issue by writing $\chi$ in polar coordinates, $\chi = t\theta$ with $t\geq 0$ and $|\theta| = 1$, perturbing the $1D$ parameter $t$ while carefully keeping track of the estimates in $\theta$. (c), (d), and (e) chose to work on the level of resolvents. From $(\mathcal{A}_\chi - zI)^{-1}$, one may recover the \textit{sum} of the spectral projections $P_\chi^1,\dots, P_\chi^n$ for the $n$ lowest eigenvalues of $\mathcal{A}_\chi$ (for small $\chi$), by taking an appropriate contour integral of $(\mathcal{A}_\chi - zI)^{-1}$ (see Figure \ref{fig:contour}). The observation that studying the sum $P_\chi^1 + \cdots + P_\chi^n$ (rather than individual $P_\chi^i$'s) is sufficient for homogenization, is one key takeaway of the operator-asymptotic approach, and we hope to make this explicit in the analysis.

A second point of comparison is the method's robustness to obtaining more refined estimates. At present, the works employing (b), (c), and (d) conclude with an $L^2\rightarrow L^2$ estimate \eqref{eqn:intro_operator_est}. On the other hand, (a) has been used to obtain $L^2\rightarrow L^2$ \cite{birman_suslina_2004}, $L^2\rightarrow H^1$ \cite{birman_suslina_2007_l2h1}, and higher-order $L^2\rightarrow L^2$ norm-resolvent estimates \cite{birman_suslina_2006_l2l2higherorder}. We will show that the method (e) could also give us these three estimates (Theorem \ref{thm:main_thm}).

The third point of comparison pertains to the spectral germ of (a), which also plays a key role in (e). To apply (a) to \eqref{eqn:intro_linearelas}, one is first required to rewrite $\mathcal{A}_\eps$ in an appropriate form \cite[Chapter~5.2]{birman_suslina_2004}. The spectral germ is then given in this form (as a $6\times 6$ matrix, where $\dim(\R_{\text{sym}}^{3\times 3}) = 6$), hence work has to be done to recast the spectral germ back to the original setting. In (e), this is simply the $3\times3$ matrix $\mathcal{A}_\chi^{\text{hom}}: \ker(\mathcal{A}_0) \rightarrow \ker(\mathcal{A}_0)$ (Definition \ref{defn:hom_matrix}). Furthermore, $\mathcal{A}_\chi^{\text{hom}}$ is defined in terms of $\chi$-dependent cell problems \eqref{focorr2}, which arises naturally in the asymptotic procedure of Section \ref{sect:the_asymp_method}, thus demonstrating that this definition for $\mathcal{A}_\chi^{\text{hom}}$ is a natural one.

In a similar vein, the asymptotic procedure of Section \ref{sect:the_asymp_method} provides a natural way to define the correctors: The corrector (operators) are defined as the solution operators to problems arising from the asymptotic procedure (Definition \ref{defn:correctors}). Due to the explicit nature of these problems, we are able to show that the first-order corrector agrees with the classical two-scale expansion formulae (Section \ref{sect:first_corr_agree}). Among the spectral approaches (a)-(e), the operator-asymptotic approach is the only one to do so for the setting of \eqref{eqn:intro_linearelas}.

The final point of comparison is the method's applicability to the hybrid problem of simultaneous dimension-reduction and homogenization. As mentioned above, this is possible for (e), as was first shown in the setting periodic thin elastic plates \cite{kirill_igor_plates}, and later for periodic thin elastic rods \cite{kirill_igor_josip_rods}. The mathematical challenge here is the presence of eigenvalues of mixed orders at the bottom of $\sigma(\mathcal{A}_\chi)$, which prevents the construction of the spectral germ. (This issue is absent for \eqref{eqn:intro_linearelas}, see Theorem \ref{thm:evalue}.) This is addressed in \cite{kirill_igor_plates, kirill_igor_josip_rods} by decomposing $\mathcal{A}_\chi$ with respect to carefully chosen invariant subspaces. It remains to be seen if the approaches above (including the non-spectral ones) could be employed to the hybrid problem to obtain operator norm results.

\subsection{Structure of the paper}

This paper is structured as follows: In Section \ref{sect:prelim}, we fix some notation and define the key operators under study. In particular, we introduce the (scaled) Gelfand transform $\mathcal{G}_\eps$ and apply it to $\mathcal{A}_\eps$ to obtain the family $\{ \mathcal{A}_\chi \}_{\chi \in [-\pi,\pi)^3}$. The section ends with the main result of the paper, Theorem \ref{thm:main_thm}. Section \ref{sect:method_summary} contains an extended summary of the operator-asymptotic method.

Sections \ref{sect:spectral_analysis} and \ref{sect:fibre_nr_est} focuses on the family $\{ \mathcal{A}_\chi \}_\chi$. Section \ref{sect:spectral_analysis} contains the necessary spectral information of $\mathcal{A}_\chi$ needed for the fibrewise (i.e.~for each $\chi$) homogenized operator $\mathcal{A}_\chi^{\text{hom}}$, which we define in Section \ref{sect:ahomchi_intro}. Several important properties of $\mathcal{A}_\chi^{\text{hom}}$ are then proved, including its connection to the homogenized operator $\mathcal{A}^\text{hom}$ (defined in Section \ref{sect:main_results}). In Section \ref{sect:the_asymp_method}, we detail an asymptotic procedure (in the quasimomentum $\chi$) for ${(\frac{1}{|\chi|^2}\mathcal{A}_\chi - zI)^{-1}}$. The results of this procedure are summarized in Section \ref{sect:asymp_results_chi}.

Section \ref{sect:norm_resolvent_est} relates the fibrewise results of Section \ref{sect:asymp_results_chi} back to the full space, giving Theorem \ref{thm:main_thm}. The proof of Theorem \ref{thm:main_thm} is broken down into Theorem \ref{thm:l2l2} ($L^2\rightarrow L^2$), Theorem \ref{thm:l2h1} ($L^2\rightarrow H^1$), Theorem \ref{thm:l2l2_higherorder} (higher-order $L^2\rightarrow L^2$), and Theorem \ref{thm:smoothingdrop} (removal of the smoothing operator, defined in Section \ref{sect:spectral_analysis}). Section \ref{sect:first_corr_agree} proves that the first-order corrector as defined through the operator-asymptotic approach agrees with classical formulae.





\section{Preliminaries, problem setup, and main results}\label{sect:prelim}
\subsection{Notation}

Fix a dimension $d \in \mathbb{N} = \{ 1, 2,\cdots \}$. If $a,b \in \mathbb{C}^d$, we will denote their inner product by $\langle a, b \rangle_{\mathbb{C}^d} = a^\top \overline{b} = a \cdot \overline{b} = \sum_{i=1}^d a_i \overline{b}_i$. We will write $\mathfrak{R}(a)$ and $\mathfrak{I}(a)$ for the vector of real and imaginary parts of $a \in \C^d$ respectively (taken component-wise). The tensor product is defined by $a \otimes b := a\overline{b^\top}$. For matrices $A=(A^i_j),B=(B^i_j) \in \mathbb{C}^{d \times d}$, $A^i_j$ denotes the $i$'th row and $j$'th column of A. $\R^{d\times d}_{\text{sym}}$ stands for the space of real $d$ by $d$ symmetric matrices, and similarly for $\C_{\text{sym}}^{3 \times 3}$. $\C^{d\times d}$ will be equipped with the Frobenius inner product, denoted by
\begin{equation}
    A:\overline{B} = \text{tr}(A^\top \overline{B}) = \sum_{1\leq i,j \leq d} A^i_j \overline{B^i_j},
\end{equation}
and the symmetric part of $A \in \C^{d \times d}$ is defined by $\sym A = \frac{1}{2} (A + A^\top).$ The corresponding norms on $\mathbb{C}^d$ and $\mathbb{C}^{d\times d}$ will be denoted by $|a|$ and $|A|$ respectively. 

We will also work with fourth-order tensors, $\mathbb{A} = (\mathbb{A}_{jl}^{ik}) \in \mathbb{C}^{d\times d\times d\times d}$. This acts on matrices $B = (B^k_l) \in \mathbb{C}^{d\times d}$ to give another matrix $\mathbb{A} B = ((\mathbb{A} B)^i_j) \in \mathbb{C}^{d\times d}$ defined by
\begin{equation}
    (\mathbb{A} B)^i_j = \sum_{k,l=1}^d \mathbb{A}_{jl}^{ik} B_l^k, 
    \quad 1\leq i,j \leq d.
\end{equation}

\subsection{The scaled Gelfand transform}\label{sect:scaled_gelfand}

In this section, we introduce more notation, with the goal of introducing the Gelfand transform $\mathcal{G}$ and its scaled version, $\mathcal{G}_\eps$.

Let $Y = [0,1)^3$ be the unit cell, and $Y' = [-\pi,\pi)^3$ be the corresponding dual cell. We will refer to each $\chi \in Y'$ as the quasimomentum. Let $\vect u:\R^3 \rightarrow \C^3$ be a measurable function. By saying that $\vect u$ is a ($\mathbb{Z}^3 -$)periodic function, we mean that
\begin{equation}
    \vect u(x + k e_j) = \vect u(x) \quad \text{a.e. in } \mathbb{R}^3, \quad k \in \mathbb{Z}.
\end{equation}
Here, $\{e_1, e_2, e_3 \}$ denotes the standard basis of $\mathbb{R}^3$.


\textbf{(Vector-valued) function spaces.} Let $\Omega \subset \mathbb{R}^3$ be open. We will need $C^\infty(\Omega;\mathbb{C}^3)$, $L^p(\Omega;\mathbb{C}^3)$, and when the boundary $\partial \Omega$ is Lipschitz, the Sobolev spaces $W^{k,p}(\Omega; \mathbb{C}^3)$, $p \in [1,\infty]$ and $k \in \mathbb{N}\cup \{0\}$. Where it is understood, we will simply omit the domain and co-domain, and write for instance, $L^p$. Of importance are the Hilbert spaces $L^2$ and $H^1 = W^{1,2}$. Another key Hilbert space is ``$H_{\#}^1$", which we discuss now.
First, consider the space of smooth periodic functions, 
\begin{align}
    C_{\#}^{\infty}(Y;\C^3) := \{ \vect u:\mathbb{R}^3 \rightarrow \mathbb{C}^3 ~|~ \vect u \text{ is smooth and periodic} \}.
\end{align}
We will identify $\vect u \in C_{\#}^{\infty}(Y;\mathbb{C}^3)$ with its restriction to $\overline{Y}$. This is used to define
\begin{equation}
    H^1_{\#}(Y;\mathbb{C}^3) := \overline{C_{\#}^{\infty}(Y;\mathbb{C}^3)}^{\|\cdot\|_{H^1}},
\end{equation}
which we will also identify as a subspace of $L^2(Y;\C^3)$.


\textbf{Gelfand transform.} To study periodic functions and operators, we will make use of the Gelfand transform $\mathcal{G}$. This is defined as follows:
\begin{align}
    &\mathcal{G} : L^2(\mathbb{R}^3;\mathbb{C}^3) \to L^2(Y';L^2(Y; \mathbb{C}^3)) =: \int_{Y'}^\oplus L^2(Y; \mathbb{C}^3)d\chi \\
    &(\mathcal{G} \vect u)(y,\chi):= \frac{1}{\left(2\pi\right)^{3/2}} \sum_{n\in \mathbb{Z}^3}e^{-i\chi \cdot (y+n)}\vect u(y+n), \quad y \in  Y, \quad \chi\in Y',
\end{align}

\noindent The Gelfand transform $\mathcal{G}$ is a unitary operator
\begin{equation}
    \left\langle \vect u, \vect v \right\rangle_{L^2(\mathbb{R}^3;\mathbb{C}^3)} = \left\langle \mathcal{G} \vect u,\mathcal{G} \vect v \right\rangle_{L^2(Y;L^2(Y'; \mathbb{C}^3))}, \quad \forall \vect u, \vect v \in L^2(\mathbb{R}^3;\mathbb{C}^3),
\end{equation}

\noindent and the inversion formula is given by
\begin{equation}
   \vect  u(x) = \frac{1}{\left(2\pi\right)^{3/2}} \int_{Y'} e^{i\chi\cdot x} (\mathcal{G} \vect u)(x,\chi)d\chi, \quad x\in\mathbb{R}^3,
\end{equation}
where the integrand $\mathcal{G} \vect u \in L^2(Y \times Y';\mathbb{C}^3)$ is extended in the $x$ variable by $\mathbb{Z}^3-$periodicity to the whole of $\mathbb{R}^3$.

\textbf{Scaled Gelfand transform.} To deal with the setting of highly oscillating material coefficients, we will consider a scaled version of the Gelfand transform, denoted by $\mathcal{G}_\varepsilon$, where $\varepsilon>0$. This is defined by:
\begin{align}
    &\mathcal{G}_\varepsilon : L^2(\mathbb{R}^3;\mathbb{C}^3) \to L^2(Y';L^2(Y; \mathbb{C}^3)) = \int_{Y'}^\oplus L^2(Y; \mathbb{C}^3)d\chi, \\
    &(\mathcal{G}_\varepsilon \vect u)(y,\chi):= \left(\frac{\varepsilon}{2\pi}\right)^{3/2} \sum_{n\in \mathbb{Z}^3}e^{-i\chi \cdot (y+n)}\vect u(\varepsilon(y+n)), \quad y \in  Y, \quad \chi\in Y'. \label{gelfand}
\end{align}

\noindent Note that $\mathcal{G}_\varepsilon = \mathcal{G} \circ \mathcal{S}_\varepsilon$, where $\mathcal{S}_\varepsilon : L^2(\mathbb{R}^3;\mathbb{C}^3) \to L^2(\mathbb{R}^3;\mathbb{C}^3)$ is the unitary scaling operator, given by: 
\begin{equation}
    \mathcal{S}_\varepsilon \vect u(x) := \varepsilon^{3/2} \vect u (\varepsilon x), \quad \vect u\in L^2(\mathbb{R}^3;\mathbb{C}^3). 
\end{equation}
Hence it is clear that that the scaled Gelfand transform $\mathcal{G}_\varepsilon$ is a unitary operator as well, i.e.
\begin{equation}\label{eqn:gelfand_unitary}
    \left\langle \vect u, \vect v \right\rangle_{L^2(\mathbb{R}^3)} = \left\langle \mathcal{G}_\varepsilon \vect u,\mathcal{G}_\varepsilon \vect v \right\rangle_{L^2(Y';L^2(Y; \mathbb{C}^3))}
    = \int_{Y'}  \langle \mathcal{G}_\eps \vect u, \mathcal{G}_\eps \vect v \rangle_{L^2(Y;\C^3)} d\chi
    , \quad \forall \vect u, \vect v \in L^2(\mathbb{R}^3;\mathbb{C}^3).
\end{equation}
The identity $\mathcal{G}_\varepsilon^{-1} = \mathcal{S}_\varepsilon^{-1} \circ \mathcal{G}^{-1}$ then gives the inversion formula for $\mathcal{G}_\varepsilon$:
\begin{equation}\label{eqn:gelfand_inversion}
   \vect  u(x) = \frac{1}{\left( 2\pi \varepsilon\right)^{3/2}} \int_{Y'} e^{i\chi\cdot {\color{blue}\frac{x}{\eps}}} (\mathcal{G}_\varepsilon \vect u)\left( \frac{x}{\eps}, \chi \right) d\chi, \quad x\in\mathbb{R}^3.
\end{equation}


\textbf{Scaling vs derivatives.} A property of $\mathcal{G}_\varepsilon$ is that it commutes with derivatives in the following way: 
\begin{equation}
\label{scalingderivatives}
     \mathcal{G}_\varepsilon (\partial_{x_j} \vect u ) = \frac{1}{\varepsilon} \left(\partial_{y_j}\left(\mathcal{G}_\varepsilon \vect u\right) + i\chi_j \left( \mathcal{G}_\varepsilon \vect u\right) \right), \quad j = 1,2,3.
\end{equation}
Thus, we have the following identity
\begin{equation}\label{eqn:scaling_vs_deriv}
        \mathcal{G}_\varepsilon\left( \nabla \vect u\right)(y,\chi) = \frac{1}{\varepsilon} \left(\nabla_y\left(\mathcal{G}_\varepsilon \vect u\right) + i \left( \mathcal{G}_\varepsilon \vect u\right)\chi^\top \right)(y,\chi).
\end{equation}
For each $\chi \in Y'$, it will be convenient to introduce the following operator
\begin{align}
    &X_{\chi}: L^2( Y; \mathbb{C}^3) \rightarrow L^2(Y;\mathbb{C}^{3\times 3}) \\
    &X_{\chi}\vect u = \sym \left(\vect u \otimes \chi \right) = \sym \left(\vect u \chi^\top \right) \nonumber\\
    &\qquad = \begin{bmatrix}
    \chi_1 u_1 & \frac{1}{2}(\chi_1 u_2 + \chi_2 u_1) & \frac{1}{2}(\chi_1 u_3 + \chi_3 u_1)  \\
    \frac{1}{2}(\chi_1 u_2 + \chi_2 u_1) & \chi_2 u_2 & \frac{1}{2}(\chi_3 u_2 + \chi_2 u_3) \\
    \frac{1}{2}(\chi_3 u_1 + \chi_1 u_3) & \frac{1}{2}(\chi_3 u_2 + \chi_2 u_3) & \chi_3 u_3
    \end{bmatrix}, \label{defxoperator}
\end{align}
so that with this notation, we have
\begin{equation}
\label{gelfandsymetricgradientformula}
    \mathcal{G}_\varepsilon\left( \simgrad \vect u\right)(y,\chi) = \frac{1}{\varepsilon} \left(\simgrad_y\left(\mathcal{G}_\varepsilon \vect u\right) + i \sym \left(\left( \mathcal{G}_\varepsilon \vect u\right)\chi^\top  \right)\right)    = \frac{1}{\varepsilon} \left(\simgrad_y\left(\mathcal{G}_\varepsilon \vect u\right) + i X_\chi\left( \mathcal{G}_\varepsilon \vect u\right) \right).
\end{equation}

\begin{remark}
    There exist constants $c_{\text{symrk1}},C_{\text{symrk1}}>0$ such that
    \begin{equation}\label{eqn:Xchi_est}
        c_{\text{symrk1}} |\chi| \|\vect u \|_{L^2(Y;\mathbb{C}^3)} 
        \leq \|X_{\chi}\vect u \|_{L^2(Y;\mathbb{C}^{3 \times 3})} 
        \leq C_{\text{symrk1}} |\chi|\|\vect u\|_{L^2(Y;\mathbb{C}^3)}, \qquad \text{for all } \vect u\in L^2(Y;\mathbb{C}^d).
    \end{equation}
    Here the lower bound comes from \eqref{rankonesymformula}. In particular, $\| X_\chi \|_{op} < \infty$.
\end{remark}

\subsection{Definition of key operators under study}\label{sect:defn_key_ops}
In this section, we will define the key operators of interest, namely $\mathcal{A}_\eps$ and $\mathcal{A}_\chi$. We begin by stating the assumptions of the coefficient tensor $\mathbb{A}(y)$:

\begin{assumption} 
\label{coffassumption}
We impose the following assumptions on tensor of material coefficients $\A$:
\begin{itemize}
    \item $\A:\R^3 \rightarrow \R^{3\times 3 \times 3 \times 3}$ is $\Z^3-$periodic.

    \item $\A$ is uniformly (in $y$) positive definite on symmetric matrices: There exists $\nu>0$ such that
    \begin{equation}\label{assump:elliptic}
        \nu|\vect \xi|^2 \leq \A(y) \vect \xi : \vect \xi \leq \frac{1}{\nu}|\vect \xi|^2, \quad \forall \, \vect \xi \in \R^{3\times 3}_\text{sym}, \quad \forall y \in Y. 
    \end{equation}
    
    \item The tensor $\A$ satisfies the following material symmetries:
    \begin{equation}\label{assump:symmetric}
        \A_{jl}^{ik} = \A_{il}^{jk} = \A_{lj}^{ki}, \quad  i,j,k,l\in\{1,2,3\}.
    \end{equation}
    \item The coefficients of $\A$ satisfy $\A_{jl}^{ik} \in L^{\infty}(Y)$, where $i,j,k,l\in\{1,2,3\}$.
\end{itemize}
\end{assumption}

Writing $\mathbb{A}_\eps = \mathbb{A}(\frac{\cdot}{\eps})$, we next define the operator $\mathcal{A}_\eps$:

\begin{definition}\label{defn:main_operator_fullspace}
    The operator $\mathcal{A}_\eps \equiv (\simgrad)^* \mathbb{A}_\eps \left( \simgrad \right) $ is the operator on $L^2(\R^3;\C^3)$ defined through its corresponding sesquilinear form $a_\eps$ with form domain $H^1(\R^3;\C^3)$ and action
    \begin{align}\label{eqn:ahom_complexHS}
        a_\eps (\vect u, \vect v) = \int_{\R^3} \mathbb{A} \left( \frac{x}{\eps} \right) \simgrad \vect u(x) : \overline{\simgrad \vect v (x)} ~ dx, \quad \vect u,\vect v \in \mathcal{D}(a_\eps) = H^1(\R^3;\C^3).
    \end{align}
\end{definition}

\begin{proposition}
    $\mathcal{A}_\eps$ is self-adjoint and non-negative.
\end{proposition}

\begin{proof}
    This follows because the form $a_\eps$ is densely defined, non-negative (by (\ref{assump:elliptic})), symmetric (by (\ref{assump:symmetric})), and closed.
\end{proof}

Before introducing the operator $\mathcal{A}_\chi$, we make a comment on the choice of the field for the Hilbert space.

\begin{remark}[Real vs Complex Hilbert spaces]
    In the context of linearized elasticity, the solution $\vect u$ to the elasticity equations measures the displacement of the medium from some reference position $\vect u_{\text{ref}}$, and is hence real-valued. It would therefore make more sense to study the operator $\mathcal{A}_\eps^{\R}$ on $L^2(\R^3;\R^3)$ with corresponding bilinear form
    \begin{align}\label{eqn:ahom_realHS}
        a_\eps^{\R} (\vect u, \vect v) = \int_{\R^3} \mathbb{A} \left( \frac{x}{\eps} \right) \simgrad \vect u(x) : \simgrad \vect v (x) ~ dx, \quad \vect u,\vect v \in H^1(\R^3;\R^3).
    \end{align}
    However, it is easier to work with the complex version $\mathcal{A}_\eps^{\C} := \mathcal{A}_\eps$, since the Gelfand transform is unitary from $L^2(\R^3;\C^3)$ to $L^2(Y\times Y';\C^3)$ (and only a partial isometry when restricted to $L^2(\R^3;\R^3)$). We will perform the analysis on $\mathcal{A}_\eps^{\C}$, but ultimately return to $\mathcal{A}_\eps^{\R}$ for the main results (see~Theorem \ref{thm:main_thm}). The passage from $\C$ to $\R$ follows from $\mathcal{A}_\eps^{\C}|_{L^2(\R^3;\R^3)} = \mathcal{A}_\eps^{\R}$. (To see this identity, compare (\ref{eqn:ahom_complexHS}) and (\ref{eqn:ahom_realHS}).) A similar remark holds for the homogenized operator $\mathcal{A}^{\text{hom}}$, defined in the next section. 
\end{remark}

Recall the operator $X_{\chi}$ defined by \eqref{defxoperator}. We next define the operator $\mathcal{A}_\chi$:

\begin{definition}\label{definitionofAchi}
    The operator  $\mathcal{A}_\chi \equiv (\simgrad + i X_\chi)^* \mathbb{A} (\simgrad + i X_\chi) $ is the operator on $L^2(Y;\C^3)$ defined through its corresponding sesquilinear form $a_\chi$ with form domain $H_{\#}^1(Y;\C^3)$ and action
    \begin{align}\label{form_eqn}
        a_\chi (\vect u, \vect v) = \int_Y \mathbb{A}(y) (\simgrad + i X_\chi) \vect u (y) : \overline{(\simgrad + i X_\chi) \vect v (y)} ~ dy, \quad \vect u,\vect v \in H^1_{\#}(Y;\C^3).
    \end{align}
\end{definition}

\begin{remark}[More notation]
    We will write $\mathcal{D}[\mathcal{A}] = \mathcal{D}(a)$ for the form domain of an operator $\mathcal{A}$ corresponding to the form $a$, and $\mathcal{D}(\mathcal{A})$ for the operator domain. $\sigma(\mathcal{A})$ denotes the spectrum for $\mathcal{A}$, and $\rho(\mathcal{A})$ is the resolvent set. We will also omit the differential ``$dy$" where it is understood.
\end{remark}

\begin{proposition}
    For each $\chi \in Y'$, the form $a_{\chi}$ is densely defined, non-negative, symmetric, and closed. As a consequence, the corresponding operator $\mathcal{A}_\chi$ is self-adjoint and non-negative. Moreover, if $\chi \in Y'\setminus\{0\}$, then the form $a_{\chi}$ is also coercive in $H^1(Y;\C^3)$ and satisfies:
    \begin{equation}
                 C\lVert \vect u\rVert_{H^1(Y;\C^3)}^2 \leq  \frac{1}{|\chi|^2}a_{\chi}(\vect u, \vect u), \quad \chi \in Y' \setminus \{ 0 \},
    \end{equation}
\end{proposition}

\begin{proof}
    The claims on denseness of the domain and closedness of $a_\chi$ are immediate. Non-negativity and symmetry follows from (\ref{assump:elliptic}) and (\ref{assump:symmetric}) of Assumption \ref{coffassumption} respectively.
    
    Next, we use (\ref{assump:elliptic}) of Assumption \ref{coffassumption} to get: 
    \begin{equation}
    \label{normbounds}
        C_1 \left\lVert \left(\simgrad + iX_\chi \right)\vect u\right\rVert_{L^2(Y;\C^{3 \times 3})}^2 \leq  a_{ \chi}(\vect u, \vect u) \leq C_2 \left\lVert \left(\simgrad + iX_\chi \right)\vect u\right\rVert_{L^2(Y;\C^{3 \times 3})}^2, \quad \forall \vect u \in \mathcal{D}(a_{\chi}) = H_{\#}^1,
    \end{equation}
    where the constants $C_1, C_2>0$ do not depend on $\chi$. Combining this with \eqref{estimate1} and \eqref{estimate11} (Proposition \ref{prop:coercive_est}), we have: 
     \begin{equation}
         C\lVert \vect u\rVert_{H^1(Y;\C^3)}^2 \leq  \frac{1}{|\chi|^2}a_{\chi}(\vect u, \vect u), \quad \chi \in Y' \setminus \{ 0 \},
     \end{equation}
     so the form is coercive when $\chi \neq 0$.
\end{proof}

The operators $\frac{1}{\eps^2}\mathcal{A}_\chi$ are actually the fibres of an operator that is unitarily equivalent, under the scaled Gelfand transform $\mathcal{G}_\varepsilon$, see \eqref{gelfand}, to $\mathcal{A}_\eps$. Let us record this observation in the result below:

\begin{proposition}[Passing to the unit cell for $\mathcal{A}_\eps$] \label{prop:pass_to_unitcell} 
    We have the following relation between the forms $a_\eps$ and $a_\chi$:
    \begin{equation}
        a_\varepsilon(\vect u,\vect v) = \int_{Y'}\frac{1}{\varepsilon^2}a_{\chi}(\mathcal{G}_\varepsilon\vect u,\mathcal{G}_\varepsilon\vect v) d\chi,  \quad \vect u,\vect v \in H^1(\R^3;\C^3),
    \end{equation}
    and the following relations between the operators $\mathcal{A}_\eps$ and $\mathcal{A}_\chi$: 
    \begin{align}
        \mathcal{A}_\varepsilon &= \mathcal{G}_\varepsilon^* \left( \int_{Y'}^\oplus \frac{1}{\varepsilon^2}\mathcal{A}_{\chi} d\chi \right) \mathcal{G}_\varepsilon, \label{eqn:vonneumannformula_ops} \\
        \left(\mathcal{A}_\varepsilon - zI \right)^{-1} &= \mathcal{G}_\varepsilon^* \left( \int_{Y'}^\oplus \left(\frac{1}{\varepsilon^2}\mathcal{A}_{\chi} - zI \right)^{-1}  d\chi \right) \mathcal{G}_\varepsilon, \qquad \text{for } z \in \rho(\mathcal{A}_\varepsilon).\label{eqn:vonneumannformula_resolvent}
    \end{align}
\end{proposition}

\begin{proof}
    The assertion on the forms is a consequence of (\ref{eqn:gelfand_unitary}) and (\ref{eqn:scaling_vs_deriv}). For the operator identities, the assertion (\ref{eqn:vonneumannformula_ops}) follows again from (\ref{eqn:scaling_vs_deriv}). Then (\ref{eqn:vonneumannformula_resolvent}) follows from (\ref{eqn:vonneumannformula_ops}) by using \cite[Theorem XIII.85]{reed_simon4}.
\end{proof}

\subsection{Homogenized tensor, homogenized operator, and main results}\label{sect:main_results}

In this section, we will define the homogenized operator $\mathcal{A}^{\text{hom}}$, and state our main result. We begin by defining a bilinear form $a^{\text{hom}}$, from which we extract the homogenized tensor $\A^{\rm hom}$ (see Proposition \ref{prop:homdefinitionreal}).

\begin{definition}
    Let $a^{\text{hom}}$ be the bilinear form on (the real vector space) $\R^{3 \times 3}_\text{sym}$ defined by
    \begin{equation}
        a^{\rm hom}(\vect \xi, \vect \zeta) = \int_{Y} \A(y) \left( \vect \xi + \simgrad \vect u^{\vect \xi} \right ) : \vect \zeta \, dy, \quad \vect \xi, \vect \zeta \in \R^{3 \times 3}_\text{sym},
    \end{equation}
    where the \textit{corrector term} $\vect u^{\vect \xi} \in H_{\#}^1( Y;\R^3)$ is the unique solution of the cell-problem:
    \begin{equation}\label{correctordefinition}
        \begin{cases}
            \int_{ Y} \A \left( \vect \xi + \simgrad \vect u^{\vect \xi} \right ) : \simgrad \vect v \, dy = 0, \quad \forall \vect v \in  H_{\#}^1( Y;\R^3), \\
            \int_Y \vect u^{ \vect \xi } = 0.
        \end{cases}
    \end{equation}
\end{definition}

We record a useful observation that we will use repeatedly throughout the text.

\begin{lemma}[Orthogonality lemma]\label{lem:ibp_ortho}
    For $\xi \in \R^{3\times 3}_\text{sym}$ and $\vect u \in H_{\#}^1(Y;\R^3)$, we have that $\xi$ and $\simgrad \vect u$ are orthogonal in $L^2(Y;\R^3)$.
\end{lemma}

\begin{proof}
    As $\xi$ is orthogonal to the skew-symmetric part of $\nabla \vect u$ (with respect to the Frobenius inner product), we have
    \begin{align}\label{eqn:ibp_ortho}
        \int_Y \xi : \simgrad \vect u \, dy = \int_Y \xi : \nabla \vect u \, dy.
    \end{align}
    Now apply integration by parts to the right hand side of (\ref{eqn:ibp_ortho}), together with the assumption that $\xi$ is constant and $\vect u$ is periodic to conclude that this is zero.
\end{proof}

\begin{remark}
    Though not required for our analysis here, let us point out that by Lemma \ref{lem:ibp_ortho} and the symmetries of $\mathbb{A}$, the expression $(\simgrad)^* \mathbb{A}_\eps (\simgrad)$ for $\mathcal{A}_\eps$ can be simplified to $-\div \mathbb{A}_\eps (\simgrad)$.
\end{remark}

\begin{proposition}[Homogenized tensor $\mathbb{A}^{\text{hom}}$]\label{prop:homdefinitionreal}
    The form $a^{\rm hom}$ is a positive symmetric bilinear form on the space $\R_{\text{sym}}^{3\times 3}$, uniquely represented with a tensor $\A^{\rm hom}\in \R^{3 \times 3 \times 3 \times 3}$ satisfying the symmetries as $\mathbb{A}$. That is,
    \begin{alignat}{2}
        &\left[\A^{\rm hom}\right]_{jl}^{ik} = \left[\A^{\rm hom}\right]_{il}^{jk} = \left[\A^{\rm hom}\right]_{lj}^{ki} \qquad  &&i,j,k,l\in\{1,2,3\}, \label{Ahom_properties}\\
        &a^{\rm{hom}} (\vect \xi, \vect \zeta) = \mathbb{A}^{\rm{hom}} \vect \xi : \vect \zeta, && \forall \, \vect \xi, \vect \zeta \in \R^{3\times 3 }_\text{sym}. \label{eqn:ahom_tensor_form_repr}
    \end{alignat}
    Furthermore, there exist some constant $\nu_{\text{hom}}>0$ depending only on $\nu$ and $\| \mathbb{A}_{jl}^{ik} \|_{L^\infty}$ (from Assumption \ref{coffassumption}), such that
    \begin{equation}\label{eqn:Ahomcoercivity_bdd}
        \nu_{\text{hom}} |\vect \xi|^2 \leq \A^{\rm hom}\vect \xi: \vect \xi \leq \frac{1}{\nu_{\text{hom}}} |\vect \xi|^2, \quad \forall \, \vect \xi \, \in \R^{3 \times 3}_\text{sym}. 
    \end{equation}
\end{proposition}

\begin{proof}
    (Bilinearity) This follows from the identity $\vect u^{{a_1 \vect\xi_1 + a_2 \vect\xi_2}} = a_1 \vect u^{\vect{\xi_1}} + a_2 \vect u^{\vect{\xi_2}}$, where $a_1,a_2 \in \R$ and $\vect{\xi_1}, \vect{\xi_2} \in \R^{3 \times 3}_\text{sym}$. This is a consequence of the uniqueness of solutions to (\ref{correctordefinition}). Here, $\vect u^{\vect \xi}$ denotes the solution to (\ref{correctordefinition}) corresponding to the matrix $\vect \xi \in \R^{3 \times 3}_\text{sym}$.

    (Symmetry of $a^{\text{hom}}$) Denote $\vect u^{\vect \zeta} \in H_{\#}^1(Y;\R^3)$ for the solution to (\ref{correctordefinition}) corresponding to the matrix $\vect \zeta \in \R^{3 \times 3}_\text{sym}$. Using $\vect u^{\vect \zeta}$ as the test function for the problem for $\vect u^{\vect \xi}$, we observe that
    \begin{align}\label{eqn:form_tensor_symm}
        a^{\text{hom}} (\vect \xi, \vect \zeta) 
        = \int_Y \mathbb{A} \left( \vect \xi + \simgrad \vect u^{\vect \xi} \right) : \vect \zeta \, dy
        = \int_Y \mathbb{A} \left( \vect \xi + \simgrad \vect u^{\vect \xi} \right) : \left( \vect \zeta + \simgrad \vect u^{\vect \zeta} \right) dy.
    \end{align}
    This identity, together with the symmetries of $\mathbb{A}$ (Assumption \ref{coffassumption}), implies that the form $a^{\text{hom}}$ is symmetric.

    (Unique correspondence of $a^{\text{hom}}$ with an $\mathbb{A}^{\text{hom}}$ satisfying (\ref{Ahom_properties})) As $a^{\text{hom}}$ is a symmetric bilinear form on the vector space $V := \R_{\text{sym}}^{3\times 3}$, it must be uniquely represented by a symmetric matrix $\mathbb{A}^{\text{hom}}$ (viewed as a linear map on the vector space $V$). That is,
    \begin{align}\label{eqn:symm_one}
        a^{\rm{hom}} (\vect \xi, \vect \zeta) = \mathbb{A}^{\rm{hom}} \vect \xi : \vect \zeta = \vect \xi : \mathbb{A}^{\rm{hom}} \vect \zeta, \quad \forall \, \vect \xi, \vect \zeta \in V.
    \end{align}
    The symmetric matrix $\mathbb{A}^{\text{hom}}$, now viewed as a tensor on $\R^{3 \times 3 \times 3 \times 3}$ (keeping the same notation $\mathbb{A}^{\text{hom}}$), satisfies the symmetries (\ref{Ahom_properties}). Indeed, $\left[\A^{\rm hom}\right]_{jl}^{ik} = \left[\A^{\rm hom}\right]_{lj}^{ki}$ follows from (\ref{eqn:symm_one}), and $\left[\A^{\rm hom}\right]_{jl}^{ik} = \left[\A^{\rm hom}\right]_{il}^{jk}$ follows from $\mathbb{A}^{\rm{hom}} \vect \xi : \vect \zeta = \mathbb{A}^{\rm{hom}} \vect \xi : \vect \zeta^\top$.

    (Positivity) We will establish the lower bound of (\ref{eqn:Ahomcoercivity_bdd}). Let $\vect \xi \in \R^{3\times 3}_\text{sym}$. Then,
    \begin{align*}
            \A^{\rm hom} \vect \xi: \vect \xi = a^{\rm hom}(\vect \xi, \vect \xi) 
            &= \int_Y \mathbb{A} \left( \vect \xi + \simgrad \vect u^{\vect \xi} \right) : \left( \vect \xi + \simgrad \vect u^{\vect \xi} \right) \, dy \qquad 
            &&\text{By (\ref{eqn:form_tensor_symm}).} \\
            & \geq \nu \left\lVert \vect \xi + \simgrad \vect u^{\vect \xi} \right\rVert_{L^2(Y;\R^{3 \times 3})}^2 
            &&\text{By (\ref{assump:elliptic}) of Assumption \ref{coffassumption}.} \\ 
            &= \nu  \left( \left\| \vect \xi \right\|_{L^2(Y;\R^{3 \times 3})}^2 + \left\| \simgrad \vect u^{\vect \xi} \right\|_{L^2(Y;\R^{3 \times 3})}^2 \right) 
            &&\text{By Lemma \ref{lem:ibp_ortho}.} \\
            &\geq  \nu \left\lVert \vect \xi  \right\rVert_{L^2(Y;\R^{3 \times 3})}^2 = \nu \lvert \vect \xi\rvert^2.
    \end{align*}
    (Upper bound of (\ref{eqn:Ahomcoercivity_bdd})) Similarly, we have
    \begin{align}\label{eqn:ahom_upperbound}
        \A^{\rm hom} \vect \xi : \vect \xi \leq \frac{1}{\nu} \| \vect \xi + \simgrad \vect u^{\vect \xi} \|_{L^2(Y;\R^{3\times 3})}^2 
        = \frac{1}{\nu} \left( \| \vect \xi \|_{L^2(Y;\R^{3\times 3})}^2 + \| \simgrad \vect u^{\vect \xi} \|_{L^2(Y;\R^{3\times 3})}^2 \right).
    \end{align}
    We may further bound $\| \simgrad \vect u^{\vect \xi} \|_{L^2}$ as follows
    \begin{alignat}{2}
        \| \simgrad \vect u^{\vect \xi} \|_{L^2(Y;\R^{3\times 3})}^2 &\leq \frac{1}{\nu} \int_Y \mathbb{A} \simgrad \vect u^{\vect \xi} : \simgrad \vect u^{\vect \xi} \, dy 
        &&\text{By (\ref{assump:elliptic}) of Assumption \ref{coffassumption}.} \nonumber\\
        &= -\frac{1}{\nu} \int_Y \mathbb{A} \vect \xi : \simgrad \vect u^{\vect \xi} \, dy
        &&\text{By the cell-problem (\ref{correctordefinition}).} \nonumber\\
        &\leq \frac{1}{\nu} \| \mathbb{A} \vect \xi \|_{L^2(Y;\R^{3\times 3})} \| \simgrad \vect u^{\vect \xi} \|_{L^2(Y;\R^{3\times 3})} \qquad
        &&\text{By H\"older's inequality.} \nonumber\\
        &\leq C \| \vect \xi \|_{L^2} \| \simgrad \vect u^{\vect \xi} \|_{L^2}
        &&\text{By Assumption \ref{coffassumption}.} \label{eqn:elliptic_est}
    \end{alignat}
    where the constant $C>0$ depends on $\nu$ and $\| \mathbb{A}_{jl}^{ik} \|_{L^\infty}$. Once again we note that $\| \vect \xi \|_{L^2} = |\xi|$. Now combine the inequality (\ref{eqn:elliptic_est}) with (\ref{eqn:ahom_upperbound}), and we get the upper bound of (\ref{eqn:Ahomcoercivity_bdd}). This completes the proof.
\end{proof}

The tensor $\mathbb{A}^{\rm hom}$ is used to define the homogenized operator $\mathcal{A}^{\rm hom}$ as follows:

\begin{definition}[Homogenized operator] \label{defn:hom_operator_fullspace}
    Let $\mathcal{A}^{\rm hom}$ be the operator on $L^2(\R^3;\C^3)$ given by the following differential expression 
    \begin{equation}
        \mathcal{A}^{\rm hom}  \equiv (\simgrad)^* \mathbb{A}^{\rm hom} \left(\simgrad \right),
    \end{equation}
    with domain $\mathcal{D}\left(\mathcal{A}^{\rm hom}\right) = H^2(\R^3;\C^3)$. Its corresponding form is given by
    \begin{equation}
        \left\langle \mathcal{A}^{\rm hom} \vect u, \vect v \right\rangle_{L^2(\R^3;\C^3)} = \int_{\R^3} \mathbb{A}^{\rm hom} \simgrad \vect u : \overline{\simgrad \vect v} \, dy, \quad \vect u \in \mathcal{D}(\mathcal{A}^{\rm hom}), \vect v \in \mathcal{D}[\mathcal{A}^{\rm hom}] := H^1(\R^3,\C^3).
    \end{equation}
\end{definition}

We are now in a position to state the main results of the paper.

\begin{theorem}\label{thm:main_thm}
    There exists $C>0$ independent of both the period of material oscillations $\varepsilon>0$ and the parameter of spectral scaling $\gamma > -2$ such that the following norm-resolvent estimates hold: 
    \begin{itemize}
        \item \textbf{$L^2 \rightarrow L^2$ estimate.}
        \begin{equation}
            \left\lVert \left(\frac{1}{\varepsilon^\gamma}\mathcal{A}_{\varepsilon} + I \right)^{-1} - \left(\frac{1}{\varepsilon^\gamma}\mathcal{A}^{\rm hom} + I \right)^{-1} \right\rVert_{L^2(\R^3,\R^3) \to L^2(\R^3,\R^3)} \leq C \varepsilon^\frac{\gamma + 2}{2}.
        \end{equation}

        \item \textbf{$L^2 \rightarrow H^1$ estimate.} For $\gamma > -1$,
        \begin{equation}
            \left\lVert \left(\frac{1}{\varepsilon^\gamma}\mathcal{A}_{\varepsilon} + I \right)^{-1} 
            - \left(\frac{1}{\varepsilon^\gamma}\mathcal{A}^{\rm hom} + I \right)^{-1} 
            -  \mathcal{R}_{\rm corr, 1}^\eps \right\rVert_{L^2(\R^3,\R^3) \to H^1(\R^3,\R^3)} 
            \leq C \max\left\{\varepsilon^{\gamma + 1}, \varepsilon^{\frac{\gamma + 2}{2}} \right\}.
        \end{equation}

        \item \textbf{Higher-order $L^2 \rightarrow L^2$ estimate.}
        \begin{equation}
            \left\lVert \left(\frac{1}{\varepsilon^\gamma}\mathcal{A}_{\varepsilon} + I \right)^{-1} 
            - \left(\frac{1}{\varepsilon^\gamma}\mathcal{A}^{\rm hom} + I \right)^{-1}
            -  \mathcal{R}_{\rm corr, 1}^\eps 
            -  \mathcal{R}_{\rm corr, 2}^\eps \right\rVert_{L^2(\R^3,\R^3) \to L^2(\R^3,\R^3)} 
            \leq C \varepsilon^{\gamma + 2}.
        \end{equation}
    \end{itemize}
    Here, $\mathcal{R}_{\rm corr, 1}^\eps$ and $\mathcal{R}_{\rm corr, 2}^\eps$ are the corrector operators defined by \eqref{trans_back_corr1} and \eqref{trans_back_corr2}, respectively.
\end{theorem}

\begin{remark}[Relation to classical two-scale expansion formulae]\label{rmk:main_thm}
    Theorem \ref{thm:main_thm} states in particular, that the zeroth-order approximation to $\vect u_\eps = (\frac{1}{\eps^{\gamma}} \mathcal{A}_\eps + I )^{-1} \vect f$, where $\vect f \in L^2$, is $(\frac{1}{\eps^{\gamma}} \mathcal{A}^{\rm hom} + I )^{-1} \vect f$, and this is in agreement with the classical two-scale expansion \cite[Chpt 7]{cioranescu_donato}. We go a step further, and show that the first-order approximation $\mathcal{R}_{\rm corr,1}^{\eps} \vect f$ coincides with the first-order term in the classical two-scale expansion. This is the content of Section \ref{sect:first_corr_agree}.
\end{remark}

\section{Summary of the operator-asymptotic method}\label{sect:method_summary}

This section contains an overview of the operator-asymptotic method. Recall that our goal (Theorem \ref{thm:main_thm}) is to show that $\mathcal{A}_\eps \equiv (\simgrad)^* \mathbb{A}_\eps \left( \simgrad \right)$ and $\mathcal{A}^{\rm hom} \equiv (\simgrad)^* \mathbb{A}^{\rm hom} \left(\simgrad \right)$ are close in the norm-resolvent sense. 

\textbf{Step 1.} In Section \ref{sect:prelim}, we have shown that rather than studying directly the operator $\mathcal{A}_\eps$ on $L^2(\R^3;\C^3)$, we may equivalently study the family $\mathcal{A}_\chi \equiv (\simgrad + i X_\chi)^* \mathbb{A} (\simgrad + i X_\chi) $ on $L^2(Y;\C^3)$. This is made possible by the Gelfand transform $\mathcal{G}_\eps$, as we see that (Proposition \ref{prop:pass_to_unitcell}) $\mathcal{A}_\eps$ is unitarily equivalent to the following operator on $L^2(Y\times Y')$:
\begin{align*}
    \int_{Y'}^\oplus \frac{1}{\eps^2} \mathcal{A}_\chi d\chi.
\end{align*}
The upshots are: (i) The expression $\frac{1}{\eps^2} \mathcal{A}_\chi$ separates the dependence on $\eps$, allowing us to focus on $\mathcal{A}_\chi$. (ii) Each $\mathcal{A}_\chi$ has discrete spectrum $\sigma(\mathcal{A}_\chi) = \sigma_p(\mathcal{A}_\chi) = \{ \lambda_n^\chi\}_n$, allowing for an easy spectral approximation.

\textbf{Step 2.} We then collect the necessary spectral information on $\mathcal{A}_\chi$ in Section \ref{sect:spectral_analysis}. Due to Korn's inequalities (Proposition \ref{prop:coercive_est}), we know that
\begin{itemize}
    \item As $|\chi|\downarrow 0$, we have $\lambda_1^\chi, \lambda_2^\chi, \lambda_3^\chi \asymp |\chi|^2$ (note two-sided inequality), and $\lambda_n^\chi \geq C > 0$ for $n\geq 4$. 
    \item The eigenspace corresponding to $0 = \lambda_1^0 = \lambda_2^0 = \lambda_3^0$ is $\C^3$. That is, $\ker(\mathcal{A}_{\chi=0}) = \C^3$.
\end{itemize}

Similarly to existing spectral methods, we will also focus on a small neighbourhood of $\chi=0$, as we will eventually show (in Section \ref{sect:norm_resolvent_est}) that the remaining values of $\chi$ do not contribute to the estimates.

\textbf{Step 3.} The spectral information (Section \ref{sect:spectral_analysis}) is used Section \ref{sect:fibre_nr_est} in the following way:
\begin{enumerate}
    \item First, we define the fibre-wise homogenized operator $\mathcal{A}_\chi^\text{hom} : \ker(\mathcal{A}_0) \rightarrow \ker(\mathcal{A}_0)$ (Definition \ref{defn:hom_matrix}). This is simply a $3\times 3$ matrix (by bullet point two) defined in terms of ($\chi$-dependent) cell problems \eqref{focorr2}. The fibre-wise homogenized operator $\mathcal{A}_\chi^\text{hom}$ is connected to the homogenized operator $\mathcal{A}^\text{hom}$ by the key formula \eqref{eqn:vn_ahom_ops}.

    \item Second, we develop an asymptotic expansion in powers of $|\chi|$ for the resolvent of $\frac{1}{|\chi|^2} \mathcal{A}_\chi$ (Section \ref{sect:the_asymp_method}). We have replaced the factor $\frac{1}{\eps^2}$ by $\frac{1}{|\chi|^2}$ so that we only have to keep track of a single parameter $\chi$. We have chosen $\alpha = 2$ in $\frac{1}{|\chi|^\alpha} \mathcal{A}_\chi$ based on bullet point one. The conclusion of Section \ref{sect:fibre_nr_est} is Theorem \ref{thm:tau_resolvent_est}. Here are some notable features of the asymptotic procedure:

    \begin{enumerate}
        \item The problems used to define $\mathcal{A}_\chi^\text{hom}$ and other important ``corrector" objects $\mathcal{R}_{\text{corr},1,\chi}(z)$ and $\mathcal{R}_{\text{corr},2,\chi}(z)$ arise naturally in the procedure.

        \item The expansion is done in iterated cycles. The $n$-th cycle captures the all terms of order $|\chi|^{n-1}$, and some terms of order $|\chi|^n$ and $|\chi|^{n+1}$. We perform two cycles, sufficient for Theorem \ref{thm:main_thm}.
    \end{enumerate}
\end{enumerate}

\textbf{Step 4.} Section \ref{sect:norm_resolvent_est} then brings the fibre-wise result (Theorem \ref{thm:tau_resolvent_est}) back to a result on the full space (Theorem \ref{thm:main_thm}). We convert the rescaling $\frac{1}{|\chi|^2}$ back to $\frac{1}{\eps^2}$ by the holomorphic functional calculus. While using the functional calculus, an appropriate contour $\Gamma$ has to be chosen (Figure \ref{fig:contour}), and this is possible by the spectral estimates on $\mathcal{A}_\chi$.

As discussed in the introduction, Theorem \ref{thm:main_thm} will be proved in several steps. The key ones are Theorem \ref{thm:l2l2} ($L^2 \rightarrow L^2$), Theorem \ref{thm:l2h1} ($L^2 \rightarrow H^1$), and Theorem \ref{thm:l2l2_higherorder} (higher-order $L^2 \rightarrow L^2$), and their proofs have a similar structure: On $L^2(Y;\C^3)$, one uses Theorem \ref{thm:tau_resolvent_est}, and separately shows that the case of large $\chi$ or large eigenvalues do not play a role in the estimates. To return to $L^2(\R^3;\C^3)$, we then use Proposition \ref{prop:pass_to_unitcell} ($\mathcal{A}_\eps$ vs $\mathcal{A}_\chi$) and Proposition \ref{prop:pass_to_unitcell2} ($\mathcal{A}^\text{hom}$ vs $\mathcal{A}_\chi^\text{hom}$).


\section{Spectral analysis of \texorpdfstring{$\mathcal{A}_\chi$}{Achi}}\label{sect:spectral_analysis}
Recall the Definition \ref{definitionofAchi} of the family of operators $\mathcal{A}_\chi$. By the Rellich-Kondrachov theorem, the space $H^1_{\#}(Y;\C^3)$ is compactly embedded into $L^2(Y;\C^3)$ \cite[Corollary 6.11]{david_borthwick}. It follows that the spectrum of $\mathcal{A}_\chi$ is discrete, with eigenvalues $\lambda_n^\chi$, which we arrange in non-decreasing order as follows
\begin{align}
    0 \leq \lambda_1^\chi \leq \lambda_2^\chi \leq \lambda_3^\chi \leq \lambda_4^\chi \leq \cdots \rightarrow \infty
\end{align}

The goal of this section is to provide further spectral analysis on the operator $\mathcal{A}_\chi$. In Section \ref{sect:evalue_estimates} we will employ the min-max principle to estimate the size of the eigenvalues of $\mathcal{A}_\chi$. In Section \ref{sect:avg_smoothing_ops} we focus on the subspace $\C^3 \subset L^2(Y;\C^3)$ and introduce two key auxillary objects: the averaging operator $S$ and the smoothing operator $\Xi_\eps$.

\subsection{Eigenvalue estimates}\label{sect:evalue_estimates}

Recall the definition of the form $a_\chi$ in (\ref{form_eqn}).

\begin{definition}
    We denote the Rayleigh quotient associated with $a_\chi$ by
    \begin{equation}
        \mathcal{R}_\chi(\vect u) = \frac{a_\chi( \vect u, \vect u)}{\Vert \vect u\Vert _{L^2(Y;\C^3)}^2}, \quad \vect u \in H^1_{\#}(Y;\C^3) \setminus \{ 0 \},
    \end{equation}
    and denote by $\Lambda_n$ the set of subspaces of $H^1_{\#}(Y;\C^3)$ with dimension $n$.
\end{definition}
With these notation at hand, we can write the min-max principle \cite[Theorem 5.15]{david_borthwick} for $\mathcal{A}_\chi$ as follows:
\begin{align}
    \lambda_n^\chi &= \min_{V \in \Lambda_n} \max_{ \vect v \in V \setminus \{ 0 \}} \mathcal{R}_\chi(\vect v) \label{eqn:min-max} \\
    &= \max_{v_1, \cdots, v_{n-1} \in L^2(Y;\C^3)} \min_{u \in H^1_{\#}(Y;\C^3) \cap \text{span}\{v_1,\cdots, v_{n-1} \}^{\perp} \setminus \{ 0 \}} \mathcal{R}_\chi(\vect v) \label{eqn:max-min}
\end{align}
where $n \in \mathbb{N}$.

The first result of this section is a collection of estimates on the Rayleigh quotient $\mathcal{R}_\chi$.

\begin{proposition}\label{prop:Rayleighestim}
    There exist constants $C_{\text{rayleigh}} > c_{\text{rayleigh}} > 0$ such that
    \begin{alignat}{2}
        c_{\text{rayleigh}}{|\chi|^2} &\leq \mathcal{R}_\chi(\vect u)    &&\forall \vect u \in H^1_{\#}(Y;\C^3)\setminus \{ 0 \}, \label{Rayleighestim_1} \\
        0 &\leq\mathcal{R}_\chi(\vect u) \leq C_{\text{rayleigh}}{|\chi|^2} \qquad   &&\forall \vect u \in \C^3 \setminus \{ 0 \}, \label{Rayleighestim_2}\\
        c_{\text{rayleigh}} &\leq \mathcal{R}_\chi(\vect u)    &&\forall \vect u \in (\C^3)^\perp \cap  H_{\#}^1(Y;\C^3) \setminus \{ 0 \}, \label{Rayleighestim_3}
    \end{alignat}
    where $\C^3$ is viewed as a subspace of $L^2(Y;\C^3)$ by identifying $\C^3$ with the space of constant functions. The notation $(\C^3)^\perp$ stands for the orthogonal complement of $\C^3$ with respect to the $L^2(Y;\C^3)$ inner product, that is, the subspace of $L^2(Y;\C^3)$ containing functions with zero mean.
\end{proposition}

\begin{proof}
    First, let $\vect u \in H_{\#}^1(Y;\C^3) \setminus \{ 0 \}$. Then, using Proposition \ref{prop:coercive_est} and Assumption \ref{coffassumption}, 
    \begin{align}
        \frac{\nu}{(C_{\text{fourier}})^2} | \chi |^2 \stackrel{(\ref{estimate1})}{\leq} 
        \nu \frac{\| (\simgrad + i X_\chi) \vect u \|_{L^2(Y;\C^{3\times 3})}^2}{ \| \vect u \|_{L^2(Y;\C^3)}^2 } 
        \stackrel{(\ref{assump:elliptic})}{\leq} \mathcal{R}_\chi(\vect u), \qquad \text{for } \chi \neq 0.
    \end{align}
    Combining this with the observation that (\ref{assump:elliptic}) implies $0\leq \mathcal{R}_\chi(\vect u)$ for all $\chi$, we obtain (\ref{Rayleighestim_1}) and the first inequality of (\ref{Rayleighestim_2}).
    
    Now assume that $\vect u \in \C^3 \setminus \{ 0 \}$. Then $\simgrad \vect u$ vanishes, and the Rayleigh quotient can be estimated above by
    \begin{align}
        \mathcal{R}_\chi (\vect u) \stackrel{(\ref{assump:elliptic})}{\leq} \frac{1}{\nu} \frac{\|X_\chi \vect u \|_{L^2(Y;\C^{3\times 3})}^2}{ \| \vect u \|_{L^2(Y;\C^3)}^2 } 
        \stackrel{(\ref{eqn:Xchi_est})}{\leq} \frac{(C_{\text{symrk1}})^2}{\nu} |\chi|^2,
    \end{align}
    giving us the second inequality of (\ref{Rayleighestim_2}).
    
    Finally, suppose that $\vect u \in (\C^3)^{\perp} \cap H_{\#}^1(Y;\C^3)$ and $\vect u \neq 0$. Then $\fint_Y \vect u = 0$, and so using (\ref{estimate12}) of Proposition \ref{prop:coercive_est} (which is valid for the case $\chi=0$ as well),
    \begin{align}
        \frac{\nu}{(C_{\text{fourier}})^2} \stackrel{(\ref{estimate12})}{\leq} 
        \nu \frac{\| (\simgrad + i X_\chi) \vect u \|_{L^2(Y;\C^{3\times 3})}^2}{ \| \vect u \|_{L^2(Y;\C^3)}^2 } 
        \stackrel{(\ref{assump:elliptic})}{\leq} \mathcal{R}_\chi(\vect u).
    \end{align}
    This concludes the proof.
\end{proof}

Motivated by this result, we now make the following convention that will apply to the rest of the paper.
\begin{quote}
    \textbf{We will henceforth identify $\C^3$ with the space of constant functions, and hence view $\C^3$ as a subspace of $L^2(Y;\C^3)$.}
\end{quote}

Combining Proposition \ref{prop:Rayleighestim} and the min-max principle ((\ref{eqn:min-max}) and (\ref{eqn:max-min})), we arrive at the following result

\begin{theorem}\label{thm:evalue}
    The spectrum $\sigma(\mathcal{A}_\chi)$ contains $3$ eigenvalues of order $\mathcal{O}({|\chi|^2})$ as $|\chi| \downarrow 0$, while the remaining eigenvalues are uniformly separated from zero.
\end{theorem}

\subsection{Averaging and smoothing operators}\label{sect:avg_smoothing_ops}

As seen from Proposition \ref{prop:Rayleighestim}, the subspace $\C^3 \subset L^2(Y;\C^3)$ plays an important role in the spectral approach to homogenization. This prompts us to make the following definition:

\begin{definition}
    The averaging operator on $Y$, $S:L^2(Y;\C^3) \rightarrow \C^3 \hookrightarrow L^2(Y;\C^3)$ is given by
    \begin{align*}
        S\vect u = \int_Y \vect u.
    \end{align*}
    That is, $S = P_{\C^3}$, the orthogonal projection of $L^2(Y;\C^3)$ onto $\C^3$.
\end{definition}

\begin{remark}[Significance of the subspace $\C^3$]\label{rmk:importance_of_c3}
    By Proposition \ref{prop:Rayleighestim},  we observe that $\C^3$ is precisely the eigenspace spanned by the first three eigenvalues $\lambda^0_1 = \lambda^0_2 = \lambda^0_3 = 0$ of the operator $\mathcal{A}_0$. That is,
    \begin{align}
        \C^3 &= \text{Eig} \left( \lambda_1^0 ; \mathcal{A}_0 \right) \oplus \text{Eig} \left( \lambda_2^0 ; \mathcal{A}_0 \right) \oplus \text{Eig} \left( \lambda_3^0 ; \mathcal{A}_0 \right) = \text{Eig} \left( 0 ; \mathcal{A}_0 \right) = \text{ker}(\mathcal{A}_0).
    \end{align}
    While we have very precise spectral information of $\mathcal{A}_\chi$ at $\chi=0$, at the bottom of the spectrum, less is known about the eigenspaces of $\mathcal{A}_\chi$, for a general $\chi$ (since the coefficient tensor $\mathbb{A}(y)$ is varying in $y$). Nonetheless, we have used the subspace $\C^3$, together with the min-max principle, to estimate the size of the eigenvalues of $\mathcal{A}_\chi$ (Proposition \ref{prop:Rayleighestim}).
\end{remark}

\begin{remark}[Another significance of $\C^3$]\label{rmk:rigid_body_motions}
    The subspace $\C^3$ also appears when we study the operator $\simgrad$ with domain $\mathcal{D}(\simgrad) = H^1(Y;\C^3)$ (Neumann boundary conditions). By Korn's inequality (Proposition \ref{appendixkorn3}) the kernel of $\simgrad$ is the set of vector fields of the form $\vect w = Ax + \vect c$, where $A \in \C^{3 \times 3}$ $A^\top = -A$, $\vect c\in \C^3$. Such vector fields are called \textit{rigid displacements}. When restricted to $Y$-periodic functions, observe that we have
    \begin{equation*}
        \ker(\simgrad)\cap H_\#^1(Y;\C^3) = \C^3.
    \end{equation*}
    This is used in the asymptotic procedure in Section \ref{sect:the_asymp_method}.
\end{remark}

The averaging operator $S$ will be used for our discussion on the operators $\mathcal{A}_\chi$ on $L^2(Y;\C^3)$. We will also need its counterpart that brings the discussion back to the operator $\mathcal{A}_\eps$ on $L^2(\R^3;\R^3)$:

\begin{definition}\label{defn:smoothing_op}
    For $\eps>0$, the smoothing operator $\Xi_\eps :L^2(\R^3;\C^3) \to L^2(\R^3;\C^3)$ is defined as follows:
   \begin{align}\label{eqn:smoothing_op}
        \Xi_\varepsilon \vect u := \mathcal{G}_\varepsilon^{*} \left( \int_{Y'}^{\oplus} S d\chi \right) \mathcal{G}_\varepsilon \vect u 
        = \mathcal{G}_\varepsilon^{*} \left( \int_{Y}  (\mathcal{G}_\varepsilon \vect u)(y,\cdot) dy \right)
    \end{align}
    Here, $\int_{Y}  (\mathcal{G}_\varepsilon \vect u)(y,\cdot) dy$ is viewed as a function of $y\in Y$ and $\chi \in Y'$, which is constant in $y$.
\end{definition}

The formula (\ref{eqn:smoothing_op}) says that $\Xi_\eps$ is simply a projection operator on the Gelfand/frequency space. We may also interpret $\Xi_\eps$ in the following way: It smooths out any given $\vect u \in L^2$ (hence the name ``smoothing operator"), by truncating $\vect u$ at frequencies $\theta \in \R^3$ whose components $\theta_i = \pm \eps^{-1} \pi$.

Indeed, let us write $\mathcal{F}(\cdot)$ and $\mathcal{F}^{-1}(\cdot)$ for the Fourier transform and inverse Fourier transform respectively, and $\theta = \eps^{-1}\chi$. Then we compute:
\begin{align}
    \left( \Xi_\eps \vect f \right) (x)
    &\stackrel{\text{(\ref{eqn:smoothing_op})}}{=} \mathcal{G}_\eps^\ast \left( \int_{Y} (\mathcal{G}_\varepsilon \vect f)(y, \eps\theta) dy \right) (x) \nonumber \\ 
    &= \mathcal{G}_\eps^\ast \left( \left( \frac{1}{2\pi\eps} \right)^{3/2} \mathcal{F}(\vect f) \left( \frac{\theta}{2\pi} \right) \right) (x) \label{eqn:smoothing_fourier_gelfand}\\
    &\stackrel{\text{(\ref{eqn:gelfand_inversion})}}{=} \left( \frac{1}{2\pi\eps} \right)^{3/2} \int_{Y'} \left( \frac{1}{2\pi\eps} \right)^{3/2} \mathcal{F}(\vect f) \left( \frac{\eps^{-1}\chi}{2\pi} \right) e^{i\chi \cdot \frac{x}{\eps}} d\chi \nonumber \\
    &= \left( \frac{1}{2\pi\eps} \right)^3 \int_{Y'} \mathcal{F}(\vect f) \left( \frac{\eps^{-1}\chi}{2\pi} \right) e^{i\frac{\chi}{\eps} \cdot x} d\chi \\
    &\stackrel{\widetilde{\theta}=\frac{\chi}{2\pi\eps}}{=} \cancel{\frac{(2\pi\eps)^3}{(2\pi\eps)^3}} \int_{(2\pi\eps)^{-1}Y'} \mathcal{F}(\vect f) \left( \widetilde{{\theta}} \right) e^{i 2\pi\widetilde{{\theta}} \cdot x} d \widetilde{\theta}
    = \left( \mathcal{F}^{-1}(\mathbbm{1}_{(2\pi\eps)^{-1} Y'}) * \vect f \right) (x). \label{eqn:smoothing_cutoff}
\end{align}
The identity (\ref{eqn:smoothing_fourier_gelfand}) relating the scaled Gelfand transform to the Fourier transform follows from a straightforward computation, which we include in Lemma \ref{lem:fourier_vs_gelfand} (in the appendix) for completeness.

\begin{remark}
    From (\ref{eqn:smoothing_cutoff}), it is also clear that the image of $L^2(\R^3;\R^3)$ (the real-valued functions) under the mapping $\Xi_\eps$ is again a subset of $L^2(\R^3;\R^3)$. 
\end{remark}

\section{Fibrewise norm-resolvent estimates}\label{sect:fibre_nr_est}

By applying the Gelfand transform, we have converted the problem on $L^2(\R^3;\C^3)$ into a family of problems on $L^2(Y;\C^3)$. This section performs the ``estimation step" on $L^2(Y;\C^3)$, and is therefore the crux of our approach to homogenization.

In Section \ref{sect:ahomchi_intro} we introduce the fibrewise (i.e.~for each $\chi$) homogenized operator $\mathcal{A}_\chi^{\text{hom}}$, and explain how it is related to the homogenized operator $\mathcal{A}^{\text{hom}}$ (defined in Section \ref{sect:main_results}). In Section \ref{sect:the_asymp_method} we detail an asymptotic procedure (in the quasimomentum $\chi$) that will give us norm-resolvent estimates for ${(\frac{1}{|\chi|^2}\mathcal{A}_\chi - zI)^{-1}}$. Finally, Section \ref{sect:asymp_results_chi} summarizes the results of the asymptotic procedure.

\subsection{The fibrewise homogenized operator \texorpdfstring{$\mathcal{A}_\chi^{\text{hom}}$}{}}\label{sect:ahomchi_intro}

In this section we estimate the distance between the resolvents of $\frac{1}{|\chi|^2} \mathcal{A}_\chi$ and the \textit{rescaled fibrewise homogenized operator} $\frac{1}{|\chi|^2}\mathcal{A}_\chi^{\rm hom}$. We first begin with the definition of $\mathcal{A}_\chi^{\rm hom}$. Recall the definition \eqref{defxoperator} of the operator $X_\chi$.

\begin{definition}\label{defn:hom_matrix}
    For each $\chi \in Y'$, set $\mathcal{A}^{\rm hom}_\chi \in \C^{3 \times 3}$ to be the constant matrix satisfying
    \begin{equation}\label{homdefinition}
        \left\langle \mathcal{A}^{\rm hom}_\chi \vect c,\vect d \right\rangle_{\C^3} = \int_{Y} \A\left(  \simgrad   \vect u_{\vect c}  +  iX_\chi   \vect c \right) : \overline{  iX_\chi  \vect d }, \qquad \forall \vect c,\vect d \in \C^3,
    \end{equation}
    where the \textit{corrector term} $\vect u_{\vect c} \in H_\#^1(Y;\C^3)$ is the unique solution of the ($\chi$ dependent) cell-problem
    \begin{equation}\label{focorr2}
        \begin{cases}
		\int_{Y} \A\left(  \simgrad \vect u_{\vect c}  +  iX_\chi  \vect c \right) : \overline{ \simgrad   \vect v} = 0, \quad \forall\vect v \in H^1_{\#}(Y;\C^3),\\
        \int_Y \vect u_{\vect c} = 0.
        \end{cases}
    \end{equation}
    (Note that $\vect u_{\vect c}$ depends linearly in $\vect c$, by the same arguments as in the proof of bilinearity in Proposition \ref{prop:homdefinitionreal}.) 
\end{definition}

\begin{remark}
    As we will see in Section \ref{sect:the_asymp_method}, the equations (\ref{homdefinition}) and (\ref{focorr2}) occur naturally in the asymptotic procedure, in (\ref{homformula2}) and (\ref{corr2new}) respectively.
\end{remark}

\begin{lemma}
    $\mathcal{A}^{\rm hom}_\chi$ is a Hermitian matrix.
\end{lemma}
\begin{proof}
    The proof is similar to that of Proposition \ref{prop:homdefinitionreal}. Writing $\vect u_{\vect d} \in H_{\#}^1(Y;\C^3)$ for the solution to (\ref{focorr2}) corresponding to the vector $\vect d \in \C^3$, and using $\vect u_{\vect d}$ as the test function for the problem for $\vect u_{\vect c}$, we get
    \begin{equation}\label{symmetricitycomplex}
         \left\langle \mathcal{A}^{\rm hom}_\chi \vect c,\vect d \right\rangle_{\C^3} = \int_{Y} \A\left(  \simgrad  \vect u_{\vect c}  +  iX_\chi   \vect c \right) : \overline{ \simgrad    \vect u_{\vect d} +  iX_\chi   \vect d }, \quad \vect c, \vect d \in \C^3.
    \end{equation}
    We therefore obtain a symmetric sesquilinear form, and so the corresponding matrix $\mathcal{A}_\chi^{\text{hom}}$ is Hermitian.
\end{proof}

It follows that $\mathcal{A}_\chi^{\text{hom}}$ is diagonalizable. Let us introduce a notation for its eigenvalues, similar to that of $\mathcal{A}_\chi$.

\begin{definition}
    Write $\lambda_1^{\text{hom},\chi}$, $\lambda_2^{\text{hom},\chi}$, and $\lambda_3^{\text{hom},\chi}$ for the eigenvalues of $\mathcal{A}_\chi^{\text{hom}}$, arranged in non-decreasing order.
\end{definition}

\begin{proposition}
\label{prop:hom_matrix}
The matrix $\mathcal{A}_\chi^{\rm hom}$ can be written in the form
    \begin{equation}
    \label{Achihom_ahom}
        \mathcal{A}_\chi^{\rm hom} = \left(iX_\chi \right)^*\A^{\rm hom} (i X_\chi),
    \end{equation}
    where $\A^{\rm hom} \in \R^{3 \times 3 \times 3 \times 3}$ is a constant tensor defined through (\ref{eqn:ahom_tensor_form_repr}). 
    In particular, $\mathcal{A}_\chi^{\rm hom}$ is quadratic in $\chi$ in the following sense: there exist a constant $\nu_1 >0$ that depends only on $\nu_{\text{hom}}$ (from (\ref{eqn:Ahomcoercivity_bdd})), $c_{\text{symrk1}}$, and $C_{\text{symrk1}}$ (from (\ref{eqn:Xchi_est})), such that 
    \begin{equation}\label{eqn:achihom_coer_bdd}
        \nu_1 |\chi|^2|\vect c|^2 \leq \langle \mathcal{A}_\chi^{\rm hom} \vect c, \vect c \rangle_{\C^3} \leq \frac{1}{\nu_1}|\chi|^2 |\vect c|^2, \quad \forall \vect c \in \C^3.
    \end{equation} 
\end{proposition}

\begin{proof}
    Our first task is to relate the corrector equations (\ref{correctordefinition}) and (\ref{focorr2}), which are used to define $\mathbb{A}^{\text{hom}}$ and $\mathcal{A}_\chi^{\text{hom}}$ respectively. Care has to be taken when passing between the real and complex vector spaces.
    
    To this end, we claim that for each $\vect c \in \R^3$, the unique solution $\vect u_{\vect c} \in H^1_{\#}(Y;\C^3)$ to the problem (\ref{focorr2}) corresponding to the vector $\vect c$ can be expressed as $\vect u_{\vect c} = i \, \widetilde{\vect u}_{\vect c}$, where the function $\widetilde{\vect u}_{\vect c} \in H^1_{\#}(Y;\R^3)$ is the unique solution to the problem
    \begin{equation}\label{eqn:corr_real}
        \begin{cases}
		\int_{Y} \mathbb{A} \left(  \simgrad \widetilde{\vect u}_{\vect c}  +  X_\chi  \vect c \right) :  \simgrad   \vect v = 0, \quad \forall \vect v \in H^1_{\#}(Y;\R^3),\\
        \int_Y \widetilde{\vect u}_{\vect c} = 0.
        \end{cases}
    \end{equation}
    Indeed, by multiplying (\ref{eqn:corr_real}) throughout by $i$, we have
    \begin{equation}\label{eqn:corr_real_i}
        \int_Y \mathbb{A} \left(  \simgrad i\, \widetilde{\vect u}_{\vect c}  +  X_\chi  i \vect c \right) :  \simgrad   \vect v = 0, \quad \forall \vect v \in H^1_{\#}(Y;\R^3).
    \end{equation}
    Now fix a test function $\vect v$ in $H_{\#}^1(Y;\C^3)$, and write $\vect v = \mathfrak{R}(\vect v) + i \mathfrak{I}(\vect v)$, where $\mathfrak{R}(\vect v), \mathfrak{I}(\vect v) \in H_{\#}^1(Y;\R^3)$. Then,
    \begin{align*}
        &\int_Y \mathbb{A} \left( \simgrad i \, \widetilde{\vect u}_{\vect c} + i X_\chi \vect c \right) : \overline{\simgrad \vect v} \\
        &\quad = \int_Y \mathbb{A} \left( \simgrad i \, \widetilde{\vect u}_{\vect c} + X_\chi i \vect c \right) : \simgrad \left( \mathfrak{R}(\vect v) - i \mathfrak{I}(\vect v) \right) \\
        &\quad = \underbrace{\int_Y \mathbb{A} \left( \simgrad i \, \widetilde{\vect u}_{\vect c} + X_\chi i \vect c \right) : \simgrad \mathfrak{R}(\vect v)}_{= 0 \text{ by (\ref{eqn:corr_real_i})}}
        - \underbrace{\int_Y \mathbb{A} \left( \simgrad i \, \widetilde{\vect u}_{\vect c} + X_\chi i \vect c \right) : i \simgrad \mathfrak{I}(\vect v)}_{= 0 \text{ by (\ref{eqn:corr_real})}} 
        = 0.
    \end{align*}
    That is, $i \, \widetilde{\vect u}_{\vect c}$ and $\vect u_{\vect c}$ both solve the cell-problem (\ref{focorr2}). By the uniqueness of solutions, they must be equal.

    Now for $\vect c \in \R^3$, we have $X_\chi \vect c \in \R^{3\times 3}_\text{sym}$, and so by comparing the problem (\ref{eqn:corr_real}) with the problem (\ref{correctordefinition}), we observe that $\widetilde{\vect u}_{\vect c}$ is simply $\vect u^{X_\chi \vect c}$ in the notation of (\ref{correctordefinition}). This completes the first task.

    As a result, we are able to use (\ref{eqn:form_tensor_symm}) in the computation below, where $\vect c, \vect d \in \R^3$:
    \begin{alignat}{2}
             \left\langle \mathcal{A}^{\rm hom}_\chi \vect c,\vect d \right\rangle_{\C^3} 
             & = \int_{Y} \A\left(  \simgrad  \vect u_{\vect c}  +  iX_\chi   \vect c \right) : \overline{ \left( \simgrad   \vect u_{\vect d} +   iX_\chi  \vect d \right)} 
             \quad &&\text{By (\ref{symmetricitycomplex}).} \nonumber\\
             & = \int_{Y} \A\left(  \simgrad  \widetilde{\vect u}_{\vect c}  +  X_\chi   \vect c \right) :  \left( \simgrad   \widetilde{\vect u}_{\vect d} +   X_\chi  \vect d \right)  
             &&\text{By the identity $\vect u_{\vect c} = i\, \widetilde{\vect u}_{\vect c}$.} \nonumber\\
             & = \left\langle \A^{\rm hom } X_\chi \vect c, X_\chi   \vect d \right\rangle_{\R^{3\times 3}} &&\text{By (\ref{eqn:form_tensor_symm}) and the identity $\widetilde{\vect u}_{\vect c} = \vect u^{X_\chi \vect c}$.} \nonumber\\
             & = \left\langle \A^{\rm hom } (iX_\chi) \vect c, (iX_\chi) \vect d \right\rangle_{\C^{3 \times 3}} \nonumber\\
             &= \left\langle \left(iX_\chi \right)^*\A^{\rm hom } (iX_\chi) \vect c, \vect d \right\rangle_{\C^3}. \label{eqn:ahomchi_symmetric}
    \end{alignat}
    By $\C-$linearity, this implies the validity (\ref{eqn:ahomchi_symmetric}) for all $\vect c$, $\vect d \in \C^3$. This verifies (\ref{Achihom_ahom}). The bounds (\ref{eqn:achihom_coer_bdd}) are now an easy consequence of (\ref{eqn:ahomchi_symmetric}), (\ref{eqn:Ahomcoercivity_bdd}), and (\ref{eqn:Xchi_est}).
\end{proof}

Recall the homogenized operator $\mathcal{A}^{\rm hom}$ defined in Definition \ref{defn:hom_operator_fullspace}. The following proposition is the counterpart to Proposition \ref{prop:pass_to_unitcell}:

\begin{proposition}[Passing to the unit cell for $\mathcal{A}^{\text{hom}}$] \label{prop:pass_to_unitcell2} 
    The following identity holds 
    \begin{align}\label{eqn:vn_ahom_ops}
        \mathcal{A}^{\rm hom} \Xi_\varepsilon 
        = \mathcal{G}_\eps^* \left(\int_{Y'}^{\oplus} \frac{1}{\eps^2} S^* \mathcal{A}_\chi^{\rm hom} S \, d\chi \right) \mathcal{G}_\eps.
    \end{align}
    As a consequence, we also have the following identity on their resolvents 
    \begin{align}\label{eqn:vn_ahom_resolvents}
        \left( \mathcal{A}^{\rm{hom}} - zI \right)^{-1} \Xi_\eps 
        &= \mathcal{G}_\eps^* \left( \int_{Y'}^{\oplus} \left( \frac{1}{\eps^{2}} \mathcal{A}_{\chi}^{\rm{hom}} - zI_{\C^3} \right)^{-1} S d\chi \right) \mathcal{G}_\eps, \qquad \text{for } z \in \rho(\mathcal{A}^{\rm{hom}}).
    \end{align}
\end{proposition}
\begin{proof}
    Since $\mathcal{A}^{\text{hom}} = (\simgrad)^* \mathbb{A}^{\text{hom}} (\simgrad)$, the same arguments of Proposition \ref{prop:pass_to_unitcell} for $\mathcal{A}_\eps$ may be applied to $\mathcal{A}^{\text{hom}}$, and we get
    \begin{align}\label{eqn:ahom_step1}
        \mathcal{A}^{\text{hom}} = \mathcal{G}_\eps^* \left( \int_{Y'}^\oplus \frac{1}{\eps^2} \mathcal{A}_\chi^{\text{hom-full}} d\chi \right) \mathcal{G}_\eps,
    \end{align}
    where
    \begin{align}
        \mathcal{A}_\chi^{\text{hom-full}} = (\simgrad + iX_\chi)^* \mathbb{A}^{\text{hom}} (\simgrad + iX_\chi), \quad \text{with} \quad \mathcal{D}\left[ \mathcal{A}_\chi^{\text{hom-full}} \right] = H_{\#}^1(Y;\C^3).
    \end{align}
    Now apply the smoothing operator $\Xi_\eps$ on both sides of (\ref{eqn:ahom_step1}), on the right:
    \begin{align}
        \mathcal{A}^{\text{hom}} \Xi_\eps 
        &= \mathcal{G}_\eps^* \left( \int_{Y'}^\oplus \frac{1}{\eps^2} \mathcal{A}_\chi^{\text{hom-full}} d\chi \right) \cancel{\mathcal{G}_\eps \mathcal{G}_\eps^*} \left( \int_{Y'}^{\oplus} S d\chi \right) \mathcal{G}_\eps 
        = \mathcal{G}_\eps^* \left( \int_{Y'}^\oplus \frac{1}{\eps^2} \mathcal{A}_\chi^{\text{hom-full}} S d\chi \right) \mathcal{G}_\eps. \label{eqn:ahom_step2}
    \end{align}
    Recall that $S = P_{\C^3}$, a projection onto the space of constant functions. Since the coefficient tensor $\mathbb{A}^{\text{hom}}$ is constant in space, it implies that
    \begin{itemize}
        \item $\C^3$ is an invariant subspace for $\mathcal{A}_\chi^{\text{hom-full}}$.
        \item $\mathcal{A}_\chi^{\text{hom-full}}|_{\C^3} = (iX_\chi)^* \mathbb{A}^{\text{hom}} (iX_\chi)$. The latter is just $\mathcal{A}_\chi^{\text{hom}}$, by Proposition \ref{prop:hom_matrix}.
    \end{itemize}
    We may therefore continue from (\ref{eqn:ahom_step2}) to obtain
    \begin{align}
        \mathcal{A}^{\text{hom}} \Xi_\eps 
        &= \mathcal{G}_\eps^* \left( \int_{Y'}^{\oplus} \frac{1}{\eps^2} S \mathcal{A}_\chi^{\text{hom}} S d\chi \right) \mathcal{G}_\eps.
    \end{align}
    By noting that $S$ is an orthogonal projection (so that $S=S^*$), we obtain (\ref{eqn:vn_ahom_ops}).

    Applying \cite[Theorem XIII.85]{reed_simon4} to (\ref{eqn:ahom_step1}), we obtain
    \begin{align}\label{eqn:ahom_resolvent_step1}
        \left( \mathcal{A}^{\text{hom}} - zI \right)^{-1}
        &= \mathcal{G}_\eps^* \left( \int_{Y'}^{\oplus} \left( \frac{1}{\eps^{2}} \mathcal{A}_{\chi}^{\text{hom-full}} - zI \right)^{-1} d\chi \right) \mathcal{G}_\eps, \qquad \text{for } z \in \rho(\mathcal{A}^{\text{hom}}).
    \end{align}
    The identity (\ref{eqn:vn_ahom_resolvents}) now follows by applying $\Xi_\eps$ on both sides of (\ref{eqn:ahom_resolvent_step1}), on the right, and then using (\ref{eqn:vn_ahom_ops}).
\end{proof}

\subsection{Asymptotics for the operator \texorpdfstring{$\left(\frac{1}{|\chi|^2} \mathcal{A}_\chi -zI\right)^{-1}$}{(|chi|-2 Achi - z I)-1}} \label{sect:the_asymp_method}

\textbf{Assume that $\chi\neq 0$ and $z \in \rho(\frac{1}{|\chi|^2} \mathcal{A}_\chi ) \cap \rho (\frac{1}{|\chi|^2} \mathcal{A}_\chi^{\text{hom}})$.} To study the resolvent $( \frac{1}{|\chi|^2} \mathcal{A}_\chi -z  I )^{-1}$, we consider the resolvent equation
\begin{equation}
    \left( \frac{1}{|\chi|^2} \mathcal{A}_\chi -z I \right) \vect u = \vect f \in L^2(Y;\C^3),
\end{equation}
which in terms of the weak formulation, is given by
\begin{equation}\label{chiproblem}
    \frac{1}{|\chi|^2} \int_{Y} \A \left( \simgrad + iX_\chi \right) \vect u  : \overline{\left( \simgrad + iX_\chi \right)  \vect v} -z \int_Y \vect u \cdot \overline{\vect v} = \int_Y \vect f \cdot \overline{\vect v} ,\quad \forall \, \vect v\in H_\#^1(Y;\C^3). 
\end{equation}
The unique solution $\vect u \in \mathcal{D}(\frac{1}{|\chi|^2} \mathcal{A}_\chi ) \subset H_\#^1(Y;\C^3)$ to the resolvent equation (\ref{chiproblem}) can be equivalently expressed as $\vect u = (\frac{1}{|\chi|^2} \mathcal{A}_\chi -z I)^{-1} \vect f$. We will now proceed with an asymptotic procedure \textit{in the quasimomentum $\chi$}, by expanding the solution $\vect u$ in orders of $|\chi|$ (as $|\chi| \downarrow 0$).

\begin{remark}
    The reader may wish to look ahead at Section \ref{sect:norm_resolvent_est} to see how the asymptotics of $( \frac{1}{|\chi|^2} \mathcal{A}_\chi - zI)^{-1}$ in $\chi$ is converted back into asymptotics in $\eps$, and in turn how the result on $L^2(Y;\C^3)$ is converted back into a result on $L^2(\R^3;\C^3)$. Let us also point out that the scaling $\frac{1}{|\chi|^2}$ used here comes from the fact that the first three eigenvalues of $\mathcal{A}_\chi$ are of order $O(|\chi|^2)$ (Proposition \ref{prop:Rayleighestim}).  
\end{remark}

\subsubsection{Cycle 1: First expansion of the solution}
The goal of this section is to expand the solution $\vect u$ of \eqref{chiproblem} in the form
\begin{equation}
\label{expansion}
   \vect u = \vect u_0 + \vect u_1 + \vect u_2 + \vect u_{\rm err}, \quad \vect u_i, \, \vect u_{\rm err} \in H_\#^1(Y;\C^3),  \quad i = 0,1,2,
\end{equation}
where the terms satisfy the following bounds
\begin{align}\label{eqn:expansion_bounds}
    \vect u_0 = \mathcal{O}(1), \quad \vect u_1 = \mathcal{O}(|\chi|), \quad \vect u_2 = \mathcal{O}(|\chi|^2), \quad \text{as $|\chi| \downarrow 0.$}
\end{align}
with respect to $H^{1}(Y;\C^3)$ norm, with explicit dependence on $\| \vect f \|_{L^2(Y;\C^3)}$, and each $\vect u_i$ solves some boundary value problem which we will derive below. 
At this point, the expansion \eqref{expansion} is purely formal, and will gain rigor as soon as we provide precise definitions of the involved terms.

To begin, we plug \eqref{expansion} into \eqref{chiproblem} to obtain the following equation, valid for all $\vect v\in H^1_\#(Y;\C^3)$:
\begin{alignat}{5}\label{firstexpansion}
        \int_Y \A \simgrad  \vect u_0  :  \overline{ \simgrad  \vect v} 
        &+ \int_Y \A \simgrad  \vect u_0  :  \overline{  iX_\chi  \vect v} 
        &&+ \int_Y \A iX_\chi  \vect u_0  :  \overline{ \simgrad  \vect v} 
        &&+ \int_Y \A   iX_\chi  \vect u_0  :  \overline{ iX_\chi  \vect v} \nonumber\\
        + \int_Y \A \simgrad  \vect u_1  :  \overline{ \simgrad  \vect v} 
        &+ \int_Y \A \simgrad  \vect u_1  :  \overline{  iX_\chi  \vect v} 
        &&+ \int_Y \A iX_\chi  \vect u_1  :  \overline{ \simgrad  \vect v} 
        &&+ \int_Y \A iX_\chi  \vect u_1  :  \overline{ iX_\chi  \vect v} \nonumber\\
        + \int_Y \A \simgrad  \vect u_2  :  \overline{ \simgrad  \vect v} 
        &+ \int_Y \A \simgrad  \vect u_2  :  \overline{  iX_\chi  \vect v} 
        &&+ \int_Y \A iX_\chi  \vect u_2  :  \overline{ \simgrad  \vect v} 
        &&+ \int_Y \A iX_\chi  \vect u_2  :  \overline{ iX_\chi  \vect v} \nonumber\\
        + \int_Y \A \simgrad  \vect u_{\rm err}  :  \overline{ \simgrad  \vect v} 
        &+ \int_Y \A \simgrad  \vect u_{\rm err}  :  \overline{  iX_\chi  \vect v} 
        &&+ \int_Y \A iX_\chi  \vect u_{\rm err}  :  \overline{ \simgrad  \vect v} 
        &&+ \int_Y \A iX_\chi  \vect u_{\rm err}  :  \overline{ iX_\chi  \vect v} \nonumber\\
        -z |\chi|^2 \int_Y \vect u_0 \cdot \overline{\vect v} 
        &-z |\chi|^2 \int_Y \vect u_1 \cdot \overline{\vect v}
        -z |\chi|^2 \int_Y \span\span \vect u_2 \cdot \overline{\vect v}  
        -z |\chi|^2 \int_Y \vect u_{\rm err} \cdot \overline{\vect v}
        \span\span= |\chi|^2\int_Y \vect f \cdot \overline{\vect v}. \span\span
\end{alignat}

The general heuristic in defining the terms $\vect u_i$ in \eqref{expansion} is to extract the terms in \eqref{firstexpansion} which would be of the same order in $|\chi|$, as $|\chi| \downarrow 0$, according to \eqref{eqn:expansion_bounds}, and combine them as to inductively define a sequence of boundary value problems for determining $\vect u_i$. It is crucial to verify the well-posedness of these problems, as to justify the method. Having done so, we estimate the error term $\vect u_{\rm err}$ in the $H^1$ norm.

Let us begin the inductive procedure. Equating the terms in \eqref{firstexpansion} which are of order $1$, we arrive at the following constraint on the leading order term $\vect u_0$, namely $\vect u_0\in H^{1}_\#(Y;\C^3)$ satisfies
\begin{equation}\label{eqn:lot_extra}
    \int_{Y} \A  \simgrad   \vect u_0  : \overline{ \simgrad \vect v} = 0, \quad \forall\vect v \in H^{1}_\#(Y;\C^3). 
\end{equation}
By Remark \ref{rmk:rigid_body_motions}, $\vect u_0$ must be a constant function $\vect u_0 \in \C^3$. In what follows, we will impose an additional condition making $\vect u_0$ unique (see (\ref{leadingorderterm})).

We proceed with the definition of the term $\vect u_1$, in which we suppose that $\vect u_0$ is readily available. Collecting the terms of order $|\chi|$ in \eqref{firstexpansion}, and noting that the $\mathcal{O}(|\chi|)$ term $\int_{Y} \A  \simgrad   \vect u_0  : \overline{  iX_\chi  \vect v}$ is zero since $\vect u_0$ is a constant, we arrive at the following problem, where we seek a unique solution $\vect u_1 \in H_{\#}^1(Y;\C^3)$:
\begin{equation}\label{corr2}
\begin{split}
    \int_{Y} \A  \simgrad   \vect u_1  : \overline{ \simgrad   \vect v} = - \int_{Y} \A  iX_\chi   \vect u_0  : \overline{ \simgrad   \vect v}, \quad \forall\vect v \in H^{1}_\#(Y;\C^3), \quad \text{ and } \quad \quad \int_Y \vect u_1 = 0.
\end{split}
\end{equation}
The quadratic form: $H^{1}_\#(Y;\C^3)\ni \vect v \mapsto \int_{Y} \A  \simgrad   \vect v  : \overline{ \simgrad   \vect v}$ is coercive on the space $\Dot{H}_\#^1(Y;\C^3)$, defined by
\begin{equation}
\label{Hdot}
    \Dot{H}_\#^1(Y;\C^3): = \left\{\vect v \in H_\#^1(Y;\C^3);  \int_Y \vect v = 0 \right\} = \ker(\simgrad)^\perp \cap H_\#^1(Y;\C^3),
\end{equation}
due to the Korn's inequality \eqref{korninequality3}. The orthogonal complement in \eqref{Hdot} is taken with respect to $L^2(Y;\C^3)$.  Furthermore, the functional: $H^{1}_\#(Y;\C^3)\ni \vect v \mapsto \int_{Y} \A  iX_\chi   \vect u_0  : \overline{ \simgrad   \vect v}$ is clearly bounded in $\Dot{H}_\#^1(Y;\C^3)$, with a bound depending on $\vect u_0$. The well-posedness of the problem \eqref{corr2} then follows from Lax-Milgram Lemma.

Next, we collect the terms of order $|\chi|^2$ to arrive at the following problem, where we seek a unique solution $\vect u_2 \in H_\#^1(Y;\C^3)$:
\begin{equation}\label{correctoru2}
\begin{split}
    \int_{Y} \A  \simgrad   \vect u_2  : \overline{ \simgrad   \vect v} = 
    & - \int_{Y} \A  iX_\chi    \vect u_1  : \overline{ \simgrad  \vect v} 
    - \int_{Y} \A  \simgrad   \vect u_1  : \overline{ i X_\chi   \vect v} - \int_{Y} \A  iX_\chi    \vect u_0  : \overline{ i X_\chi   \vect v} \\
    &+z |\chi|^2\int_Y \vect u_0 \cdot\overline{\vect v} + |\chi|^2\int_Y \vect f  \cdot \overline{\vect v}, 
    \quad \forall\vect v \in H^{1}_\#(Y;\C^3), \quad \text{ and } \quad \int_Y \vect u_2 = 0. 
\end{split}
\end{equation}
By choosing a constant test function $\vect v \in \C^3$ in the first equation of \eqref{correctoru2}, we note that the left-hand side vanishes. Since we would like the problem (\ref{correctoru2}) to have a solution, we ask that the right-hand side should vanish as well. That is, for $\vect v \in \C^3 \subset H^{1}_\#(Y;\C^3)$,
\begin{equation}\label{eqn:necessary_u2}
    - \int_{Y} \A  iX_\chi    \vect u_1  : \overline{ \simgrad  \vect v} 
    - \int_{Y} \A  \simgrad   \vect u_1  : \overline{ i X_\chi   \vect v} - \int_{Y} \A  iX_\chi    \vect u_0  : \overline{ i X_\chi   \vect v} - z|\chi|^2\int_Y \vect u_0 \cdot\overline{\vect v} + |\chi|^2\int_Y \vect f  \cdot \overline{\vect v} = 0.
\end{equation}
Under the assumption that (\ref{eqn:necessary_u2}) holds, a Lax-Milgram type argument similar to \eqref{corr2} shows that the solution $\vect u_2$ to (\ref{correctoru2}) exists and is uniquely determined.

Let us now derive a condition that will determine the constant $\vect u_0 \in \C^3$, which will in turn satisfy assumption (\ref{eqn:necessary_u2}): By testing with $\vect v = \vect v_0 \in \C^3$ in \eqref{correctoru2}, we obtain
\begin{align}
    &\cancel{\int_{Y} \A  \simgrad   \vect u_2  : \overline{ \simgrad   \vect v_0}} \nonumber \\
    &=
    \cancel{- \int_{Y} \A  iX_\chi    \vect u_1  : \overline{ \simgrad  \vect v_0}} 
     - \int_{Y} \A  \simgrad   \vect u_1  : \overline{ iX_\chi \vect v_0} - \int_{Y} \A  iX_\chi    \vect u_0  : \overline{ iX_\chi \vect v_0}  +z |\chi|^2\int_Y \vect u_0 \cdot \overline{\vect v_0} + |\chi|^2\int_Y \vect f \cdot \overline{\vect v_0} \nonumber\\
    &\overset{\eqref{homdefinition}}{=} - \left\langle \mathcal{A}_\chi^{\rm hom} \vect u_0,\vect v_0 \right\rangle +z |\chi|^2\int_Y \vect u_0 \cdot \overline{\vect v_0} + |\chi|^2\int_Y \vect f \cdot \overline{\vect v_0}.
\end{align}
In other words, we obtain the following equation: 
\begin{equation}\label{leadingorderterm_weak}
     \frac{1}{|\chi|^2}\left\langle \mathcal{A}_\chi^{\rm hom} \vect u_0,\vect v_0 \right\rangle_{\C^3} -z \int_Y \vect u_0 \cdot \overline{\vect v_0} = \int_Y \vect f \cdot \overline{\vect v_0}, 
    \quad \forall\vect v_0 \in \C^3,
\end{equation}
or equivalently (since $\int_Y = 1$),
\begin{equation}
\label{leadingorderterm}
    \frac{1}{|\chi|^2}\mathcal{A}_\chi^{\rm hom} \vect u_0 -z \vect u_0 = \int_Y \vect f.
\end{equation}
Thus, by defining the leading order term $\vect u_0 \in \C^3$ as the solution to (\ref{eqn:lot_extra}) and (\ref{leadingorderterm}), we have ensured the existence and uniqueness of the solution $\vect u_2$ to \eqref{correctoru2}.

\begin{remark}
    Under the condition \eqref{eqn:necessary_u2}, the problem \eqref{correctoru2} can be equivalently formulated as follows: Find $\vect u_2 \in \dot{H}_\#^1(Y;\C^3)$ such that:
    \begin{equation}\label{correctoru2Hdot}
    \begin{split}
        \int_{Y} \A  \simgrad   \vect u_2  : \overline{ \simgrad   \vect v} = 
        & - \int_{Y} \A  iX_\chi    \vect u_1  : \overline{ \simgrad  \vect v} 
        - \int_{Y} \A  \simgrad   \vect u_1  : \overline{ i X_\chi   \vect v} - \int_{Y} \A  iX_\chi    \vect u_0  : \overline{ i X_\chi   \vect v} \\
        & +|\chi|^2\int_Y S^\perp \vect f  \cdot \overline{\vect v}, 
        \quad \forall\vect v \in \dot{H}^{1}_\#(Y;\C^3),
    \end{split}
    \end{equation}
    where $S^\perp :L^2(Y;\C^3) \to L^2(Y;\C^3)$ is the projection operator defined by: 
    \begin{equation*}
        \label{Sperp} S^\perp \vect u:= (I-S) \vect u = \vect u - \int_Y \vect u. \qedhere
    \end{equation*}
\end{remark}

\begin{table}[ht]
\centering
\begin{tabular}{ |c|c|c|c| }
    \hline
    Function    & Equation(s) that it solves                            & Space that it belongs to  & Dependencies (excl.~$\chi$, $z$, and $\vect f$) \\
    \hline
    $\vect u_0$ & \eqref{eqn:lot_extra} and \eqref{leadingorderterm}    & $\C^3$
    & --- \\ 
    $\vect u_1$ & \eqref{corr2}                                         & $\Dot{H}_{\#}^1(Y;\C^3)$
    & $\vect u_0$\\ 
    $\vect u_2$ & \eqref{correctoru2}                                   & $\Dot{H}_{\#}^1(Y;\C^3)$
    & $\vect u_0$, $\vect u_1$\\ 
    $\vect u$   & \eqref{chiproblem}                                    & $H_{\#}^1(Y;\C^3)$ 
    & $\vect u_0$, $\vect u_1$, $\vect u_2$ \\ 
    \hline
\end{tabular}
\caption{The equations that $\vect u, \, \vect u_0, \, \vect u_1, \, \vect u_2 \in H_{\#}^1(Y;\C^3)$ uniquely solve, \\ with additional information.}\label{tab:first_cycle}
\centering
\end{table}

We summarize our derivation thus far in Table \ref{tab:first_cycle}. In the next subsection, we verify the estimates in \eqref{eqn:expansion_bounds}, and justify the expansion \eqref{expansion} with error estimates.

\subsubsection{Cycle 1: Estimates}\label{sect:the_asymp_method_cycle1est}

Having derived the equations for $\vect u_0, \vect u_1, \vect u_2$, we now prove the estimates in (\ref{eqn:expansion_bounds}), and study the error term $\vect u_{\rm err} := \vect  u - \vect u_0 - \vect u_1 -\vect u_2$. Note that $\vect u_{\rm err}$ is uniquely determined, since $\vect u_j$'s and $\vect u$ are uniquely determined. To keep track of the dependence on $z$, it will be convenient to introduce the following notation:

\begin{definition}[Constants depending on $z$]
    \begin{equation}\label{eqn:dist_to_spec}
        D_{\text{hom}}(z) := \text{dist} (z, \sigma(\tfrac{1}{|\chi|^2} \mathcal{A}_\chi^{\text{hom}} ) ), \qquad
        D(z) := \text{dist} (z, \sigma(\tfrac{1}{|\chi|^2} \mathcal{A}_\chi ) ).
    \end{equation}
\end{definition}
We claim that the following estimates hold:
\begin{equation}\label{bound11}
    \| \vect u_0 \|_{H^1(Y;\C^3)} \leq \frac{C}{D_{\text{hom}}(z)} \| \vect f \|_{L^2(Y;\C^3)},
\end{equation}
\begin{equation}\label{bound12}
    \left\lVert \vect u_1 \right\rVert_{H^{1}(Y;\C^3)} \leq \frac{C}{D_{\text{hom}}(z)} |\chi| \left\lVert \vect f \right\rVert_{L^2(Y;\C^3)},
\end{equation}
\begin{equation}\label{bound13}
    \left\lVert \vect u_2 \right\rVert_{H^{1}(Y;\C^3)} \leq C \left[\frac{\max \{1,|z|\}}{D_\text{hom}(z)} + 1 \right] |\chi|^2 \left\lVert \vect f \right\rVert_{L^2(Y;\C^3)},
\end{equation}
where the constant $C>0$ is independent of $\chi$ and $z$. These inequalities follow by routine applications of Assumption \ref{coffassumption}, Korn's inequalities (Appendix \ref{sect:useful_ineq}), and the definition of $\vect u_i$'s. We refer the reader to Appendix \ref{sect:routine_estimates} for their proof.

Now we turn our focus to the error term $\vect u_{\rm err}$. By collecting the leftover terms in (\ref{firstexpansion}), we see that $\vect u_{\text{err}} \in H_{\#}^1$ solves the following equation:
\begin{equation}\label{firsterrorequation}
    \frac{1}{|\chi|^2} \int_{Y} \A \left( \simgrad + iX_\chi \right) \vect u_{\rm err}  : \overline{\left( \simgrad + iX_\chi \right) \vect v} -z \int_Y \vect u_{\rm err}\cdot \overline{\vect v}  = \frac{1}{|\chi|^2}\mathcal{R}_{\rm err}(\vect v) ,\quad \forall\vect v\in H^{1}_\#(Y;\C^3), 
\end{equation}
where the functional $\mathcal{R}_{\rm err}$ is given by:
\begin{equation}
    \begin{split}
    \mathcal{R}_{\rm err}(\vect v) := &-  \int_{Y} \A   iX_\chi  \vect u_1  : \overline{ iX_\chi  \vect v} -  \int_{Y} \A  \simgrad   \vect u_2  : \overline{  iX_\chi  \vect v} -  \int_{Y} \A   iX_\chi  \vect u_2  : \overline{ \simgrad   \vect v} -  \int_{Y} \A   iX_\chi  \vect u_2  : \overline{ iX_\chi  \vect v} \\
    &+z |\chi|^2\int_Y \vect u_1 \cdot \overline{\vect v} +z |\chi|^2\int_Y \vect u_2 \cdot \overline{\vect v}.
    \end{split}
\end{equation}
By employing the estimates \eqref{bound11}, \eqref{bound12}, and \eqref{bound13} we obtain (see Appendix \ref{sect:routine_estimates} for details):
\begin{equation}\label{bounderror1}
    |\mathcal{R}_{\rm err}(\vect v)| \leq C \left[ \frac{\max \{1,|z|^2\}}{D_\text{hom}(z)} + \max \{1,|z|\} \right] |\chi|^3 \left\lVert\vect v \right\rVert_{H^{1}(Y;\C^3)} \left\lVert \vect f \right\rVert_{L^2(Y;\C^3)}. 
\end{equation}
That is, $\frac{1}{|\chi|^2}\mathcal{R}_\text{err}$ is a bounded linear functional on $H^1(Y;\C^3)$, with operator norm satisfying
\begin{equation}
    \left\| \frac{1}{|\chi|^2} \mathcal{R}_{\text{err}} \right\|_{(H^1(Y;\C^3))^*} \leq C \left[ \frac{\max \{1,|z|^2\}}{D_\text{hom}(z)} + \max \{1,|z|\} \right] |\chi| \| \vect f \|_{L^2}.
\end{equation}
Applying Proposition \ref{prop:abstract_ineq} with $H = L^2(Y;\C^3)$, $X = H^1_{\#}(Y;\C^3)$, $\mathcal{A} = \frac{1}{|\chi|^2} \mathcal{A}_\chi$, $\mathcal{R} = \frac{1}{|\chi|^2} \mathcal{R}_{\text{err}}$, together with Remark \ref{rmk:equiv_h1_a_X}, we obtain
\begin{equation}\label{bounduerr1}
    \| \vect u_{\rm err} \|_{H^1(Y;\C^3)} 
    \leq C \max \left\{1, \frac{|z +1|}{D(z)} \right\} \left[ \frac{\max \{1,|z|^2\}}{D_\text{hom}(z)} + \max \{1,|z|\} \right] |\chi| \| \vect f \|_{L^2(Y;\C^3)}. 
\end{equation}
Therefore, terms $\vect u_1$, $\vect u_2$ may be dropped, as they do not contribute to the order of the estimate. That is, 
\begin{align}
    \| \vect u - \vect u_0 \|_{H^{1}(Y;\C^3)} 
    &= \| \vect u_{\rm err} + \vect u_1 + \vect u_2 \|_{H^{1}(Y;\C^3)} \nonumber\\
    &\leq C \max \left\{1, \frac{|z +1|}{D(z)} \right\} \left[ \frac{\max \{1,|z|^2\}}{D_\text{hom}(z)} + \max \{1,|z|\} \right] |\chi| \| \vect f \|_{L^2(Y;\C^3)}. \label{chi1estimate}
\end{align}

\subsubsection{Cycle 2: Second expansion of the solution}

The procedure thus far proves \eqref{eqn:tau_resolvent_est1} of Theorem \ref{thm:tau_resolvent_est}, which, as we will see in Section \ref{sect:norm_resolvent_est}, is sufficient for the $L^2 \to L^2$ result of Theorem \ref{thm:main_thm}. On the contrary, it turns out that the first cycle alone is insufficient for the $L^2 \to H^1$ and higher-order $L^2\to L^2$ result, and thus we have to continue with the asymptotic procedure.

Since $\vect u_0$, $\vect u_1$, and $\vect u_2$ have been uniquely specified, we require a deeper investigation into $\vect u_{\text{err}}$, bringing us to a core idea of this procedure. The method we will demonstrate now was first employed in \cite{kirill_igor_plates} and later in \cite{kirill_igor_josip_rods}. Let us refine the expansion \eqref{expansion} in the following way:
\begin{equation}\label{expansion2}
    \begin{split}
        \vect u = \vect u_0 &+ \vect u_1 + \vect u_2  \\
        &+ \vect u_0^{(1)} + \vect u_1^{(1)} + \vect u_2^{(1)} + \vect u_{\rm err}^{(1)},
    \end{split}
\end{equation}
where the new corrector terms $\vect u_i^{(1)} \in H_\#^1(Y;\C^3)$ satisfy the following estimates in $H^{1}(Y;\C^3)$ norm:
\begin{align}\label{eqn:expansion_bounds2}
    \vect u_0^{(1)} = \mathcal{O}(|\chi|), \quad \vect u_1^{(1)} = \mathcal{O}(|\chi|^2), \quad \vect u_2^{(1)} = \mathcal{O}(|\chi|^3).
\end{align}
(We have specified the leading order term $\vect u_0^{(1)}$ to be of order $|\chi|$, as this is the order of $\vect u_{\text{err}}$ (by \eqref{bounduerr1}.)

Plugging the expansion \eqref{expansion2} into the equation \eqref{chiproblem} results again in a sizeable expression:
\begin{alignat}{5}\label{secondexpansion}
    \cancel{\int_Y \A \simgrad  \vect u_0  :  \overline{ \simgrad  \vect v}}
    &+ \cancel{\int_Y \A \simgrad  \vect u_0  :  \overline{ iX_\chi  \vect v}}
    &&+ \cancel{\int_Y \A iX_\chi  \vect u_0  :  \overline{ \simgrad  \vect v}}
    &&+ \cancel{\int_Y \A iX_\chi  \vect u_0  :  \overline{ iX_\chi  \vect v}} \nonumber\\
    + \cancel{\int_Y \A \simgrad  \vect u_1  :  \overline{ \simgrad  \vect v}} 
    &+ \cancel{\int_Y \A \simgrad  \vect u_1  :  \overline{ iX_\chi  \vect v}} 
    &&+ \cancel{\int_Y \A iX_\chi  \vect u_1  :  \overline{ \simgrad  \vect v}} 
    &&+ \int_Y \A iX_\chi  \vect u_1  :  \overline{ iX_\chi  \vect v} \nonumber\\
    + \cancel{\int_Y \A \simgrad  \vect u_2  :  \overline{ \simgrad  \vect v}} 
    &+ \int_Y \A \simgrad  \vect u_2  :  \overline{ iX_\chi  \vect v} 
    &&+ \int_Y \A iX_\chi  \vect u_2  :  \overline{ \simgrad  \vect v} 
    &&+ \int_Y \A iX_\chi  \vect u_2  :  \overline{ iX_\chi  \vect v} \nonumber\\
    + \int_Y \A \simgrad  \vect u_0^{(1)}  :  \overline{ \simgrad  \vect v} 
    &+ \int_Y \A \simgrad  \vect u_0^{(1)}  :  \overline{ iX_\chi  \vect v} 
    &&+ \int_Y \A iX_\chi  \vect u_0^{(1)}  :  \overline{ \simgrad  \vect v} 
    &&+ \int_Y \A iX_\chi  \vect u_0^{(1)}  :  \overline{ iX_\chi  \vect v} \nonumber\\
    + \int_Y \A \simgrad  \vect u_1^{(1)}  :  \overline{ \simgrad  \vect v} 
    &+ \int_Y \A \simgrad  \vect u_1^{(1)}  :  \overline{ iX_\chi  \vect v} 
    &&+ \int_Y \A iX_\chi  \vect u_1^{(1)}  :  \overline{ \simgrad  \vect v} 
    &&+ \int_Y \A iX_\chi  \vect u_1^{(1)}  :  \overline{ iX_\chi  \vect v} \nonumber\\
    + \int_Y \A \simgrad  \vect u_2^{(1)}  :  \overline{ \simgrad  \vect v} 
    &+ \int_Y \A \simgrad  \vect u_2^{(1)}  :  \overline{ iX_\chi  \vect v} 
    &&+ \int_Y \A iX_\chi  \vect u_2^{(1)}  :  \overline{ \simgrad  \vect v} 
    &&+ \int_Y \A iX_\chi  \vect u_2^{(1)}  :  \overline{ iX_\chi  \vect v} \nonumber\\
    + \int_Y \A \simgrad  \vect u_{\rm err}^{(1)}  :  \overline{ \simgrad  \vect v} 
    &+ \int_Y \A \simgrad  \vect u_{\rm err}^{(1)}  :  \overline{ iX_\chi  \vect v} 
    &&+ \int_Y \A iX_\chi  \vect u_{\rm err}^{(1)}  :  \overline{ \simgrad  \vect v} 
    &&+ \int_Y \A iX_\chi  \vect u_{\rm err}^{(1)}  :  \overline{ iX_\chi  \vect v} \nonumber\\
    -z |\chi|^2\int_Y \vect u_0^{(1)} \cdot \overline{\vect v} 
    \span -z |\chi|^2\int_Y \vect u_1^{(1)} \cdot \overline{\vect v} 
    \span\span -z |\chi|^2\int_Y \vect u_2^{(1)} \cdot \overline{\vect v} 
    \span\span -z \cancel{|\chi|^2\int_Y \vect u_0 \cdot \overline{\vect v}} 
    -z |\chi|^2\int_Y \vect u_1 \cdot \overline{\vect v} 
    -z |\chi|^2\int_Y \vect u_2 \cdot \overline{\vect v}  \nonumber\\ 
    -z |\chi|^2\int_Y \vect u_{\rm err} \cdot \overline{\vect v} 
    = \cancel{|\chi|^2\int_Y \vect f \cdot \overline{\vect v}},
\end{alignat}
where we can ignore the cancelled terms due to the definitions for $\vect u_0$, $\vect u_1$, and $\vect u_2$.

We now begin an inductive procedure to derive the problems for $\vect u_i^{(1)}$. Equating the terms in \eqref{secondexpansion} of order $|\chi|$, we arrive at the following problem, where we seek a unique solution $\vect u_0^{(1)} \in H_{\#}^1(Y;\C^3)$:
\begin{equation}\label{eqn:lotnew_extra}
    \int_Y \A \simgrad  \vect u_0^{(1)}  :  \overline{ \simgrad  \vect v} = 0, \quad \forall \vect v \in H^1_\#(Y;\C^3). 
\end{equation}
Once again, Remark \ref{rmk:rigid_body_motions} asserts that $\vect u_0^{(1)}$ must be a constant function $\vect u_0^{(1)} = \vect c^{(1)} \in \C^3$. In that follows, we will impose an additional condition on $\vect u_0^{(1)}$ making $\vect u_0^{(1)}$ unique (see \eqref{newconstantcorrector}).

Next, we collect the terms of order $|\chi|^2$, arriving at the following problem, where we seek a unique solution $\vect u_1^{(1)} \in H_{\#}^1(Y;\C^3)$:
\begin{equation}\label{corr2new}
    \int_Y \A \simgrad  \vect u_1^{(1)}  :  \overline{ \simgrad  \vect v} 
    = - \int_Y \A iX_\chi  \vect u_0^{(1)}  :  \overline{ \simgrad  \vect v},
    \quad \forall \vect v \in H^1_\#(Y;\C^3), \quad \text{and} \qquad \int_Y \vect u_1^{(1)} = 0.
\end{equation}
Note that the $\mathcal{O}(|\chi|^2)$ term $\int_Y \mathbb{A} \simgrad \vect u_0^{(1)} : \overline{iX_\chi \vect v}$ is zero, since $\vect u_0^{(1)}$ is constant. A Lax-Milgram type argument similar to that for \eqref{corr2} (the problem for $\vect u_1$) shows that the solution $\vect u_1^{(1)}$ exists and is unique.

Next, we collect the terms of order $|\chi|^3$ to arrive at the following problem, where we seek a unique solution $\vect u_2^{(1)} \in H_{\#}^1(Y;\C^3)$:
\begin{align}\label{correctoru3new}
    \int_Y \A \simgrad  \vect u_2^{(1)}  :  \overline{ \simgrad  \vect v} 
    = &- \int_Y \A iX_\chi  (\vect u_1^{(1)} + \vect u_2)  :  \overline{ \simgrad  \vect v} 
    - \int_Y \A \simgrad  (\vect u_1^{(1)} + \vect u_2)  :  \overline{ iX_\chi  \vect v} \nonumber\\
    &- \int_Y \A iX_\chi  (\vect u_0^{(1)} + \vect u_1)  :  \overline{ iX_\chi  \vect v} 
    +z |\chi|^2\int_Y \vect u_0^{(1)} \cdot \overline{\vect v} 
    +z |\chi|^2\int_Y \vect u_1 \cdot \overline{\vect v}, 
    \quad \forall \vect v \in H^{1}_\#(Y;\C^3), \nonumber\\
    &\quad \text{and} \qquad \int_Y \vect u_2^{(1)} = 0.
\end{align}

\noindent
Similarly to the problem for $\vect u_2$, we assume that the right-hand side of the first equation of \eqref{correctoru3new} should vanish when tested against constants $\vect v \in \C^3$. Under this assumption, a Lax-Milgram type argument similar to \eqref{corr2} shows that the solution $\vect u_2^{(1)}$ to \eqref{correctoru3new} exists and is unique.

Let us now derive a condition that will determine the consant $\vect u_0^{(1)} \in \C^3$, which will in turn satisfy the assumption that the right-hand side of \eqref{correctoru3new} vanishes when tested against constants: By testing $\vect v_0 \in \C^3$ in \eqref{correctoru3new}, we have
\begin{align}\label{correctoru3new2}
    \cancel{\int_Y \A \simgrad  \vect u_2^{(1)}  :  \overline{ \simgrad  \vect v_0}}
    = &\cancel{- \int_Y \A iX_\chi  (\vect u_1^{(1)} + \vect u_2)  :  \overline{ \simgrad  \vect v_0}} 
    - \int_Y \A \simgrad  (\vect u_1^{(1)} + \vect u_2)  :  \overline{ iX_\chi  \vect v_0} \nonumber\\
    &- \int_Y \A iX_\chi (\vect u_0^{(1)} + \vect u_1)  :  \overline{ iX_\chi  \vect v_0} 
    +z |\chi|^2 \int_Y \vect u_0^{(1)} \cdot \overline{\vect v_0}
    \cancel{+z |\chi|^2 \int_Y \vect u_1 \cdot \overline{\vect v_0}}, \nonumber\\
    = - \int_Y \A \simgrad  (\vect u_1^{(1)} + \vect u_2)  :  \overline{ iX_\chi  \vect v_0} 
    \span - \int_Y \A iX_\chi  (\vect u_0^{(1)} + \vect u_1)  :  \overline{ iX_\chi  \vect v_0} 
    +z |\chi|^2 \int_Y \vect u_0^{(1)} \cdot \overline{\vect v_0}.
\end{align}
Here, the term $\int \vect u_1 \cdot \overline{\vect v_0}$ is zero, since $\vect u_1 \in \Dot{H}_\#^1(Y;\C^3)$. Now similarly to the previous cycle, the right-hand side of \eqref{correctoru3new2} can be expressed in terms of the matrix $\mathcal{A}_\chi^{\text{hom}}$. Indeed, let us recall that by Definition \ref{defn:hom_matrix} and by \eqref{corr2new} (the problem for $\vect u_1^{(1)}$), we have
\begin{equation}\label{homformula2}
    \left\langle \mathcal{A}^{\rm hom}_\chi \vect u_0^{(1)}, \vect v_0 \right\rangle_{\C^3} 
    = \int_Y \A \left( \simgrad  \vect u_1^{(1)}  +  iX_\chi  \vect u_0^{(1)} \right)  :  \overline{ iX_\chi \vect v_0 }, \quad \forall \vect v_0\in \C^3.
\end{equation}
Therefore, we may substitute \eqref{homformula2} into \eqref{correctoru3new2} to obtain the following problem that will determine the constant $\vect u_0^{(1)}$:
\begin{equation}
    \label{newconstantcorrector}
    \begin{split}
        \left\langle \mathcal{A}^{\rm hom}_\chi \vect u_0^{(1)}, \vect v_0 \right\rangle_{\C^3} 
        -z |\chi|^2 \int_Y \vect u_0^{(1)} \cdot \overline{\vect v_0} 
        & = - \int_Y \A  \left( \simgrad  \vect u_2 + iX_\chi  \vect u_1 \right)  :  \overline{ i X_\chi    \vect v_0} , 
        \quad \forall \vect v_0 \in \C^3.
    \end{split}
\end{equation}

\begin{table}[ht]
\centering
\begin{tabular}{ |c|c|c|c| }
    \hline
    Function    & Equation(s) that it solves                            & Space that it belongs to  & Dependencies (excl.~$\chi$, $z$, and $\vect f$) \\
    \hline
    $\vect u_0^{(1)}$ & \eqref{eqn:lotnew_extra} and \eqref{newconstantcorrector}    & $\C^3$
    & $\vect u_0$, $\vect u_1$, $\vect u_2$ \\ 
    $\vect u_1^{(1)}$ & \eqref{corr2new}                                         & $\Dot{H}_{\#}^1(Y;\C^3)$
    & $\vect u_0$, $\vect u_1$, $\vect u_2$, $\vect u_0^{(1)}$\\ 
    $\vect u_2^{(1)}$ & \eqref{correctoru3new}                                   & $\Dot{H}_{\#}^1(Y;\C^3)$
    & $\vect u_0$, $\vect u_1$, $\vect u_2$, $\vect u_0^{(1)}$, $\vect u_1^{(1)}$ \\
    \hline
\end{tabular}
\caption{The equations that $\vect u_0^{(1)}, \, \vect u_1^{(1)}, \, \vect u_2^{(1)} \in H_{\#}^1(Y;\C^3)$ uniquely solves, \\ with additional information.}\label{tab:second_cycle}
\centering
\end{table}

We summarize our derivation thus far in Table \ref{tab:second_cycle}. In the next subsection we verify the estimates in \eqref{eqn:expansion_bounds2}, and justify the expansion \eqref{expansion2} with error estimates.

\subsubsection{Cycle 2: Estimates}\label{sect:the_asymp_method_cycle2est}
Due to the already calculated estimates on the terms obtained in the first cycle of approximation, it is straightforward to obtain the following set of estimates:
\begin{equation}\label{bound21}
    \| \vect u_0^{(1)} \|_{H^1(Y;\C^3)} \leq C \left[ \frac{\max \{1,|z|\}}{D_\text{hom}(z)^2} + \frac{1}{D_\text{hom}(z)} \right] |\chi| \lVert \vect f \rVert_{L^2(Y;\C^3)},
\end{equation}
\begin{equation}\label{bound22}
    \| \vect u_1^{(1)} \|_{H^1(Y;\C^3)} \leq C \left[ \frac{\max \{1,|z|\}}{D_\text{hom}(z)^2} + \frac{1}{D_\text{hom}(z)} \right] |\chi|^2 \lVert \vect f \rVert_{L^2(Y;\C^3)},
\end{equation}
\begin{equation}\label{bound23}
    \| \vect u_2^{(1)} \|_{H^1(Y;\C^3)} 
    \leq C \left[ \frac{\max \{1,|z|^2\}}{D_\text{hom}(z)^2} 
    + \frac{\max \{1,|z|\}}{D_\text{hom}(z)} + 1 \right] |\chi|^3 \lVert \vect f \rVert_{L^2(Y;\C^3)},
\end{equation}
where the constant $C>0$ do not depend on $\chi$ nor $z$, and $D_{\rm hom}(z)$ is defined with \eqref{eqn:dist_to_spec}. Once again, we refer the reader to Appendix \ref{sect:routine_estimates} for their proof.

Now we turn to the error term $\vect u_{\rm err}^{(1)}:= \vect u - \vect u_0 - \vect u_1 -\vect u_2 - \vect u_0^{(1)} - \vect u_1^{(1)} -\vect u_2^{(1)}$. By collecting the leftover terms in \eqref{secondexpansion}, we see that $\vect u_{\text{err}}^{(1)}$ solves the following equation
\begin{equation}
    \label{seconderrorequation}
        \frac{1}{|\chi|^2} \int_Y \A \left( \simgrad + iX_\chi \right) \vect u_{\rm err}^{(1)}  : \overline{\left( \simgrad + iX_\chi \right) \vect v} 
        -z \int_Y \vect u_\text{err}^{(1)} \cdot \overline{\vect v}
        = \frac{1}{|\chi|^2}\mathcal{R}_{\rm err}^{(1)}(\vect v) ,\quad \forall \vect v\in H^1(Y;\C^3), 
\end{equation}
where the functional $\mathcal{R}_{\rm err}^{(1)}$ is given by:
\begin{align}
    \mathcal{R}_{\text{err}}^{(1)} (\vect v) := 
    &- \int_Y \A iX_\chi  \vect u_2  :  \overline{ iX_\chi  \vect v}
    - \int_Y \A iX_\chi  \vect u_1^{(1)}  :  \overline{ iX_\chi  \vect v}
    - \int_Y \A \simgrad  \vect u_2^{(1)}  :  \overline{ iX_\chi  \vect v} 
    - \int_Y \A iX_\chi  \vect u_2^{(1)}  :  \overline{ \simgrad  \vect v} \nonumber\\
    &- \int_Y \A iX_\chi  \vect u_2^{(1)}  :  \overline{ iX_\chi  \vect v}
    +z |\chi|^2\int_Y \vect u_1^{(1)} \cdot \overline{\vect v} 
    +z |\chi|^2\int_Y \vect u_2^{(1)} \cdot \overline{\vect v} 
    +z |\chi|^2\int_Y \vect u_2 \cdot \overline{\vect v}.
\end{align}
In the same manner as $\mathcal{R}_{\text{err}}$, we have the following bound for $\mathcal{R}_{\rm err}^{(1)}$ (see Appendix \ref{sect:routine_estimates} for details):
\begin{equation}\label{bounderror2}
    |\mathcal{R}_{\rm err}^{(1)}(\vect v)| 
    \leq C \left[ \frac{\max \{1,|z|^3\}}{D_\text{hom}(z)^2} 
    + \frac{\max \{1,|z|^2\}}{D_\text{hom}(z)} 
    + \max \{1,|z|\} \right]
    |\chi|^4 \lVert \vect v \rVert_{H^1(Y;\C^3)} \lVert \vect f \rVert_{L^2(Y;\C^3)}.
\end{equation}
By an application of Proposition \ref{prop:abstract_ineq}, we obtain
\begin{align}
    &\| \vect u - \vect u_0 - \vect u_1 - \vect u_2 - \vect u_0^{(1)} - \vect u_1^{(1)} -\vect u_2^{(1)} \|_{H^1(Y;\C^3)} = \| \vect u_{\text{err}}^{(1)} \|_{H^1(Y;\C^3)} \nonumber \\
    &\qquad \leq C \text{max} \left\{1, \frac{|z+1|}{D(z)} \right\}
    \left[ \frac{\max \{1,|z|^3\}}{D_\text{hom}(z)^2} 
    + \frac{\max \{1,|z|^2\}}{D_\text{hom}(z)} 
    + \max \{1,|z|\} \right]
    |\chi|^2 \| \vect f \|_{L^2(Y;\C^3)}. \label{estim81}
\end{align}
Here, we can ignore the terms $\vect u_2$, $\vect u_1^{(1)}$ and $\vect u_2^{(1)}$ due to the order of the estimate \eqref{estim81} to obtain
\begin{align}
    &\lVert \vect u - \vect u_0 - \vect u_1 - \vect u_0^{(1)} \rVert_{H^{1}(Y;\C^3)} \nonumber\\
    &\qquad \leq C \text{max} \left\{1, \frac{|z+1|}{D(z)} \right\}
    \left[ \frac{\max \{1,|z|^3\}}{D_\text{hom}(z)^2} 
    + \frac{\max \{1,|z|^2\}}{D_\text{hom}(z)} 
    + \max \{1,|z|\} \right]
    |\chi|^2 \| \vect f \|_{L^2(Y;\C^3)}. \label{chi2estimate}
\end{align}

\begin{remark}
    The ``restarting procedure" can be iterated in a similar fashion to achieve better approximations in $\chi$. For purposes of demonstration, we have done the ``restart" once, which is already enough to give us the results of Theorem \ref{thm:main_thm}.
\end{remark}

\subsection{Fibrewise results}\label{sect:asymp_results_chi}

We now summarize the results of the asymptotic procedure. Define the following operators

\begin{definition}[Corrector operators]\label{defn:correctors}
    Let $\chi \in Y'\setminus \{ 0 \}$ and $z \in \rho(\frac{1}{|\chi|^2} \mathcal{A}_\chi ) \cap \rho(\frac{1}{|\chi|^2} \mathcal{A}_\chi^{\text{hom}} )$. Let $\mathcal{R}_{\rm corr, 1, \chi}(z)$ and $\mathcal{R}_{\rm corr, 2, \chi}(z)$ be the bounded operators
    \begin{equation}
        \mathcal{R}_{\rm corr, 1, \chi}(z) : L^2(Y;\C^3) \to H_\#^1(Y;\C^3) \hookrightarrow L^2(Y;\C^3), 
        \quad \mathcal{R}_{\rm corr, 2, \chi}(z) : L^2(Y;\C^3) \to \C^3 \hookrightarrow L^2(Y;\C^3),
    \end{equation}
    with 
    \begin{equation}
    \label{corrector_defn}
        \mathcal{R}_{\rm corr, 1, \chi}(z) \vect f:= \vect u_1, \quad \mathcal{R}_{\rm corr, 2, \chi}(z) \vect f:= \vect u_0^{(1)},
    \end{equation}
    where $\vect u_1$, $\vect u_0^{(1)}$ are defined using the procedure above (by formulas \eqref{corr2} and \eqref{newconstantcorrector} respectively). We will refer to $\mathcal{R}_{\rm corr, 1, \chi}(z)$ is the first-order corrector operator, and to $\mathcal{R}_{\rm corr, 2, \chi}(z)$ as the second-order corrector operator.
\end{definition}

Note that \eqref{bound12} and \eqref{bound21} provide bounds for the operator norms of $\mathcal{R}_{\rm corr, 1, \chi}(z)$ and $\mathcal{R}_{\rm corr, 2, \chi}(z)$.

Recall the notation for the averaging operator $S$ defined in Section \ref{sect:avg_smoothing_ops}, and the constants $D_\text{hom}(z)$, $D(z)$ from \eqref{eqn:dist_to_spec}. The results of the asymptotic procedure can now be stated as follows:

\begin{theorem} \label{thm:tau_resolvent_est}
    Let $\chi \in Y'\setminus \{ 0 \}$ and $z \in \rho(\frac{1}{|\chi|^2} \mathcal{A}_\chi ) \cap \rho(\frac{1}{|\chi|^2} \mathcal{A}_\chi^{\text{hom}} )$, where the operator $\mathcal{A}_{\chi}$ is defined in Definition \ref{definitionofAchi}. There exists a constant $C > 0$, which does not depend on $\chi$ and $z$, such that the following norm-resolvent estimates hold:
    \begin{align}
        &\left\lVert \left(\frac{1}{|\chi|^2}\mathcal{A}_{\chi} - zI \right)^{-1} - \left(\frac{1}{|\chi|^2}\mathcal{A}_\chi^{\rm hom} - zI_{\C^3} \right)^{-1}S  \right\rVert_{L^2(Y;\C^3) \to H^{1}(Y;\C^3)} \nonumber\\
        &\qquad \leq C \text{max} \left\{1, \frac{|z+1|}{D(z)} \right\}
        \left[ \frac{\max \{1,|z|^2\}}{D_\text{hom}(z)} 
        + \max \{1,|z|\} \right]
        |\chi|, \label{eqn:tau_resolvent_est1} \\
        &\left\lVert \left(\frac{1}{|\chi|^2}\mathcal{A}_{\chi} - z I \right)^{-1} - \left(\frac{1}{|\chi|^2}\mathcal{A}_\chi^{\rm hom} - zI_{\C^3} \right)^{-1}S  - \mathcal{R}_{\rm corr, 1, \chi}(z) - \mathcal{R}_{\rm corr, 2, \chi}(z) \right\rVert_{L^2(Y;\C^3) \to H^{1}(Y;\C^3)} \nonumber \\
        &\qquad \leq C \text{max} \left\{1, \frac{|z+1|}{D(z)} \right\}
        \left[ \frac{\max \{1,|z|^3\}}{D_\text{hom}(z)^2} 
        + \frac{\max \{1,|z|^2\}}{D_\text{hom}(z)} 
        + \max \{1,|z|\} \right]
        |\chi|^2. \label{eqn:tau_resolvent_est2}
    \end{align}
    Here, $\mathcal{A}_\chi^{\rm hom}$ is the constant matrix defined by (\ref{homdefinition}). Moreover, $\mathcal{A}_\chi^{\rm hom}$ is quadratic in $\chi$, with
    \begin{equation}
        \mathcal{A}_\chi^{\rm hom} = \left(iX_\chi \right)^*\mathbb{A}^{\rm hom} \left( i X_\chi \right),
    \end{equation}
    where $\mathbb{A}^{\rm hom} \in \R^{3 \times 3 \times 3 \times 3}$ is a constant tensor satisfying the same symmetries as $\mathbb{A}$ in (\ref{assump:symmetric}).
\end{theorem}
\begin{proof}
    The norm-resolvent estimate \eqref{eqn:tau_resolvent_est1} is the consequence of the estimate \eqref{chi1estimate}, while the norm-resolvent estimate \eqref{eqn:tau_resolvent_est2} is a consequence of the estimate \eqref{chi2estimate}.
\end{proof}

\begin{remark}
    \begin{itemize}
        \item In the following section, we will only consider those $z$ taken from a compact set $\Gamma$, which is uniformly bounded away from both $\sigma(\frac{1}{|\chi|^2} \mathcal{A}_\chi )$ and $\sigma (\frac{1}{|\chi|^2} \mathcal{A}_\chi^{\text{hom}})$. (See Lemma \ref{lemma:contour}.) This permits us to further bound \eqref{eqn:tau_resolvent_est1} by $C |\chi| \| \vect f \|_{L^2}$, and \eqref{eqn:tau_resolvent_est2} by $C |\chi|^2 \| \vect f \|_{L^2}$, where $C>0$ is independent of $\chi$, and only depends on $z$ through the set $\Gamma$.

        \item To bring the discussion from $L^2(Y;\C^3)$ back to $L^2(\R^3;\C^3)$, we will need to define full-space analogues of $\mathcal{R}_{\rm corr,1,\chi}(z)$ and $\mathcal{R}_{\rm corr,2,\chi}(z)$. This will be done in Section \ref{sect:corrector_ops}. \qedhere
    \end{itemize}
\end{remark}

\section{Norm-resolvent estimates}\label{sect:norm_resolvent_est}

In this section, we will explain how the eigenvalue estimates in Section \ref{sect:evalue_estimates}, together with the resolvent estimates in Section \ref{sect:asymp_results_chi} combine to give us the norm-resolvent estimates on the full space.

As it will become clear in the proof of the theorems below, the case of large $\chi$ will not contribute to the final error estimate. We therefore focus our efforts on those quasimomenta $\chi$ in the neighborhood of $\chi=0$.

\begin{lemma}
\label{lemma:contour}
    There exists a closed contour  $\Gamma \subset \left\{z\in \C, \Re(z)>0\right\}$, oriented anticlockwise, where the following are valid:
    
    \begin{itemize}
        \item \textbf{(Separation of spectrum)} There exist some $\mu>0$, such that for each $\chi \in [-\mu,\mu]^3 \setminus \left\{0\right\}$, $\Gamma$ encloses the three smallest eigenvalues of the operators $\frac{1}{|\chi|^2}\mathcal{A}_\chi$ and $\frac{1}{|\chi|^2}\mathcal{A}_{\chi}^{\rm hom}$. That is, the points
        \begin{align}\label{eqn:contour_separation}
            \frac{1}{|\chi|^2} \lambda_i^\chi, \quad \frac{1}{|\chi|^2} \lambda_i^{\text{hom},\chi}, \quad i=1,2,3.
        \end{align}
        Furthermore, $\Gamma$ does not enclose any other eigenvalues of $\frac{1}{|\chi|^2}\mathcal{A}_\chi$ (and $\frac{1}{|\chi|^2}\mathcal{A}_{\chi}^{\rm hom}$).

        \item \textbf{(Buffer between contour and spectra)} There exist some $\rho_0 > 0$ such that 
        \begin{align}\label{eqn:contour_buffer}
            \inf_{\substack{ z\in \Gamma, \\ \chi \in [-\mu,\mu]^3\setminus \{ 0 \} \\ i \in \{1,2,3,4\}}} \left|z - \frac{1}{|\chi|^2}\lambda_i^{\chi} \right| \geq \rho_0
            \qquad \text{ and } \quad
            \inf_{\substack{ z\in \Gamma, \\ \chi \in [-\mu,\mu]^3\setminus \{ 0 \} \\ i \in \{1,2,3\}} } \left|z - \frac{1}{|\chi|^2}\lambda_i^{\text{hom},\chi} \right| \geq \rho_0.
        \end{align}
        
    \end{itemize}
\end{lemma}
\begin{proof}
    From Proposition \ref{prop:Rayleighestim}, we know that the first three eigenvalues $\frac{1}{|\chi|^2} \lambda_1^{\chi}, \frac{1}{|\chi|^2} \lambda_2^{\chi}, \frac{1}{|\chi|^2} \lambda_3^{\chi}$ of the rescaled operator $\frac{1}{|\chi|^2} \mathcal{A}_\chi$ are bounded and bounded away from zero, uniformly in $\chi$ (as $|\chi| \downarrow 0$). By Proposition \ref{prop:hom_matrix} the same is also true for the eigenvalues $\frac{1}{|\chi|^2} \lambda_1^{\text{hom},\chi}, \frac{1}{|\chi|^2} \lambda_2^{\text{hom},\chi}, \frac{1}{|\chi|^2} \lambda_3^{\text{hom},\chi}$ of the rescaled matrix $\frac{1}{|\chi|^2} \mathcal{A}_\chi^{\text{hom}}$. The remaining eigenvalues of $\frac{1}{|\chi|^2} \mathcal{A}_\chi$ are of order $\mathcal{O}(|\chi|^{-2})$. 
\end{proof}

\begin{remark}
    The contour $\Gamma$ and the constants $\mu$ and $\rho_0$ are independent of $\chi$ and $\eps$. This is crucial for the arguments we will make in this section. It is also crucial that $\Gamma$ does not cross the imaginary axis $i\mathbb{R}$, see Lemma \ref{lem:g_str}.
\end{remark}

We refer the reader to Figure \ref{fig:contour} for a schematic of the contour $\Gamma$, for small $\chi$.

\begin{figure}[h]
  \centering
  \includegraphics[page=1, clip, trim=4cm 4cm 5cm 5cm, width=0.90\textwidth]{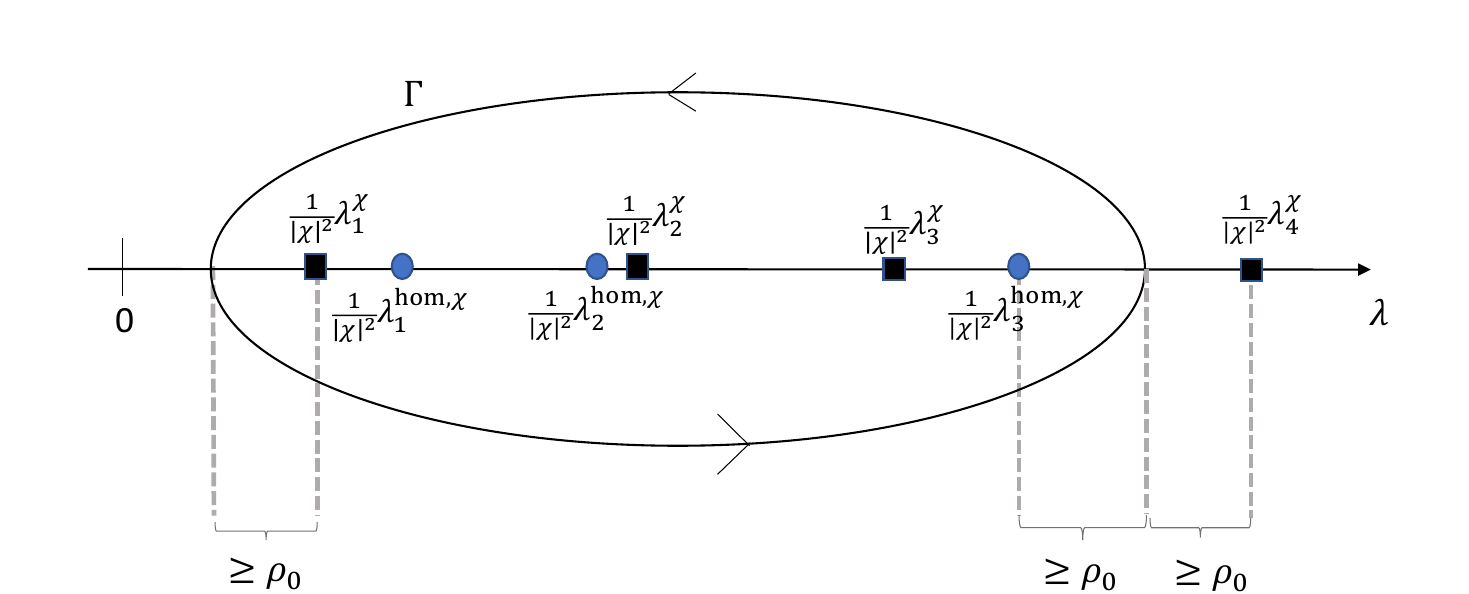}
  \caption{A schematic of the contour $\Gamma$, for quasimomentum $\chi \in [-\mu,\mu]^3\setminus \{ 0 \}$.} \label{fig:contour}
\end{figure}

Next, we would like to use Theorem \ref{thm:tau_resolvent_est} to deduce estimates on the resolvents of $\frac{1}{\eps^2} \mathcal{A}_\chi$ and $\frac{1}{\eps^2} \mathcal{A}_\chi^{\rm hom}$. To that end, it will be convenient to introduce the following function:

\begin{definition}
    For each $\eps>0$, $\chi \neq 0$, $\gamma > -2$, define the function $g_{\eps,\chi}: \{z \in \C : \Re(z) > 0 \} \rightarrow \C$ 
    \begin{equation}\label{def_g_analytic}
       g_{\eps,\chi}(z) := \left(\frac{|\chi|^2}{\eps^{\gamma + 2}}z + 1 \right)^{-1}. 
    \end{equation} 
\end{definition}

\begin{remark}
    For simplicity, the reader might wish to focus on the case $\gamma = 0$ (which will recover the results of \cite{birman_suslina_2004}. The parameter $\gamma$ acts as a scaling on the spectral parameter, meant for a discussion on various high\,/\,low frequency regimes.
\end{remark}

Note that $g_{\eps,\chi}$ is analytic on the right half plane, thus for $a \in \C$ enclosed with $\Gamma$ the Cauchy integral formula gives us:
$$g_{\eps,\chi}(a):= \frac{1}{2\pi{\rm i}} \oint_\Gamma \frac{g_{\eps,\chi}(z)}{z-a}dz.$$
Using the Dunford-Riesz functional calculus \cite[Sect VII.9]{dunford_schwartz1}, a similar formula is valid when $g_{\eps,\chi}$ is applied to a closed operator $\mathcal{A}$:
\begin{align}
    g_{\eps,\chi}(\mathcal{A})P_{\Gamma, \mathcal{A}}:=-\frac{1}{2\pi{\rm i}} \oint_\Gamma g_{\varepsilon,\chi}(z)(\mathcal{A}-zI)^{-1}dz, \label{eqn:cauchy_int_formula}
\end{align}
where $P_{\Gamma, \mathcal{A}}$ is the projection operator onto the eigenspace corresponding to the eigenvalues enclosed by the contour $\Gamma$, in the case where $\mathcal{A}$ has compact resolvent.

Another property that we will need from $g_{\eps,\chi}$ is the following estimate:

\begin{lemma}\label{lem:g_str}
    For every fixed $\eta > 0$, function  $g_{\varepsilon,\chi}$ is bounded on the half plane $\{z \in \C, \Re(z) \geq \eta\}$, with
    \begin{equation}
        |g_{\varepsilon,\chi}(z)| \leq C_\eta \left(\max\left\{\frac{|\chi|^2}{\varepsilon^{\gamma + 2}}, 1\right\}\right)^{-1}.
    \end{equation}
\end{lemma}
 
\begin{proof}
    Focusing on the real part $\Re(g_{\eps,\chi}(z))$ of the complex number $g_{\eps,\chi}(z)$, we estimate
    \begin{align*}
        |g_{\varepsilon,\chi}(z)|^{-1}
        = \left\lvert\frac{|\chi|^2}{\varepsilon^{\gamma + 2}}z + 1\right\rvert 
        \geq \frac{|\chi|^2}{\varepsilon^{\gamma + 2}}\eta + 1 
        \geq C \max\left\{\frac{|\chi|^2}{\varepsilon^{\gamma + 2}}, 1\right\}. \tag*{\qedhere}
    \end{align*} 
\end{proof}

\subsection{\texorpdfstring{$L^2 \to L^2$}{L2 to L2} estimates}

Recall the definition of the homogenized operator $\mathcal{A}^{\rm hom}$ in Definition \ref{defn:hom_operator_fullspace}, and the smoothing operator $\Xi_\varepsilon$ in Definition \ref{defn:smoothing_op}.

\begin{theorem}\label{thm:l2l2}
    For $\eps>0$ small enough, there exists a constant $C>0$, independent of $\eps$, such that, the following norm-resolvent estimate holds for all $\gamma > -2$:
    \begin{equation}\label{eqn:l2l2}
        \left\Vert \left( \frac{1}{\varepsilon^{\gamma }}\mathcal{A}_\varepsilon + I\right)^{-1} -\left( \frac{1}{\varepsilon^{\gamma}}\mathcal{A}^{\rm hom} + I\right)^{-1} \Xi_\varepsilon\right\Vert_{L^2(\R^3;\R^3) \to L^2(\R^3;\R^3)}  
        \leq C \varepsilon^\frac{\gamma + 2}{2}, 
    \end{equation}
    where the operator $\mathcal{A}_\varepsilon$ is given by the Definition \ref{defn:main_operator_fullspace}.
\end{theorem}

\begin{proof}
    \subsubsection*{Step 1: Estimates on \texorpdfstring{$L^2(Y;\C^3)$}{L2(Y;C3)}.}
    
    Let us denote $P_\chi$ for the projection onto the three-dimensional eigenspace of $\mathcal{A}_\chi$ associated with three eigenvalues of order $|\chi|^2$ (Proposition \ref{prop:Rayleighestim}). Then, we may decompose the resolvent of $\frac{1}{\eps^{\gamma + 2}}\mathcal{A}_\chi$ as follows:
    \begin{align}
        \left(\frac{1}{\eps^{\gamma + 2}}\mathcal{A}_\chi + I\right)^{-1} 
        = P_\chi\left(\frac{1}{\eps^{\gamma + 2}}\mathcal{A}_\chi + I\right)^{-1} P_\chi 
        + (I - P_\chi) \left(\frac{1}{\eps^{\gamma + 2}}\mathcal{A}_\chi + I\right)^{-1} (I - P_\chi).
    \end{align}
    To estimate the second term, we use (\ref{Rayleighestim_3}) of Proposition \ref{prop:Rayleighestim}, giving us
    \begin{align}\label{eqn:l2l2_other_evalues}
            \left\| (I - P_\chi) \left( \frac{1}{\eps^{\gamma + 2}}\mathcal{A}_\chi + I\right)^{-1} (I - P_\chi) \right\|_{L^2(Y;\C^3) \to L^2(Y;\C^3)} \leq C\eps^{\gamma + 2}.
    \end{align}
    As for the first term, we split into two further cases:

    \textbf{Case 1: $\chi \in [-\mu,\mu]^3\setminus \{ 0 \}$.} We first apply the Cauchy formula (\ref{eqn:cauchy_int_formula}) with the contour $\Gamma$ provided by the Lemma \ref{lemma:contour}, giving us 
    \begin{align}
            P_\chi\left(\frac{1}{\eps^{\gamma + 2}}\mathcal{A}_\chi + I\right)^{-1} P_\chi 
            = g_{\eps,\chi}\left( \frac{1}{|\chi|^2} \mathcal{A}_\chi \right) P_{\Gamma,\frac{1}{|\chi|^2} \mathcal{A}_\chi} 
            \stackrel{\text{(\ref{eqn:cauchy_int_formula})}}{=} -\frac{1}{2\pi i} \oint_\Gamma g_{\eps,\chi}(z)\left(\frac{1}{|\chi|^2}\mathcal{A}_\chi - zI \right)^{-1}  dz, \label{eqn:contour_calc1}\\
            \left(\frac{1}{\eps^{\gamma + 2}}\mathcal{A}^{\rm hom}_\chi + I_{\C^3}\right)^{-1} S
            = g_{\eps,\chi}\left( \frac{1}{|\chi|^2} \mathcal{A}_\chi^{\rm hom} \right) P_{\Gamma,\frac{1}{|\chi|^2} \mathcal{A}_\chi^{\rm hom}} 
            \stackrel{\text{(\ref{eqn:cauchy_int_formula})}}{=} -\frac{1}{2\pi i} \oint_\Gamma g_{\eps,\chi}(z)\left(\frac{1}{|\chi|^2}\mathcal{A}_\chi^{\rm hom} - zI \right)^{-1}  dz. \label{eqn:contour_calc2}
    \end{align}
    Then, we may estimate the difference between the first term and $\left( \frac{1}{\eps^{\gamma + 2}}\mathcal{A}_\chi^{\rm hom} + I\right)^{-1}$ as follows: 
    \begin{align}
        &\left\Vert P_\chi\left( \frac{1}{\eps^{\gamma + 2}}\mathcal{A}_\chi + I\right)^{-1} P_\chi -\left( \frac{1}{\eps^{\gamma + 2}}\mathcal{A}_\chi^{\rm hom} + I_{\C^3} \right)^{-1} S \right\Vert_{L^2(Y;\C^3) \to L^2(Y;\C^3)}  
        \nonumber \\
        &\quad \leq \frac{1}{2\pi} \oint_{\Gamma} | g_{\eps,\chi}(z)|\left\Vert \left( \frac{1}{|\chi|^{2}}\mathcal{A}_\chi - zI \right)^{-1} -\left( \frac{1}{|\chi|^{2}}\mathcal{A}_\chi^{\rm hom} - zI_{\C^3} \right)^{-1} S \right\Vert_{L^2(Y;\C^3) \to L^2(Y;\C^3)} dz \label{eqn:link_eps_and_chi}\\
        &\qquad\qquad\qquad \text{By (\ref{eqn:contour_calc1}) and (\ref{eqn:contour_calc2}).} \nonumber \\
        &\quad \leq  C\left( \max\left\{\frac{|\chi|^{2}}{\eps^{\gamma + 2}}, 1\right\} \right)^{-1}  |\chi| \label{eqn:smallchi_l2l2_intermediate} \\
        &\qquad\qquad\qquad \parbox{30em}{By Lemma \ref{lem:g_str} and Theorem \ref{thm:tau_resolvent_est}, since $\chi \neq 0$. \\ $C>0$ does not depend on $z$, because $\Gamma$ satisfies (\ref{eqn:contour_buffer}).} \nonumber \\
        &\quad \leq C \eps^\frac{\gamma + 2}{2}. \label{eqn:smallchi_l2l2_final}
    \end{align}

    \textbf{Case 2: $\chi \in Y'\setminus [-\mu,\mu]^3$.} Then $3\mu^2 \leq |\chi|^2$. Combining this inequality with (\ref{Rayleighestim_1}) of Proposition \ref{prop:Rayleighestim}, we observe that 
    \begin{align}\label{eqn:largechi_l2l2_achi}
        C \leq \mathcal{R}_\chi (\vect u), \quad \forall \vect u \in H_\#^1(Y;\C^3) \setminus \{ 0 \},
    \end{align}
    where $C>0$ depends only on $c_\text{rayleigh}$ and $\mu$. Assuming further that $\eps$ is small enough such that
    \begin{align}\label{eqn:largechi_l2l2_eps}
        0 < \eps < 3\mu^2 \leq |\chi|^2,
    \end{align}
    we obtain the following resolvent estimate for $\mathcal{A}_\chi$:
    \begin{align}\label{eqn:l2l2_first3_evalues}
        \left\| P_\chi \left( \frac{1}{\eps^{\gamma+2}} \mathcal{A}_\chi + I \right)^{-1} P_\chi \right\|_{L^2(Y;\C^3) \rightarrow L^2(Y;\C^3)} \leq C \eps^{\gamma+2}.
    \end{align}
    In a similar fashion, we can combine (\ref{eqn:largechi_l2l2_eps}) with Proposition \ref{prop:hom_matrix} to obtain
    \begin{align}\label{eqn:largechi_ahomchi}
        \left\| \left( \frac{1}{\eps^{\gamma+2}} \mathcal{A}_\chi^{\text{hom}} + I_{\C^3} \right)^{-1} S \right\|_{L^2(Y;\C^3) \rightarrow L^2(Y;\C^3)} \leq C \eps^{\gamma+2},
    \end{align}
    where the constant $C>0$ depends only on $\nu_1$ (from Proposition \ref{prop:hom_matrix}). This gives us an estimate of
    \begin{align}\label{eqn:largechi_l2l2_final}
        \left\Vert P_\chi\left( \frac{1}{\varepsilon^{\gamma + 2}}\mathcal{A}_\chi + I\right)^{-1} P_\chi -\left( \frac{1}{\varepsilon^{\gamma + 2}}\mathcal{A}_\chi^{\rm hom} + I_{\C^3}\right)^{-1} S \right\Vert_{L^2(Y;\C^3) \to L^2(Y;\C^3)} \leq C \eps^{\gamma+2}.
    \end{align}

    \textbf{Conclusion of Step 1.} By combining (\eqref{eqn:l2l2_other_evalues}), (\eqref{eqn:smallchi_l2l2_final}), and (\eqref{eqn:largechi_l2l2_final}), we obtain the overall estimate
    \begin{align}
        \left\Vert \left( \frac{1}{\varepsilon^{\gamma + 2}}\mathcal{A}_\chi + I\right)^{-1} - \left( \frac{1}{\varepsilon^{\gamma + 2}}\mathcal{A}_\chi^{\rm hom} + I_{\C^3}\right)^{-1} S \right\Vert_{L^2(Y;\C^3) \to L^2(Y;\C^3)} \leq C \eps^{\frac{\gamma+2}{2}},
    \end{align}
    where the constant $C>0$ does not depend on $\chi$ and $\eps$.

    \subsubsection*{Step 2: Back to estimates on \texorpdfstring{$L^2(\R^3;\R^3)$}{L2(R3;R3)}.}

    \noindent
    To pass the estimates from the unit cell $Y$ back to the full space $\R^3$, we recall from Proposition \ref{prop:pass_to_unitcell} and Proposition \ref{prop:pass_to_unitcell2} that
    \begin{align*}
        \left( \frac{1}{\eps^{\gamma}} \mathcal{A}_\eps + I \right)^{-1} 
        &= \mathcal{G}_\eps^* \left( \int_{Y'}^{\oplus} \left( \frac{1}{\eps^{\gamma+2}} \mathcal{A}_\chi + I \right)^{-1} d \chi  \right) \mathcal{G}_\eps, \\
        \left( \frac{1}{\eps^{\gamma}} \mathcal{A}^{\text{hom}} + I \right)^{-1} \Xi_\eps 
        &= \mathcal{G}_\eps^* \left( \int_{Y'}^{\oplus} \left( \frac{1}{\eps^{\gamma+2}} \mathcal{A}_{\chi}^{\text{hom}} + I_{\C^3} \right)^{-1} S d\chi \right) \mathcal{G}_\eps.
    \end{align*}
    Since the Gelfand transform is a unitary operator from $L^2(\R^3;\C^3)$ to $L^2(Y';L^2(Y;\C^3))$, this implies that
    \begin{align*}
        \left\| \left( \frac{1}{\varepsilon^{\gamma }}\mathcal{A}_\varepsilon + I\right)^{-1} -\left( \frac{1}{\varepsilon^{\gamma}}\mathcal{A}^{\rm hom} + I\right)^{-1} \Xi_\varepsilon\right\|_{L^2(\R^3;\C^3) \to L^2(\R^3;\C^3)}  
        \leq C \varepsilon^\frac{\gamma + 2}{2}, 
    \end{align*}
    which in turn implies the $L^2(\R^3;\R^3)\rightarrow L^2(\R^3;\R^3)$ norm estimate (\ref{eqn:l2l2}). This completes the proof.
\end{proof}

\subsection{Definition of rescaled and full-space correctors}\label{sect:corrector_ops}

To obtain $L^2 \rightarrow H^1$ and higher-order $L^2 \rightarrow L^2$ norm-resolvent estimates, we will rely on (\ref{eqn:tau_resolvent_est2}) of Theorem \ref{thm:tau_resolvent_est}. This requires us to introduce rescaled and full-space analogues of the corrector operators, which we will do in this section. \textbf{In this section, we will assume that $z \in \rho ( \frac{1}{|\chi|^2} \mathcal{A}_\chi^{\rm hom} ) \cap \rho ( \frac{1}{|\chi|^2} \mathcal{A}_\chi )$ and $\chi \in Y' \setminus \{ 0 \}$, unless stated otherwise.}

\subsubsection*{First-order corrector operator}

Let us observe that the asymptotic procedure of Section \ref{sect:the_asymp_method} begins by fixing $\vect f \in L^2(Y;\C^3)$ in the resolvent equation $(\frac{1}{|\chi|^2} \mathcal{A}_\chi - zI ) \vect u = \vect f$. This $\vect f$ determines the constant $\vect u_0 \in \C^3$ through (\ref{leadingorderterm}). That is,
\begin{align*}
    \vect u_0 = \left( \frac{1}{|\chi|^2} \mathcal{A}_\chi^{\text{hom}} - z I_{\C^3} \right)^{-1} S \vect f.
\end{align*}
By Definition \ref{defn:correctors}, the first-order corrector operator $\mathcal{R}_{\rm corr,1,\chi}(z)$ takes $\vect f$ to $\vect u_0$, and then to the solution $\vect u_1$ of the problem \eqref{corr2}. This suggests that $\mathcal{R}_{\rm corr,1,\chi}(z)$ may be expressed in the following way:
\begin{align}\label{eqn:R1corr_z}
    \mathcal{R}_{\rm corr,1,\chi}(z) = \mathcal{B}_{\text{corr},1,\chi} \left( \frac{1}{|\chi|^2} \mathcal{A}_\chi^{\rm hom} - zI_{\C^3} \right)^{-1} S,
\end{align}
where $\mathcal{B}_{\rm corr,1,\chi}$ is defined as follows:
\begin{definition}
    For all $\chi \in Y'$, let $\mathcal{B}_{\text{corr}, 1, \chi} : \C^3 \rightarrow L^2(Y;\C^3)$ be the bounded operator defined by 
    \begin{equation}
        \mathcal{B}_{\text{corr}, 1, \chi} \vect c := \vect u_1,
    \end{equation}
    where $\vect u_1 \in H_\#^1(Y;\C^3)$ solves the following ($\chi$-dependent) problem (compare this with (\ref{corr2}))
    \begin{equation}\label{eqn:B1corr}
        \begin{cases}
             \int_{Y} \A  (\simgrad   \vect u_1 + iX_\chi \vect c) : \overline{ \simgrad   \vect v} = 0, \quad \forall \vect v \in H^{1}_\#(Y;\C^3), \\
             \int_Y \vect u_1 = 0. 
        \end{cases}
    \end{equation}\qedhere
\end{definition}

\begin{lemma}\label{lem:bcorr1chi_bound}
    $\| \mathcal{B}_{\text{corr},1,\chi} \|_{L^2\rightarrow H^1} \leq C |\chi|$. 
\end{lemma}

\begin{proof}
    By testing \eqref{eqn:B1corr} with $\vect u_1$ itself, and using Assumption \ref{assump:elliptic} and \eqref{eqn:Xchi_est}, we have
    \begin{align}
        \nu \| \simgrad \vect u_1 \|_{L^2(Y;\C^{3\times 3})}^2 \leq C_{\text{symrk1}} |\chi| \, |\vect c| \, \| \simgrad \vect u_1 \|_{L^2(Y;\C^{3\times 3})}.
    \end{align}
    This, combined with \eqref{korninequality33} (as $\int \vect u_1 = 0$) completes the proof.
\end{proof}

We would like to relate the scaling factor of $\frac{1}{|\chi|^2}$ in \eqref{eqn:R1corr_z} back to a scaling in $\eps$, by means of the Dunford-Riesz calculus (see (\ref{eqn:cauchy_int_formula})). Recall the contour $\Gamma$, and function $g_{\eps,\chi}$ defined at the start of Section \ref{sect:norm_resolvent_est}.

\begin{definition}[Rescaled first-order corrector]
    For $\chi \in Y' \setminus \{0\}$, $\eps>0$, and $\gamma > -2$, the rescaled first-order corrector operator $\mathcal{R}_{\rm corr,1,\chi}^\eps : L^2(Y;\C^3) \rightarrow L^2(Y;\C^3)$ is defined by
    \begin{align}\label{eqn:R1corr_fctcalc}
        \mathcal{R}_{\rm corr,1,\chi}^\eps := \mathcal{B}_{\text{corr},1,\chi} \left( \frac{1}{\eps^{\gamma+2}} \mathcal{A}_\chi^{\text{hom}} + I_{\C^3} \right)^{-1} S.
    \end{align}
\end{definition}

\begin{remark}\label{rmk:first_corr_smallchi}
    If $\chi \in [-\mu,\mu]^3 \setminus \{0\}$, then we may express $\mathcal{R}_{\text{corr},1,\chi}^\eps$ in contour integral form
    \begin{align}
        \begin{split}
        &\mathcal{R}_{\rm corr,1,\chi}^\eps
        \stackrel{\text{\eqref{eqn:contour_calc2}}}{=} - \frac{1}{2\pi i} \mathcal{B}_{\text{corr},1,\chi} \oint_{\Gamma} g_{\eps,\chi}(z) \left( \frac{1}{|\chi|^2} \mathcal{A}_\chi^{\text{hom}} - zI_{\C^3} \right)^{-1} S dz
        = - \frac{1}{2\pi i} \oint_{\Gamma} g_{\eps,\chi}(z) \mathcal{R}_{\rm corr,1,\chi}(z) dz.
        \end{split}\qedhere
    \end{align}
\end{remark}

To bring the discussion from $L^2(Y;\C^3)$ back to the full space $L^2(\R^3;\C^3)$, we introduce:

\begin{definition}
    For $\eps>0$, the full-space first-order corrector operator $\mathcal{R}_{\rm corr, 1}^\eps: L^2(\R^3;\C^3) \to L^2(\R^3;\C^3)$ is given by
    \begin{equation}\label{trans_back_corr1}
        \mathcal{R}_{\text{corr}, 1}^\eps = \mathcal{G}_\eps^\ast \left( \int_{Y'}^{\oplus} \mathcal{R}_{\text{corr},1,\chi}^\eps d\chi \right) \mathcal{G}_\eps.
    \end{equation}
\end{definition}

\subsubsection*{Second-order corrector operator}

Recall from Definition \ref{defn:correctors} that $\mathcal{R}_{\rm corr,2,\chi}(z) : L^2(Y;\C^3) \to \C^3 \hookrightarrow L^2(Y;\C^3)$ takes $\vect f$ to $\vect u_0^{(1)}$ (the solution to \eqref{newconstantcorrector}). We now introduce rescaled and full-space versions of $\mathcal{R}_{\rm corr,2,\chi}(z)$.

\begin{definition}[Rescaled second-order corrector] 
    For $\chi \in Y' \setminus \{ 0 \}$, $\eps>0$, and $\gamma > -2$, the rescaled second-order corrector operator $\mathcal{R}_{\rm corr,2,\chi}^\eps : L^2(Y;\C^3) \rightarrow L^2(Y;\C^3)$ is defined by
    \begin{align}
    \label{eqn:corr2_rescaled}
        \mathcal{R}_{\rm corr,2,\chi}^\varepsilon :=
        \left\{ \begin{array}{cc}
             - \frac{1}{2\pi i} \oint_{\Gamma} g_{\eps,\chi}(z) \mathcal{R}_{\rm corr,2,\chi}(z) dz, &  \text{ if } \chi \in [-\mu,\mu]^3 \setminus \{0\} \\
            0 & \text{ otherwise. }
        \end{array}     \right.
    \end{align}
\end{definition}

\begin{definition}
    For $\eps>0$, the full-space second-order corrector operator $\mathcal{R}_{\text{corr},2}^{\eps}: L^2(\R^3;\C^3) \to L^2(\R^3;\C^3)$ is given by
        \begin{equation}\label{trans_back_corr2}
            \mathcal{R}_{\text{corr}, 2}^\eps = \mathcal{G}_\eps^\ast \left( \int_{Y'}^{\oplus} \mathcal{R}_{\text{corr}, 2,\chi}^\eps d\chi \right) \mathcal{G}_\eps.
        \end{equation}
\end{definition}

\subsection{\texorpdfstring{$L^2 \to H^1$}{L2 to H1} estimates}

In this section we will prove an $L^2 \to H^1$ estimate. The proof builds on the proof of Theorem \ref{thm:l2l2}.

\begin{theorem}\label{thm:l2h1}
    There exists a constant $C>0$, independent of $\eps>0$, such that, the following norm-resolvent estimate holds for all $\gamma > -1$:
    \begin{equation}\label{eqn:l2h1}
        \left\Vert \left( \frac{1}{\eps^{\gamma }}\mathcal{A}_\eps + I\right)^{-1} -\left( \frac{1}{\eps^{\gamma}}\mathcal{A}^{\rm hom} + I\right)^{-1} \Xi_\eps -  \mathcal{R}_{\rm corr, 1}^\eps   \right\Vert_{L^2(\R^3;\R^3) \to H^1(\R^3;\R^3)}  
     \leq C \max \left\{ \eps^{\gamma + 1}, \eps^\frac{\gamma + 2}{2} \right\}, 
    \end{equation}
    where the homogenized operator $\mathcal{A}^{\rm hom}$ is given by Definition \ref{defn:hom_operator_fullspace}, $\mathcal{A}_\varepsilon$ is given by Definition \ref{defn:main_operator_fullspace}, the smoothing operator $\Xi_\varepsilon$ is given by Definition \ref{defn:smoothing_op}, and the corrector operator $\mathcal{R}_{\rm corr, 1}^\eps$ is defined by \eqref{trans_back_corr1}.
\end{theorem}

\begin{proof}
    All the operators involved (e.g. $(\eps^{-\gamma} \mathcal{A}_\eps+I)^{-1}$ on $L^2(\R^3;\C^3)$, or $\mathcal{R}_{\text{corr},1,\chi}$ on $L^2(Y;\C^3)$), have their images contained in $H^1$ on their respective spaces. Thus it makes sense to look at the $L^2\rightarrow H^1$ norm. We will split our discussion of the $H^1$ norm into the $L^2$ norm and the $L^2$ norm of the gradients.
    
    \subsubsection*{Step 1: Estimates on $L^2(Y;\C^3)$.}
    
    \textbf{$L^2$ norm.} We split into two further cases:
    
    \textbf{Case 1: $\chi \in [-\mu,\mu]^3 \setminus \{ 0 \}$.} By using \eqref{eqn:tau_resolvent_est2} of Theorem \ref{thm:tau_resolvent_est} and \eqref{bound21}, we have that for $z \in \Gamma$,
    \begin{align}
        &\left\| \left( \frac{1}{|\chi|^2} \mathcal{A}_\chi - zI \right)^{-1} 
        - \left( \frac{1}{|\chi|^2} \mathcal{A}_\chi^{\text{hom}} - zI_{\C^3} \right)^{-1} S
        - \mathcal{R}_{\text{corr},1,\chi}(z)
        \right\|_{L^2\rightarrow L^2} \nonumber\\
        &\leq \left\| \left( \frac{1}{|\chi|^2} \mathcal{A}_\chi - zI \right)^{-1} 
        - \left( \frac{1}{|\chi|^2} \mathcal{A}_\chi^{\text{hom}} - zI_{\C^3} \right)^{-1} S
        - \mathcal{R}_{\text{corr},1,\chi}(z) 
        - \mathcal{R}_{\text{corr},2,\chi}(z)
        \right\|_{L^2\rightarrow L^2} + \| \mathcal{R}_{\text{corr},2,\chi}(z) \|_{L^2\rightarrow L^2} \nonumber\\
        &\leq C |\chi|^2 + C |\chi| \leq C |\chi|. \label{eqn:l2bit_est}
    \end{align}
    This implies that we may modify the estimates in \eqref{eqn:smallchi_l2l2_final} to give us
    \begin{align}
        \Bigg\| \underbrace{\left( \frac{1}{\eps^{\gamma + 2}} \mathcal{A}_\chi + I \right)^{-1}}_{=: \mathbb{X}} 
        - \underbrace{\left( \frac{1}{\eps^{\gamma + 2}} \mathcal{A}_\chi^{\text{hom}} + I_{\C^3} \right)^{-1} S}_{=: \mathbb{Y}}
        - \underbrace{ \phantom{\bigg(} \mathcal{R}_{\text{corr},1,\chi}^\eps \phantom{\bigg)} }_{=: \mathbb{Z}}
        \Bigg\|_{L^2\rightarrow L^2} \leq C \eps^{\frac{\gamma + 2}{2}}. \label{eqn:three_terms}
    \end{align}
    Here, the projections $P_\chi$ were dropped from $(\eps^{-\gamma - 2} \mathcal{A}_\chi + I )^{-1}$ by using \eqref{eqn:l2l2_other_evalues}.

    \textbf{Case 2: $\chi \in Y' \setminus [-\mu,\mu]^3$.} Of the three terms in the left-hand side of \eqref{eqn:three_terms}, we know that the first two terms are norm bounded by $C \eps^{\gamma+2}$, due to \eqref{eqn:l2l2_other_evalues} and \eqref{eqn:largechi_l2l2_final}. We claim that the same bound holds for $\mathcal{R}_{\text{corr},1,\chi}^{\eps}$. Indeed, by recalling \eqref{eqn:R1corr_fctcalc}, the claim now follows from \eqref{eqn:largechi_ahomchi} and Lemma \ref{lem:bcorr1chi_bound}. In summary,
    \begin{align}\label{eqn:three_terms_bigchi}
        \| \mathbb{X} - \mathbb{Y} - \mathbb{Z} \|_{L^2\to L^2} \leq C \eps^{\gamma + 2}.
    \end{align}

    \noindent
    \textbf{$L^2$ norm of the gradient.} Similarly, we have two further cases:

    \textbf{Case 1: $\chi \in [-\mu,\mu]^3 \setminus \{ 0 \}$.} Let $\vect f \in L^2(Y;\C^3)$. Then,
    \begin{align}
        &\left\| \nabla P_\chi \left( \frac{1}{\eps^{\gamma + 2}} \mathcal{A}_\chi + I \right)^{-1} P_\chi \vect f
        - \nabla \left( \frac{1}{\eps^{\gamma + 2}} \mathcal{A}_\chi^{\text{hom}} + I_{\C^3} \right)^{-1} S \vect f
        - \nabla \mathcal{R}_{\text{corr},1,\chi}^\eps \vect f \right\|_{L^2(Y;\C^{3\times 3})} \\
        &= \frac{1}{2\pi} \left\| \nabla \oint_\Gamma g_{\eps,\chi}(z) \left[ \left( \frac{1}{|\chi|^2} \mathcal{A}_\chi - zI \right)^{-1} \vect f 
        - \left( \frac{1}{|\chi|^2} \mathcal{A}_\chi^{\text{hom}} 
        - zI_{\C^3} \right)^{-1}S \vect f 
        - \mathcal{R}_{\text{corr},1,\chi}(z) \right] dz \right\|_{L^2}    \\
        &\leq \frac{1}{2\pi} \oint_{\Gamma} |g_{\eps,\chi}(z)| \left\|  \nabla \left( \frac{1}{|\chi|^2} \mathcal{A}_\chi - zI \right)^{-1} \vect f 
        - \nabla \left( \frac{1}{|\chi|^2} \mathcal{A}_\chi^{\text{hom}} 
        - zI_{\C^3} \right)^{-1}S \vect f
        - \nabla \mathcal{R}_{\text{corr},1,\chi}(z) \vect f \right\|_{L^2} dz \\ 
        &\qquad\qquad\qquad \text{Swap $\nabla$ and integral (by noting that the integral is a finite sum).} \nonumber \\
        &= \frac{1}{2\pi} \oint_{\Gamma} |g_{\eps,\chi}(z)| \left\|  \nabla \left( \frac{1}{|\chi|^2} \mathcal{A}_\chi - zI \right)^{-1} \vect f 
        - \nabla \left( \frac{1}{|\chi|^2} \mathcal{A}_\chi^{\text{hom}} 
        - zI_{\C^3} \right)^{-1}S \vect f
        - \nabla \mathcal{R}_{\text{corr},1,\chi}(z) \vect f \right. \nonumber\\
        &\qquad\qquad\qquad\qquad\qquad - \left. \phantom{\bigg(} \nabla \mathcal{R}_{\text{corr},2,\chi}(z) \vect f \right\|_{L^2} dz \\
        &\qquad\qquad\qquad \text{As $\mathcal{R}_{\text{corr},2,\chi}(z) \vect f$ is a constant vector.} \nonumber \\
        &\leq C \left( \max \left\{ \frac{|\chi|^2}{\eps^{\gamma + 2}}, 1 \right\} \right)^{-1} |\chi|^2 \, \| \vect f \|_{L^2} \\
        &\qquad\qquad\qquad \text{By Lemma \ref{lem:g_str} and Theorem \ref{thm:tau_resolvent_est}, since $\chi \neq 0$.} \nonumber \\
        &\leq C \eps^{\gamma + 2} \| \vect f \|_{L^2}. \label{eqn:h1bit_smallchi1}
    \end{align}

    We now look at the term $\nabla ( \eps^{-\gamma - 2} \mathcal{A}_\chi + I )^{-1}(I - P_\chi) \vect f$. Writing $\vect u = ( \eps^{-\gamma - 2} \mathcal{A}_\chi + I )^{-1} \vect g$ for $\vect g \in L^2$, we have that $\vect u \in H_{\#}^1$ solves the following problem
    \begin{align}
        \frac{1}{\eps^{\gamma + 2}} \int_Y \mathbb{A} (\simgrad + i X_\chi) \vect u : \overline{(\simgrad + i X_\chi) \vect v} + \int_Y \vect u \cdot \overline{\vect v} = \int_Y \vect g \cdot \overline{\vect v}, \quad \forall \vect v \in H_{\#}^1(Y;\C^3).
    \end{align}
    By testing with $\vect u$ itself, we have
    \begin{align}
        \frac{\nu}{\eps^{\gamma + 2}} \| (\simgrad + iX_\chi) \vect u \|_{L^2}^2 + \| \vect u \|_{L^2}^2 \leq \| \vect g \|_{L^2} \| \vect u \|_{L^2}.
    \end{align}
    Combining this inequality with \eqref{estimate11}, we arrive at 
    \begin{align}\label{eqn:l2h1_gradient_est}
        \| \nabla \vect u \|_{L^2}^2 \leq C \eps^{\gamma + 2} \left( \| \vect g \|_{L^2} \| \vect u \|_{L^2} - \| \vect u \|_{L^2}^2 \right).
    \end{align}
    In the present case, we have $\vect g = (I - P_\chi)\vect f$. This allows us to further estimate $\| \vect u \|_{L^2}$ in terms of $\| \vect f \|_{L^2}$ by employing \eqref{eqn:l2l2_other_evalues}, which gives us
    \begin{align}
        \| \nabla \vect u \|_{L^2}^2 
        \leq C \eps^{\gamma + 2} \left( \widetilde{C} \eps^{\gamma + 2} \| \vect f \|_{L^2}^2 - \widetilde{C}^2 \eps^{2(\gamma + 2)} \| \vect f \|_{L^2}^2  \right)
        \leq C \eps^{2(\gamma + 2)} \| \vect f \|_{L^2}^2,
    \end{align}
    assuming $\eps$ is small enough. That is,
    \begin{align}
        \left\| \nabla \left( \frac{1}{\eps^{\gamma + 2}} \mathcal{A}_\chi + I \right)^{-1} (I - P_\chi) \vect f \right\|_{L^2} \leq C \eps^{\gamma + 2} \| \vect f \|_{L^2}. \label{eqn:h1bit_smallchi2}
    \end{align}
    To summarize the present case, we combine \eqref{eqn:h1bit_smallchi1} and \eqref{eqn:h1bit_smallchi2} to get
    \begin{align}
        \| \nabla [\mathbb{X} - \mathbb{Y} - \mathbb{Z}] \vect f \|_{L^2} \leq C \eps^{\gamma + 2} \| \vect f \|_{L^2}.
    \end{align}

    \textbf{Case 2: $\chi \in Y' \setminus [-\mu,\mu]^3$.} Let $\vect f \in L^2(Y;\C^3)$. Similarly to the proof of Theorem \ref{thm:l2l2}, we estimate the three terms $\mathbb{X}$, $\mathbb{Y}$, and $\mathbb{Z}$ separately. We begin with \textbf{the term $\mathbb{Y}$}. Since $\mathcal{A}_\chi^{\text{hom}}:\C^3\rightarrow \C^3$, we have
    \begin{align}\label{eqn:grad_largechi_Y}
        \left\| \nabla \left( \frac{1}{\eps^{\gamma + 2}} \mathcal{A}_\chi^{\text{hom}} + I_{\C^3} \right)^{-1} S\vect f \right\|_{L^2(Y;\C^{3\times 3})} = 0.
    \end{align}
    Next, \textbf{we look at $\mathbb{Z}$}. By combining Lemma \ref{lem:bcorr1chi_bound} with \eqref{eqn:largechi_ahomchi}, we have
    \begin{align}\label{eqn:grad_largechi_Z}
        \left\| \nabla \mathcal{B}_{\text{corr},1,\chi} \left( \frac{1}{\eps^{\gamma + 2}} \mathcal{A}_\chi^{\text{hom}} + I_{\C^3} \right)^{-1} S \vect f  \right\|_{L^2}
        \leq C \left\| \left( \frac{1}{\eps^{\gamma + 2}} \mathcal{A}_\chi^{\text{hom}} + I_{\C^3} \right)^{-1} S \vect f \right\|_{L^2}
        \leq C \eps^{\gamma + 2} \| \vect f \|_{L^2}.
    \end{align}
    Finally, \textbf{we look at $\mathbb{X}$}. Write $\vect u = ( \eps^{-\gamma - 2} \mathcal{A}_\chi + I )^{-1} \vect f$. As pointed out in \eqref{eqn:l2h1_gradient_est} of Case 1, we know that 
    \begin{align}
        \| \nabla \vect u \|_{L^2}^2 \leq C \eps^{\gamma + 2} \left( \| \vect f \|_{L^2} \| \vect u \|_{L^2} - \| \vect u \|_{L^2}^2 \right).
    \end{align}
    In the present case, we have $\chi \in Y' \setminus [-\mu,\mu]^3$. This allows to further estimate $\| \vect u \|_{L^2}$ in terms of $\| \vect f \|_{L^2}$ using \eqref{eqn:l2l2_other_evalues} and \eqref{eqn:l2l2_first3_evalues} to get
    \begin{align}
        \| \nabla \vect u \|_{L^2}^2 
        \leq C \eps^{\gamma + 2} \left( \widetilde{C} \eps^{\gamma + 2} \| \vect f \|_{L^2}^2 - \widetilde{C}^2 \eps^{2(\gamma + 2)} \| \vect f \|_{L^2}^2  \right)
        \leq C \eps^{2(\gamma + 2)} \| \vect f \|_{L^2}^2.
    \end{align}
    assuming $\eps$ is small enough. That is,
    \begin{align}\label{eqn:grad_largechi_X}
        \left\| \nabla \left( \frac{1}{\eps^{\gamma + 2}} \mathcal{A}_\chi + I \right)^{-1} \vect f \right\|_{L^2} \leq C \eps^{\gamma + 2} \| \vect f \|_{L^2}.
    \end{align}
    Combining \eqref{eqn:grad_largechi_Y}, \eqref{eqn:grad_largechi_Z}, and \eqref{eqn:grad_largechi_X}, we conclude that
    \begin{align}
        \| \nabla [\mathbb{X} - \mathbb{Y} - \mathbb{Z}] \vect f \|_{L^2} \leq C \eps^{\gamma + 2} \| \vect f \|_{L^2}.
    \end{align}
    
    \subsubsection*{Step 2: Back to estimates on $L^2(\R^3;\C^3)$.}
    
    \textbf{$L^2$ norm.} The argument here is similar to Step 2 of the proof of Theorem \ref{thm:l2l2}. In Step 1, we have shown that the $L^2(Y;\C^3)$ norm contributes to an overall $\mathcal{O}(\eps^{\frac{\gamma + 2}{2}})$ estimate. Since the Gelfand transform is unitary from $L^2(\R^3;\C^3)$ to $L^2(Y';L^2(Y;\C^3))$, the same $\mathcal{O}(\eps^{\frac{\gamma + 2}{2}})$ estimate holds on $L^2(\R^3;\C^3)$.

    \noindent
    \textbf{$L^2$ norm of the gradient.} In Step 1, we have shown that for $\vect g \in L^2(Y\times Y';\C^3)$,
    \begin{align}\label{eqn:h1bit_step1}
        &\left\| \left( \nabla_y \left( \frac{1}{\eps^{\gamma+2}} \mathcal{A}_\chi + I \right)^{-1} \vect g \right) (\cdot, \chi) - \left( \nabla_y \left( \frac{1}{\eps^{\gamma+2}} \mathcal{A}_\chi^{\text{hom}} + I_{\C^3} \right)^{-1}S \vect g \right) (\cdot, \chi) - \left( \nabla_y \mathcal{R}_{\text{corr},1,\chi} \vect g\right) (\cdot, \chi) \right\|_{L^2(Y;\C^{3\times 3})} \nonumber\\
        &\leq C\eps^{\gamma + 2 } \| \vect g(\cdot, \chi)\|_{L^2(Y;\C^3)},
    \end{align}
    for almost every $\chi \neq 0$. Here, $g$ is a function of $y \in Y$ and $\chi \in Y'$, so that $\nabla_y$ denotes the gradient with respect to the $y$-variable.

    Let $\vect f \in L^2(\R^3;\C^3)$. Our goal now is to prove the following estimate
    \begin{align}\label{eqn:h1bit_claim}
        \bigg\| \nabla \underbrace{\left( \frac{1}{\eps^{\gamma}}\mathcal{A}_\eps + I \right)^{-1}}_{=: \mathcal{X}} \vect f 
        - \nabla \underbrace{\left( \frac{1}{\eps^{\gamma}} \mathcal{A}^{\text{hom}} + I_{\C^3} \right)^{-1} \Xi_\eps}_{=: \mathcal{Y}} \vect f 
        - \nabla \underbrace{\mathcal{R}_{\text{corr},1}^\eps}_{=: \mathcal{Z}} \vect f \bigg\|_{L^2(\R^3;\C^{3\times 3})}
        \leq C \eps^{\gamma + 1} \| \vect f \|_{L^2(\R^3;\R^3)}.
    \end{align}

    To begin, we apply the Gelfand transform $\mathcal{G}_\eps$ to the $L^2$ function $\nabla (\mathcal{X} - \mathcal{Y} - \mathcal{Z})\vect f$. As $\mathcal{G}_\eps$ is unitary, the left-hand side of \eqref{eqn:h1bit_claim} now becomes
    \begin{align}
        &\| \mathcal{G}_\eps \left[ \nabla (\mathcal{X} - \mathcal{Y} - \mathcal{Z}) \vect f \right] \|_{L^2(Y\times Y' ; \C^{3\times 3})} \nonumber\\
        &\stackrel{\text{\eqref{eqn:scaling_vs_deriv}}}{\leq} \frac{1}{\eps} \left\| \nabla_y (\mathcal{G}_\eps [\mathcal{X} - \mathcal{Y} - \mathcal{Z}] \vect f ) \right\|_{L^2(Y\times Y';\C^{3\times 3})} 
        + \frac{1}{\eps} \left\| \left( \mathcal{G}_\eps [\mathcal{X} - \mathcal{Y} - \mathcal{Z}] \vect f \right) \otimes \chi \right\|_{L^2(Y\times Y';\C^{3\times 3})}. \label{eqn:h1bit_claim1_a}
    \end{align}

    We investigate the two terms in \eqref{eqn:h1bit_claim1_a} separately.
    
    \textbf{The first term, $\nabla_y (\mathcal{G}_\eps [\mathcal{X} \vect f - \mathcal{Y} \vect f - \mathcal{Z} \vect f] )$.} Let us note that by applying $\nabla_y \mathcal{G}_\eps$ on the left, on both sides of \eqref{eqn:vonneumannformula_resolvent}, we have
    \begin{align}\label{eqn:h1bit_fullspace}
        \nabla_y \mathcal{G}_\eps \mathcal{X} \vect f 
        = \nabla_y \mathcal{G}_\eps \left( \frac{1}{\eps^{\gamma}} \mathcal{A}_\eps + I \right)^{-1} \vect f 
        = \nabla_y \left( \int_{Y'}^{\oplus} \left( \frac{1}{\eps^{\gamma + 2}} \mathcal{A}_\chi + I \right)^{-1} d\chi \right) \mathcal{G}_\eps \vect f.
    \end{align}
    Similar statements hold for $\mathcal{Y}$ and $\mathcal{Z}$. We now use the result \eqref{eqn:h1bit_step1} from Step 1 as follows:
    \begin{alignat}{2}
        &\frac{1}{\eps^2} \| \nabla_y \mathcal{G}_\eps [\mathcal{X} - \mathcal{Y} - \mathcal{Z}] \vect f \|_{L^2(Y\times Y';\C^{3\times 3})}^2
        \span \span = \frac{1}{\eps^2} \int_{Y'} \left\| \bigg(\nabla_y \mathcal{G}_\eps [\mathcal{X} - \mathcal{Y} - \mathcal{Z}] \vect f \bigg) (\cdot,\chi) \right\|_{L^2(Y;\C^{3\times 3})}^2 d\chi \nonumber\\
        &\leq \frac{1}{\eps^2} C \eps^{2(\gamma + 2)} \int_{Y'} \| \mathcal{G}_\eps \vect f(\cdot, \chi) \|_{L^2(Y;\C^3)}^2 d\chi 
        &&\qquad\qquad \text{By \eqref{eqn:h1bit_fullspace} and \eqref{eqn:h1bit_step1}.} \nonumber\\
        &= C \eps^{2(\gamma + 1)} \| \mathcal{G}_\eps \vect f \|_{L^2(Y\times Y';\C^3)}^2 \nonumber\\
        &= C \eps^{2(\gamma + 1)} \| \vect f \|_{L^2(\R^3;\C^3)}^2
        &&\qquad\qquad \text{As $\mathcal{G}_\eps$ is unitary.}
    \end{alignat}

    \textbf{The second term, $\left( \mathcal{G}_\eps [\mathcal{X} - \mathcal{Y} - \mathcal{Z}] \vect f \right) \otimes \chi$.} We have:
    \begin{align}\label{eqn:h1bit_claim1_b}
        \frac{1}{\eps} \| \left( \mathcal{G}_\eps [\mathcal{X} - \mathcal{Y} - \mathcal{Z}] \vect f \right) \otimes \chi \|_{L^2(Y\times Y';\C^{3\times 3})}
        = \left\| \mathcal{G}_\eps [\mathcal{X} - \mathcal{Y} - \mathcal{Z}] \vect f \otimes \frac{\chi}{\eps} \right\|_{L^2(Y\times Y' ; \C^{3\times 3})}.
    \end{align}

    Once again, by applying $\mathcal{G}_\eps$ on the left, on both sides of \eqref{eqn:vonneumannformula_resolvent}, we have
    \begin{align}\label{eqn:h1bit_fullspace2}
        \mathcal{G}_\eps \mathcal{X} \vect f 
        = \mathcal{G}_\eps \left( \frac{1}{\eps^{\gamma}} \mathcal{A}_\eps + I \right)^{-1} \vect f 
        = \left( \int_{Y'}^{\oplus} \left( \frac{1}{\eps^{\gamma + 2}} \mathcal{A}_\chi + I \right)^{-1} d\chi \right) \mathcal{G}_\eps \vect f.
    \end{align}
    Similar statements hold for $\mathcal{Y}$ and $\mathcal{Z}$. Making use of the right-hand side of \eqref{eqn:h1bit_fullspace2}, we now argue that the term $\frac{\chi}{\eps}$ from \eqref{eqn:h1bit_claim1_b} improves the overall $L^2\to L^2$ estimate. Indeed,
    \begin{itemize}
        \item If $\chi \in [-\mu,\mu]^3 \setminus \{0\}$, then we continue from the proof of Theorem \ref{thm:l2l2}, using \eqref{eqn:l2bit_est} instead of \eqref{eqn:link_eps_and_chi}. Then, instead of \eqref{eqn:smallchi_l2l2_intermediate}, we now have the bound
        \begin{equation}\label{eqn:h1bit_upgrade1}
            C \left( \max \left\{ \frac{|\chi|^2}{\eps^{\gamma + 2}}, 1 \right\} \right)^{-1} |\chi| \frac{|\chi|}{\eps}
            = C \eps^{\gamma + 2} \eps^{-1}
            = C \eps^{\gamma + 1}.
        \end{equation}
        
        \item If $\chi \in Y'\setminus [\mu,\mu]^3$, then by \eqref{eqn:three_terms_bigchi} (which builds on the proof of Theorem \ref{thm:l2l2}), we have the bound
        \begin{equation}\label{eqn:h1bit_upgrade2}
            C \eps^{\gamma + 2} \frac{|\chi|}{\eps} = C \eps^{\gamma + 1}.
        \end{equation}
    \end{itemize}

    Then, continuing from \eqref{eqn:h1bit_claim1_b},
    \begin{alignat}{2}
        &\frac{1}{\eps} \| \left( \mathcal{G}_\eps [\mathcal{X} - \mathcal{Y} - \mathcal{Z}] \vect f \right) \otimes \chi \|_{L^2(Y\times Y';\C^{3\times 3})} \nonumber\\
        &\leq C \left\|\int_{Y'}^{\oplus} \left(\mathbb{X} - \mathbb{Y} - \mathbb{Z}\right) d\chi \otimes \frac{\chi}{\eps} \right\|_{L^2(Y\times Y';\C^3)\to L^2(Y\times Y';\C^{3\times 3})} \| \mathcal{G}_\eps \vect f \|_{L^2(Y\times Y';\C^3)}\qquad 
        &&\text{By \eqref{eqn:h1bit_fullspace2}.} \nonumber\\
        &\leq C \eps^{\gamma + 1} \| \mathcal{G}_\eps \vect f \|_{L^2(Y\times Y';\C^3)}
        &&\text{By \eqref{eqn:h1bit_upgrade1} and \eqref{eqn:h1bit_upgrade2}.} \nonumber\\
        &= C \eps^{\gamma + 1} \| \vect f \|_{L^2(\R^3;\C^3)}.
        &&\text{As $\mathcal{G}_\eps$ is unitary.}
    \end{alignat}
    We have completed our investigation of the two terms. Returning to \eqref{eqn:h1bit_claim1_a}, we have proved that 
    \begin{align}
        \| \mathcal{G}_\eps [\nabla(\mathcal{X} - \mathcal{Y} - \mathcal{Z}) \vect f] \|_{L^2(Y\times Y';\C^{3\times 3})} \leq C \eps^{\gamma + 1} \| \vect f\|_{L^2(\R^3;\C^3)},
    \end{align}
    and this proves our claim \eqref{eqn:h1bit_claim}.

    \textbf{Conclusion.} We have shown that the $L^2$ norm contributes to an $\mathcal{O}(\eps^{\frac{\gamma + 2}{2}})$ estimate, and $L^2$ norm of the gradients contributes to an $\mathcal{O}(\eps^{\gamma + 1})$ estimate. This concludes the proof of the theorem.
\end{proof}

\begin{remark}
    To obtain the $L^2 \rightarrow H^1$ norm-resolvent estimate of Theorem \ref{thm:l2h1}, we have used (a) the $L^2 \rightarrow L^2$ $\mathcal{O}(|\chi|)$ estimate (\ref{eqn:tau_resolvent_est1}) and (b) the $L^2 \rightarrow H^1$ $\mathcal{O}(|\chi|^2)$ estimate (\ref{eqn:tau_resolvent_est2}). We did not use the $L^2\rightarrow H^1$ $\mathcal{O}(|\chi|)$ estimate in (\ref{eqn:tau_resolvent_est1}). The reason for this boils down to the identity (\ref{scalingderivatives}), which implies for instance in the case $\gamma=0$, that we would only get an $\mathcal{O}(1)$ estimate on the gradients.
\end{remark}

\subsection{Higher-order \texorpdfstring{$L^2 \to L^2$}{L2 to L2} estimates}

In this section we discuss higher-order $L^2\rightarrow L^2$ estimates. Unlike Theorem \ref{thm:l2h1} ($L^2 \rightarrow H^1$), the proof here is a straightforward modification of the proof of Theorem \ref{thm:l2l2}. For this reason, we omit the proof.

\begin{theorem}\label{thm:l2l2_higherorder}
    There exists a constant $C>0$, independent of $\eps > 0$, such that, for all $\gamma > -2$, the following norm-resolvent estimate holds:
    \begin{equation}\label{eqn:l2l2_higherorder}
        \left\| \left( \frac{1}{\eps^{\gamma }}\mathcal{A}_\eps + I\right)^{-1} -\left( \frac{1}{\eps^{\gamma}}\mathcal{A}^{\rm hom} + I\right)^{-1}\Xi_{\eps} -  \mathcal{R}_{\rm corr, 1}^\eps - \mathcal{R}_{\rm corr, 2}^\eps  \right\|_{L^2(\R^3;\R^3) \to L^2(\R^3;\R^3)}  
        \leq C  \eps^{\gamma + 2}, 
    \end{equation}
    where the homogenized operator $\mathcal{A}^{\rm hom}$ is given by Definition \ref{defn:hom_operator_fullspace}, $\mathcal{A}_\varepsilon$ is given by Definition \ref{defn:main_operator_fullspace}, the smoothing operator $\Xi_\varepsilon$ is given by Definition \ref{defn:smoothing_op}, and the corrector operators $\mathcal{R}_{\rm corr, 1}^\eps$ and $\mathcal{R}_{\rm corr, 2}^\eps$ are defined by \eqref{trans_back_corr1} and \eqref{trans_back_corr2} respectively.
\end{theorem}
\begin{proof}
     We demonstrate only the case $\chi \in [-\mu,\mu]^3\setminus \{ 0 \}$, since the rest is identical to \eqref{eqn:l2l2_other_evalues}.  Consider the contour $\Gamma$ provided by the Lemma \ref{lemma:contour} and recall the contour integration formulas \eqref{eqn:contour_calc1} and \eqref{eqn:contour_calc2}.
     Furthermore, recall the definitions of the corrector operators \eqref{eqn:R1corr_fctcalc} and \eqref{eqn:corr2_rescaled}. With these at hand, the resolvent estimate \eqref{eqn:tau_resolvent_est2} yields
    \begin{align}
        &\left\| P_\chi\left(\frac{1}{\eps^{\gamma + 2}}\mathcal{A}_\chi + I\right)^{-1} P_\chi 
        -   \left(\frac{1}{\eps^{\gamma + 2}}\mathcal{A}^{\rm hom}_\chi + I_{\C^3}\right)^{-1} S
        - \mathcal{R}_{\text{corr},1,\chi}^\varepsilon(z) 
        - \mathcal{R}_{\text{corr},2,\chi}^\varepsilon(z)
        \right\|_{L^2(Y;\C^3) \to L^2(Y;\C^3)} \nonumber\\
        &\leq \frac{1}{2\pi} \oint_{\Gamma} | g_{\eps,\chi}(z)|\left\Vert \left( \frac{1}{|\chi|^{2}}\mathcal{A}_\chi - zI \right)^{-1} -\left( \frac{1}{|\chi|^{2}}\mathcal{A}_\chi^{\rm hom} - zI_{\C^3} \right)^{-1} S  - \mathcal{R}_{\text{corr},1,\chi}(z) 
        - \mathcal{R}_{\text{corr},2,\chi}(z)\right\Vert_{L^2 \to L^2} dz \nonumber\\
         &\quad \leq  C\left( \max\left\{\frac{|\chi|^{2}}{\eps^{\gamma + 2}}, 1\right\} \right)^{-1}  |\chi|^2 \\
        &\qquad\qquad\qquad \parbox{30em}{By Lemma \ref{lem:g_str} and Theorem \ref{thm:tau_resolvent_est}, since $\chi \neq 0$. \\ $C>0$ does not depend on $z$, because $\Gamma$ satisfies (\ref{eqn:contour_buffer}).} \nonumber \\
        &\quad \leq C \eps^{\gamma + 2}. 
    \end{align}
\end{proof}

\subsection{Dropping the smoothing operator}\label{sect:smoothingdrop}
In this section, we show that the smoothing operator $\Xi_\varepsilon$ can be omitted from the left-hand sides of \eqref{eqn:l2l2}, \eqref{eqn:l2h1}, and \eqref{eqn:l2l2_higherorder}, thereby giving the claimed results of Theorem \ref{thm:main_thm}.
\begin{theorem}\label{thm:smoothingdrop}
     There exists a constant $C>0$, independent of $\eps$, such that,  for all $\gamma > -2$ we have:
    \begin{align}
        \left\Vert \left( \frac{1}{\varepsilon^{\gamma}}\mathcal{A}^{\rm hom} + I\right)^{-1} \left( I- \Xi_\varepsilon \right)\right\Vert_{L^2(\R^3;\R^3) \to L^2(\R^3;\R^3)}  
        \leq C \varepsilon^{\gamma + 2}, \label{eqn:l2l2drop} \\
        \left\Vert \left( \frac{1}{\varepsilon^{\gamma}}\mathcal{A}^{\rm hom} + I\right)^{-1} \left( I- \Xi_\varepsilon \right)\right\Vert_{L^2(\R^3;\R^3) \to H^1(\R^3;\R^3)}  
        \leq C \varepsilon^{\gamma + 1}. \label{eqn:l2h1drop}
    \end{align}
\end{theorem}

\begin{proof}
    We begin by taking $\vect f \in L^2(\R^3;\R^3)$ and applying the Fourier transform, denoted by $\mathcal{F}(\cdot)$: 
    \begin{alignat}{3}
        &\mathcal{F}\left( \left( \frac{1}{\eps^{\gamma}}\mathcal{A}^{\rm hom} + I\right)^{-1} \left( I - \Xi_\eps  \right)  \vect f \right) (\xi) \nonumber\\
        &= \left(\frac{1}{\eps^{\gamma}}\left(iX_\xi \right)^*\A^{\rm hom}i X_\xi  + I\right)^{-1} \mathcal{F}\left( \left( I - \Xi_\eps \right) \vect f \right) (\xi)  \qquad
        &&\text{As $\mathbb{A}^{\rm hom}$ is constant in space.} \nonumber\\
        &= \left(\frac{1}{\eps^{\gamma}}\left(iX_\xi \right)^*\A^{\rm hom}i X_\xi  + I\right)^{-1} \mathcal{F}(\vect f)(\xi)  \mathbbm{1}_{ \C^3 \setminus (2\pi\eps)^{-1}Y' }(\xi)
        &&\text{By \eqref{eqn:smoothing_cutoff}.} \nonumber\\
        &=: \mathcal{F}(\vect u)(\xi) \mathbbm{1}_{ \C^3 \setminus (2\pi\eps)^{-1}Y' }(\xi). \label{eqn:l2l2drop_step1}
     \end{alignat}

    By \eqref{eqn:achihom_coer_bdd}, we have 
    \begin{equation}\label{eqn:l2l2drop_step2}
        |\mathcal{F}(\vect u)(\xi)| \leq  \frac{ |\mathcal{F}(\vect f)(\xi)|}{\frac{\nu_1 |\xi|^2}{\eps^\gamma} +1 }.
    \end{equation}
    
    If we assume further that $\xi \in \C^3 \setminus (2\pi\eps)^{-1}Y'$, then
    \begin{equation}\label{eqn:l2l2drop_step3}
        |\mathcal{F}(\vect u)(\xi)| \leq C\frac{\eps^{\gamma + 2} |\mathcal{F}(\vect f)(\xi)| }{ 1 + \eps^{\gamma + 2}} \leq C \eps^{\gamma + 2} |\mathcal{F}(\vect f)(\xi)|.
    \end{equation}
    
    Now combine \eqref{eqn:l2l2drop_step1} and \eqref{eqn:l2l2drop_step3} to get
    \begin{equation}
        \begin{split}
            \left\lVert\mathcal{F}\left( \left( \frac{1}{\varepsilon^{\gamma}}\mathcal{A}^{\rm hom} + I\right)^{-1} \left( I-  \Xi_\varepsilon  \right)  \vect f \right)\right\Vert_{L^2(\C^3;\C^3)} \leq C\varepsilon^{ \gamma + 2}\left\Vert \mathcal{F}(\vect f) \right\Vert_{L^2(\C^3;\C^3)}, 
        \end{split}
    \end{equation}
    which, together with Plancherel's identity, verifies \eqref{eqn:l2l2drop}.

    
    As for \eqref{eqn:l2h1drop}, we modify \eqref{eqn:l2l2drop_step1} to get
    \begin{align}
        \mathcal{F}\left( \nabla \left( \frac{1}{\eps^{\gamma}}\mathcal{A}^{\rm hom} + I\right)^{-1} \left( I - \Xi_\eps  \right)  \vect f \right) (\xi)
        = (i\widetilde{X}_\xi) \underbrace{\left(\frac{1}{\eps^{\gamma}}\left(iX_\xi \right)^*\A^{\rm hom}i X_\xi  + I\right)^{-1} \mathcal{F}(\vect f)(\xi)}_{\widehat{\vect u}(\xi)}  \mathbbm{1}_{ \C^3 \setminus (2\pi\eps)^{-1}Y' }(\xi),
    \end{align}
    where $\widetilde{X}_\xi: L^2 \to L^2$ is defined by $\widetilde{X}_\xi \vect u = \vect u \otimes \xi$ (compare this with $X_\xi$). This gives us
    \begin{align}
        |\xi| |\mathcal{F}(\vect u)(\xi)| 
        &\stackrel{\text{\eqref{eqn:l2l2drop_step2}}}{\leq} \frac{ |\xi| }{\frac{\nu_1 |\xi|^2}{\eps^\gamma} +1 } |\mathcal{F}(\vect f)(\xi)|
        \leq C \frac{ |\xi| }{\frac{|\xi|^2}{\eps^\gamma} } |\mathcal{F}(\vect f)(\xi)| \nonumber\\
        &\leq C \frac{\eps^{\gamma}}{|\xi|} |\mathcal{F}(\vect f)(\xi)|
        \stackrel{\text{for $\xi \in \C^3 \setminus (2\pi\eps)^{-1} Y'$}}{\leq} C \eps^{\gamma+1} |\mathcal{F}(\vect f)(\xi)|,
    \end{align}
    which, together with Plancherel's identity, verifies \eqref{eqn:l2h1drop}.
\end{proof}

\subsection{Agreement of first-order corrector with classical results}\label{sect:first_corr_agree}

Section \ref{sect:smoothingdrop} shows that the zeroth-order approximation to $\vect u_\eps = (\frac{1}{\eps^{\gamma}} \mathcal{A}_\eps + I )^{-1} \vect f$ may be upgraded from $(\frac{1}{\eps^{\gamma}} \mathcal{A}^{\rm hom} + I )^{-1} \Xi_\eps \vect f$ to $(\frac{1}{\eps^{\gamma}} \mathcal{A}^{\rm hom} + I )^{-1} \vect f$, the latter is in agreement with the classical two-scale expansion \cite[Chpt 7]{cioranescu_donato}. In this section, we show a similar result for the first-order approximation, namely, that $\mathcal{R}_{\rm corr,1}^{\eps} \vect f$ coincides with the first-order term in the classical two-scale expansion.

In order to see this, we focus on the cell-problem \eqref{correctordefinition}. Let $\{\vect e_1, \ldots, \vect e_6 \}$ be an orthonormal basis for $\R_{\text{sym}}^{3 \times 3}$. For each $i = 1, \ldots, 6$, let $\vect N_i$ be the solution to problem: Find $\vect N_i \in H_{\#}^1( Y;\R^3)$ that solves
\begin{equation}\label{Nidefinitionreal}
        \begin{cases}
            \int_{ Y} \A \left( \vect e_i + \simgrad \vect N_i \right ) : \simgrad \vect v \, dy = 0, \quad \forall \vect v \in  H_{\#}^1( Y;\R^3), \\
            \int_Y \vect N_i = 0.
        \end{cases}
\end{equation}
(That is, $\vect N_i = \vect u^{\vect e_i}$, in the notation of \eqref{correctordefinition}.) Note that $\vect N_i$ is also the solution to problem: Find $\vect N_i \in H_{\#}^1( Y;\C^3)$ that solves
\begin{equation}\label{Nidefinitioncomplex}
        \begin{cases}
            \int_{ Y} \A \left( \vect e_i + \simgrad \vect N_i \right ) : \overline{\simgrad \vect v} \, dy = 0, \quad \forall \vect v \in  H_{\#}^1( Y;\C^3), \\
            \int_Y \vect N_i = 0.
        \end{cases}
\end{equation}
This follows from an argument similar to the proof of Proposition \ref{prop:hom_matrix}: Write $\vect v = \Re(\vect v) + i\Im(\vect v)$, then split \eqref{Nidefinitioncomplex} into two sub-problems, then show that $\vect N_i$ solves both sub-problems, then use the fact that the solution to \eqref{Nidefinitioncomplex} is unique.

Fix $\vect f \in L^2(\R^3;\R^3)$. Denote by $\vect u_0$ the solution to the homogenized equation with smoothened right hand side $\Xi_\eps \vect f$.
\begin{equation}
\label{u0solution}
    \vect u_0 := \left(\tfrac{1}{\eps^{\gamma}} \mathcal{A}^{\rm hom} + I \right)^{-1}\Xi_\varepsilon \vect f \in H^1(\R^3;\R^3).
\end{equation}
Since $\left( \simgrad \vect u_0 \right) (x) \equiv \left( \simgrad_x \vect u_0 \right) (x) \in \R_{\text{sym}}^{3 \times 3}$, for a.e.\,$x \in \R^3$, this allows us to write
\begin{equation}\label{linearity_u0solution}
    \simgrad_x \vect u_0(x) 
    = \sum_{i=1}^6 \langle \simgrad_x \vect u_0(x), \vect e_i \rangle_{\R^{3\times 3}} \, \vect e_i
    = \sum_{i=1}^6 \left( \simgrad_x \vect u_0(x) : \vect e_i \right) \, \vect e_i,
    \quad \text{for a.e. $x \in \R^3$}.
\end{equation}
Recall that the problem \eqref{correctordefinition} is linear in $\vect \xi$ in the following sense: If we write $\vect u^{\vect \xi}$ for the solution to \eqref{correctordefinition} with respect to $\vect \xi \in \R_{\text{sym}}^{3\times 3}$, then $\vect u^{\sum_{i=1}^6 a_i \vect e_i} = \sum_{i=1}^6 a_i \vect u^{\vect e_i}$, where $a_i \in \R$. This implies that for almost every $x \in \R^3$,
\begin{align}\label{eqn:first_order_classical}
    \vect u_1(x, \cdot)
    := \vect u^{\simgrad_x \vect u_0(x)} 
    = \sum_{i=1}^6 \left( \simgrad_x \vect u_0(x) : \vect e_i \right) \vect u^{\vect e_i}
    = \sum_{i=1}^6 \left( \simgrad_x \vect u_0(x) : \vect e_i \right) \vect N_i
    \in H^1_\#(Y;\R^3).
\end{align}

\begin{remark}[More notation]\label{rmk:first_order_collect}
Let us attempt write the right-hand side of \eqref{eqn:first_order_classical} in a compact form. First, we collect the solutions $\mathbf{N}_i$ to the cell-problem \eqref{Nidefinitionreal} into a single function
\begin{equation}
    \N = \vect N_1 \vect e_1 + \cdots + \vect N_6 \vect e_6 \in H_{\#}^1(Y;\R^3) \otimes \R_{\text{sym}}^{3\times 3}.
\end{equation}
That is, $\N(y)\vect e_i = \vect N_i(y)$. With $\N(y)$ in hand, we may now write \eqref{eqn:first_order_classical} as follows: 
\begin{align}\label{u1solution}
        \vect u_1(x,y) 
        = \N(y)  \simgrad_x \vect u_0 (x).
    \end{align}
    The reader may find \eqref{u1solution} useful when comparing with classical formulae such as those found in \cite{cioranescu_donato}.
    %
\end{remark}

Finally, we define the corrector term $\vect u_{\rm corr}^\eps$ as follows
\begin{equation}
\label{ucorrepsilon}
    \vect u_{\rm corr}^\eps(x) 
    := \eps \,\vect u_1\left(x,\frac{x}{\eps}\right) 
    = \eps \, \N \left(\frac{x}{\eps}\right)  \simgrad_x \vect u_0(x), \quad \text{for a.e.\,$x \in \R^3$.}
\end{equation}
In other words, $\vect u_{\rm corr}^\eps$ coincides with the first-order term of the classical two-scale expansion.

Our goal is to show that $\vect u_{\text{corr}}^\eps = \mathcal{R}_{\text{corr,1}}^\eps \vect f$. To this end, let us prove a result relating $\vect u_{\text{corr}}^\eps$ and the solution to the $\chi$ dependent cell-problem \eqref{focorr2}.

\begin{proposition}\label{prop:first_corr_link}
    For each $\vect f \in L^2(\R^3,\R^3)$, the function $\vect u_{\rm corr}^\eps$ also belongs to $L^2(\R^3;\R^3)$. Moreover, we have
    \begin{equation}
         \mathcal{G}_\eps \left(\vect u_{\rm corr}^\eps\right)(y,\chi) = \widetilde{\vect u}_1(y,\chi), 
    \end{equation}
    where $\widetilde{\vect u}_1(\cdot, \chi) \in H_\#^1(Y;C^3)$ is the unique solution to the ($\chi$ dependent) problem 
    \begin{equation}\label{eqn:u1tilde_problem}
        \begin{cases}
            \int_{Y} \A \left(iX_\chi    \widetilde{\vect u}_0  +  \simgrad   \widetilde{\vect u}_1 \right) : \overline{ \simgrad   \vect v}  = 0, \quad \forall \vect v \in  H_{\#}^1( Y;\C^3), \\
            \int_Y \widetilde{\vect u}_1 = 0,
        \end{cases}
    \end{equation}
    with $\widetilde{\vect u}_0 
    = \widetilde{\vect u}_0(\chi) 
    :=  \left(\int_{Y'}^\oplus \left(\frac{1}{\eps^{\gamma + 2}} \mathcal{A}^{\rm hom}_\chi + I_{\C^3} \right)^{-1} S \, d\chi \right) \mathcal{G}_\eps \vect f \in \C^3$.
\end{proposition}

\begin{proof}
   \textbf{Step 1:} We verify an identity relating the solution $\widetilde{\vect u}_1$ to the function $\widetilde{\vect u}_0$
    \begin{equation}\label{eqn:u0_vs_u1_solution}
        \widetilde{\vect u}_1 (y,\chi) 
        = \sum_{i=1}^6 \left( iX_\chi \widetilde{\vect u}_0 (\chi) : \vect e_i \right) \vect N_i (y)
        = \N(y)  iX_\chi \widetilde{\vect u}_0 (\chi), 
        \quad \text{for } y \in Y,~\chi \in Y'.
    \end{equation}
    To prove this, we note that $iX_\chi \widetilde{\vect u}_0(\chi) \in \C_{\text{sym}}^{3\times 3}$, for a.e.~$\chi \in Y'$. We also note that the basis $\{ \vect e_1, \ldots, \vect e_6 \}$ for $\R_{\text{sym}}^{3\times 3}$ is also an orthonormal basis for the complex vector space $\C_{\text{sym}}^{3\times 3}$. This allows us to write
    \begin{align}
        iX_\chi \widetilde{\vect u}_0(\chi) = \sum_{i=1}^6 \left( iX_\chi \widetilde{\vect u}_0(\chi) : \vect e_i \right) \vect e_i, 
        \quad \text{for a.e.~$\chi \in Y'$}.
    \end{align}
    (Compare this with \eqref{linearity_u0solution}.) Now consider the problem: For $\vect \xi \in \C_{\text{sym}}^{3\times 3}$, find $\widetilde{\vect u}^{\vect \xi} \in H_{\#}^1(Y;\C^3)$ that solves
    \begin{equation}\label{eqn:problem_utildexi}
        \begin{cases}
            \int_{Y} \A ( \vect \xi +  \simgrad   \widetilde{\vect u}^{\vect \xi} ) : \overline{ \simgrad   \vect v} = 0, \quad \forall \vect v \in  H_{\#}^1( Y;\C^3), \\
            \int_Y \widetilde{\vect u}^{\vect \xi} = 0,
        \end{cases}
    \end{equation}
    This problem is linear in $\vect \xi$. Now recall that $\vect N_i$, defined as the solution to \eqref{Nidefinitionreal}, is also the unique solution to problem \eqref{Nidefinitioncomplex}. (That is, $\vect N_i = \widetilde{\vect u}^{\vect e_i}$ in the notation of \eqref{eqn:problem_utildexi}.) This implies that for almost every $\chi \in Y'$,
    \begin{align}
        \widetilde{\vect u}_1(\cdot, \chi) 
        = \widetilde{\vect u}^{iX_\chi \widetilde{\vect u}_0(\chi)}
        = \sum_{i=1}^6 \left( iX_\chi \widetilde{\vect u}_0(\chi) : \vect e_i \right) \widetilde{\vect u}^{\vect e_i}
        = \sum_{i=1}^6 \left( iX_\chi \widetilde{\vect u}_0(\chi) : \vect e_i \right) \vect N_i \in H_{\#}^1(Y;\C^3).
    \end{align}
    This proves \eqref{eqn:u0_vs_u1_solution}.

    \textbf{Step 2:} The formula \eqref{eqn:u0_vs_u1_solution} shows that $\widetilde{\vect u}_1$ is continuous in $\chi$ (by \eqref{Achihom_ahom}) and $H_{\#}^1$ in $y$, and thus $\widetilde{\vect u}_1 \in L^2(Y\times Y')$. This means that its inverse Gelfand transform $\mathcal{G}_\eps^\ast \widetilde{\vect u}_1$ is well-defined, which we may now compute:
    \begin{align}
        \mathcal{G}_\eps^\ast \left( \widetilde{\vect u}_1 \right)(x) 
        &\stackrel{\text{\eqref{eqn:u0_vs_u1_solution}}}{=} \frac{1}{\left(2\pi\eps\right)^{3/2}} \int_{Y'} e^{i\chi\cdot \frac{x}{\eps}} \, \left( \N\left(\frac{x}{\eps}\right)  iX_\chi \widetilde{\vect u}_0 (\chi) \right) \, d\chi \nonumber\\
        &= \N\left(\frac{x}{\eps}\right)  \frac{1}{\left(2\pi\eps\right)^{3/2}} \int_{Y'} e^{i\chi\cdot } \,  iX_\chi \widetilde{\vect u}_0 \, d\chi
        =  \N\left(\frac{x}{\eps}\right)  \mathcal{G}_\eps^\ast \left( iX_\chi \widetilde{\vect u}_0 \right)(x), \quad x\in\mathbb{R}^3. \label{eqn:u1tildesolution_step1}
    \end{align}
    Here, $iX_\chi \widetilde{\vect u}_0$, as a function of $y$ and $\chi$, is constant in $y$.

    \textbf{Step 3:} We now express the term $\mathcal{G}_\eps^\ast \left( iX_\chi \widetilde{\vect u}_0 \right)$ in terms of $\vect u_0$. First, we observe that 
    \begin{align}
        \mathcal{G}_\eps \vect u_0
        \stackrel{\text{\eqref{u0solution}}}{=} \mathcal{G}_\eps \left(\tfrac{1}{\eps^{\gamma}} \mathcal{A}^{\rm hom} + I \right)^{-1}\Xi_\eps \vect f
        \stackrel{\text{\eqref{eqn:vn_ahom_resolvents}}}{=} \left(\int_{Y'}^\oplus \left(\frac{1}{\eps^{\gamma + 2}} \mathcal{A}^{\rm hom}_\chi + I_{\C^3} \right)^{-1} S \, d\chi \right) \mathcal{G}_\eps \vect f. \label{gu0solution}
    \end{align}
    In particular, $\mathcal{G}_\eps \vect u_0$, as a function of $y$ and $\chi$, is constant in $y$. This means that 
    \begin{align}
        \eps \mathcal{G}_\eps (\simgrad \vect u_0)
        &\stackrel{\text{\eqref{gelfandsymetricgradientformula}}}{=} \cancel{\eps} \frac{1}{\cancel{\eps}} \left( \cancel{\simgrad_y (\mathcal{G}_\eps \vect u_0)} + iX_\chi(\mathcal{G}_\eps \vect u_0) \right) \nonumber\\
        &\stackrel{\text{\eqref{gu0solution}}}{=} iX_\chi \left(\int_{Y'}^\oplus \left(\frac{1}{\eps^{\gamma + 2}} \mathcal{A}^{\rm hom}_\chi + I_{\C^3} \right)^{-1} S \, d\chi \right) \mathcal{G}_\eps \vect f 
        = iX_\chi \widetilde{\vect u}_0. \label{eqn:gu0_vs_u0tilde}
    \end{align}
    Now apply $\mathcal{G}_\eps^\ast$ to both sides of \eqref{eqn:gu0_vs_u0tilde}, and we get
    \begin{align}\label{eqn:u1tildesolution_step2}
        \mathcal{G}_\eps^\ast \left( iX_\chi \widetilde{\vect u}_0 \right) = \eps \, \simgrad \vect u_0.
    \end{align}

    \textbf{Conclusion:} By combining \eqref{eqn:u1tildesolution_step1} and \eqref{eqn:u1tildesolution_step2}, we obtain $\mathcal{G}_\eps^\ast \widetilde{\vect u}_1 = \vect u_{\rm corr}^\eps$ and $\vect u_{\rm corr}^\eps \in L^2(\R^3;\R^3)$.
\end{proof}
We may now link the corrector operator $\mathcal{R}_{\rm corr, 1}^\eps$ to $\vect u_{\text{corr}}^\eps$ as follows: For $\vect f \in L^2(\R^3;\R^3)$,
\begin{alignat*}{3}
    \mathcal{G}_\eps \mathcal{R}_{\rm corr, 1}^\eps \vect f (y,\chi)
    &= \left( \int_{Y'}^{\oplus} \mathcal{B}_{\text{corr}, 1,\chi} d\chi \right) 
    \bigg(\int_{Y'}^{\oplus} \left( \frac{1}{\eps^{\gamma+2}} \mathcal{A}_\chi^{\text{hom}} + I_{\C^3} \right)^{-1} &&S d\chi \bigg) \mathcal{G}_\eps \vect f (y,\chi)
    \quad \text{By \eqref{trans_back_corr1} and \eqref{eqn:R1corr_fctcalc}.}\\
    &= \left( \int_{Y'}^{\oplus} \mathcal{B}_{\text{corr}, 1,\chi} d\chi \right) \widetilde{\vect u}_0 (y,\chi) 
    &&\text{Definition of $\widetilde{\vect u}_0$.}\\
    &= \widetilde{\vect u}_1 (y,\chi)
    &&\parbox{30em}{Compare the problems \eqref{eqn:B1corr} and \eqref{eqn:u1tilde_problem}.}\\
    &= \mathcal{G}_\eps \vect u_{\text{corr}}^\eps (y,\chi)
    &&\text{By Proposition \ref{prop:first_corr_link}.}
\end{alignat*}

We have therefore demonstrated that $\mathcal{R}_{\rm corr,1}^{\eps} \vect f$ coincides with the first-order term $\vect u_{\text{corr}}^\eps$ of the classical two-scale expansion.

\section{Acknowledgements}
YSL is supported by a scholarship from the EPSRC Centre for Doctoral Training in Statistical Applied Mathematics at Bath (SAMBa), under the project EP/S022945/1. JŽ is supported by the Croatian Science Foundation under Grant agreement No. IP-2018-01-8904 (Homdirestroptcm) and under Grant agreement No. IP-2022-10-5181 (HOMeOS). For the purpose of open access, YSL has applied a Creative Commons Attribution (CC-BY) licence to any Author Accepted Manuscript version arising from the submission. YSL and JŽ would like to thank Prof.~Kirill~D.~Cherednichenko and Prof.~Igor~Velčić for their guidance and helpful discussions throughout this project.

\appendix
\renewcommand{\thesubsection}{\Alph{subsection}} 
\counterwithin{theorem}{section} 

\section{Korn's inequalities and its consequences}\label{sect:useful_ineq}
Here we put Korn's inequalities as can be found in \cite{Oleinik2012MathematicalPI}.

\begin{theorem}[The second Korn's inequality, \cite{Oleinik2012MathematicalPI}, Theorem 2.4] 
\label{appendixkorn}
Let $\Omega \subset \R^3$ be a bounded open set with Lipschitz boundary. There exists a constant $C>0$ which depends only on $\Omega$, such that for every $\vect u \in H^1(\Omega;\C^3)$, we have the following estimate:
\begin{equation}
    \label{korninequality1}
    \lVert \vect u \rVert_{H^1(\Omega;\C^3)} \leq C \left(\lVert \vect u \rVert_{L^2(\Omega;\C^3)} +  \left\lVert\simgrad \vect u \right\rVert_{L^2(\Omega;\C^{3\times 3})} \right).
\end{equation}
\end{theorem}

\begin{theorem}[\cite{Oleinik2012MathematicalPI}, Theorem 2.5, Corollary 2.6]
\label{appendixkorn3}
    Let $\Omega \subset \R^3$ be a bounded open set with Lipschitz boundary. Let $\mathcal{R} \subset H^1(\Omega;\C^3)$ be the set of rigid displacements: $\mathcal{R} = \{\vect w = Ax + \vect c, \, A \in \C^{3 \times 3}, \, A^\top = -A, \, \vect c\in \C^3\}.$    
    There exists a constant $C>0$ which depends only on $\Omega$, such that for every $\vect u \in H^1(\Omega;\C^3)$ which satisfies $\langle \vect u, \vect w \rangle_{L^2(\Omega;\C^3)} = 0 $, $\forall \vect w \in \mathcal{R}$, we have the following estimate:
    \begin{equation}
\label{korninequality3}
    \lVert \vect u \rVert_{H^1(\Omega;\C^3)} \leq C \left\lVert\simgrad \vect u \right\rVert_{L^2(\Omega;\C^{3\times 3})},
\end{equation}
\end{theorem}

\begin{corollary}
 There exists a constant $C>0$ such that we have the following estimate:
\begin{equation}
\label{korninequality33}
    \lVert \vect u \rVert_{H^1(Y;\C^3)} \leq C \left\lVert \simgrad \vect u \right\rVert_{L^2(Y;\C^{3\times 3})},
\end{equation}
for every $\vect u \in H^1_\#(Y;\C^3)$ satisfying $\int_Y \vect u = 0$.
\end{corollary}
\begin{proposition}
Let $\Omega \subset \R^3$ be a bounded open set. There exist constants $C_0,C_1>0$, independent of $\chi \in Y'$, such that
\begin{equation}
\label{equivalenceofnorms}
    C_0 \left\lVert \vect u \right\rVert_{H^1(\Omega;\C^3)} \leq  \left\lVert e^{i\chi y}\vect u \right\rVert_{H^1(\Omega;\C^3)} \leq  C_1 \left\lVert \vect u \right\rVert_{H^1(\Omega;\C^3)}, \quad \forall \chi \in Y', \quad \vect u \in H^1(\Omega;\C^3).
\end{equation}
\end{proposition}
\begin{proof}
We clearly have (due to $|e^{i\chi y}|=1$) 
\begin{equation}
    \left\lVert e^{i\chi y}\vect u\right\rVert_{L^2(\Omega;\C^3)} = \left\lVert \vect u\right\rVert_{L^2(\Omega;\C^3)}, \quad \left\lVert \nabla \left( e^{i\chi y}\vect u \right)\right\rVert_{L^2(\Omega;\C^3)} = \left\lVert  \nabla \vect u +  \vect u \otimes i\chi\right\rVert_{L^2(\Omega;\C^3)}.
\end{equation}
Now, we calculate
\begin{equation}
\begin{split}
         \left\lVert e^{i\chi y}\vect u\right\rVert_{H^1(\Omega;\C^3)}^2 = \left\lVert \vect u\right\rVert_{L^2(\Omega;\C^3)}^2 + \left\lVert  \nabla \vect u +  \vect u \otimes i\chi\right\rVert_{L^2(\Omega;\C^3)}^2 & \leq \left(1 + |\chi|^2 \right) \left\lVert \vect u\right\rVert_{L^2(\Omega;\C^3)}^2 + \left\lVert  \nabla  \vect u\right\rVert_{L^2(\Omega;\C^3)}^2 \\
         & \leq C\left\lVert \vect u \right\rVert_{H^1(\Omega;\C^3)}^2.
\end{split}
\end{equation}
Conversely: 
\begin{equation}
\begin{split}
        \left\lVert \vect u \right\rVert_{H^1(\Omega;\C^3)}^2 = \left\lVert \vect u\right\rVert_{L^2(\Omega;\C^3)}^2 + \left\lVert  \nabla  \vect u\right\rVert_{L^2(\Omega;\C^3)}^2 & \leq \left\lVert \vect u\right\rVert_{L^2(\Omega;\C^3)}^2 + \left\lVert  \nabla \vect u+  \vect u \otimes i\chi\right\rVert_{L^2(\Omega;\C^3)}^2 + |\chi|^2\left\lVert \vect u\right\rVert_{L^2(\Omega;\C^3)}^2 \\
        & = \left(1 + |\chi|^2 \right)\left\lVert \vect u\right\rVert_{L^2(\Omega;\C^3)}^2 + \left\lVert  \nabla \vect u+  \vect u \otimes i\chi\right\rVert_{L^2(\Omega;\C^3)}^2 \\ & \leq C\left\lVert e^{i\chi y}\vect u \right\rVert_{H^1(\Omega;\C^3)}.
\end{split}
\end{equation}
\end{proof}

\begin{proposition}
Let $\Omega$ be a bounded open set with Lipschitz boundary. There exists a constant $C$ such that for every $\vect u \in H^1(\Omega;\C^3)$ and $|\chi|\in Y'$, we have the following estimate: 
\begin{equation}
\label{korninequality5}
    \lVert \vect u \rVert_{H^1(\Omega;\C^3)} \leq C\left( \left\lVert \left(\simgrad + iX_{\chi}\right) \vect u \right\rVert_{L^2(\Omega;\C^{3 \times 3})} + \lVert  \vect u \rVert_{L^2(\Omega;\C^3)} \right),
\end{equation}
where the constant $C$ depends only on the domain $\Omega$.
\end{proposition}
\begin{proof}
By plugging $\vect w = e^{i \chi y} \vect u$, $\vect u\in H^1(\Omega;\C^3)$, in the inequality \eqref{korninequality1}, we obtain: 
\begin{equation}
\begin{split}
        \left\lVert \vect u\right\rVert_{H^1(\Omega;\C^3)} \leq \left\lVert \vect w\right\rVert_{H^1(\Omega;\C^3)} & \leq  C\left( \lVert \simgrad \vect w \rVert_{L^2(\Omega;\C^{3 \times 3})} + \lVert  \vect w \rVert_{L^2(\Omega;\C^3)} \right) \\
        & =  C\left( \left\lVert \left(\simgrad + i X_\chi \right) \vect u \right\rVert_{L^2(\Omega;\C^{3 \times 3})} + \lVert  \vect u \rVert_{L^2(\Omega;\C^3)} \right),
\end{split}
\end{equation}
where we have also used \eqref{equivalenceofnorms}.
\end{proof}

\begin{lemma}
    There exists $C>0$ such that for all $\vect a, \vect b \in \C^3$ we have:
    \begin{equation}
    \label{rankonesymformula}
         \lvert \vect a \otimes \vect b\rvert 
         \leq C \lvert \sym( \vect a \otimes \vect b )\rvert.
    \end{equation}
\end{lemma}
\begin{proof}
   Let $\vect a = (a_1,a_2,a_3), \vect b=(b_1,b_2,b_3) \in \C^3$. The following calculation proves the claim: 
   \begin{align*}
       \lvert \sym(\vect a \otimes \vect b )\rvert^2 
       &= \sum_{i,j}\left( \frac{a_ib_j+a_jb_i}{2}\right)^2 = \sum_{i}a_i^2 b_i^2 + \sum_{i \neq j} \left( \frac{a_ib_j+a_jb_i}{2}\right)^2 \\ 
       &= \sum_i a_i^2b_i^2 + \frac{1}{2} \sum_{i \neq j}a_i^2 b_j^2 + \sum_{i \neq j}a_ib_ja_jb_i \\
       &= \frac{1}{2}\sum_{i} a_i^2b_i^2 + \frac{1}{4}\sum_{i \neq j}a_i^2 b_j^2 + \frac{1}{8} \sum_{i \neq j}(a_ib_j + a_jb_i)^2 + \frac{1}{8}\sum_{i \neq j}(a_ib_i + a_jb_j)^2 \\
       &\geq \frac{1}{2}\sum_{i} a_i^2b_i^2 + \frac{1}{4}\sum_{i \neq j}a_i^2 b_j^2 \geq C \sum_{i,j}a_i^2 b_j^2 = C \lvert \vect a \otimes \vect b\rvert^2. \qedhere
   \end{align*}
\end{proof}

\begin{proposition}\label{prop:coercive_est}
    There exists a constant $C_{\text{fourier}}>0$ such that for every $\vect u \in H_\#^1(Y;\C^3)$, $\chi \in Y'\setminus\{0\}$ we have the following estimates: 
    \begin{equation}
    \label{estimate1}
        \left\lVert \vect u \right\rVert_{L^2(Y;\C^3)} \leq \frac{C_{\text{fourier}}}{|\chi|}\left\lVert \left(\simgrad + iX_{\chi}\right) \vect u \right\rVert_{L^2(Y;\C^{3 \times 3})}.
    \end{equation}
    \begin{equation}
    \label{estimate11}
        \left\lVert \nabla \vect u \right\rVert_{L^2(Y;\C^{3 \times 3})} \leq C_{\text{fourier}}\left\lVert \left(\simgrad + iX_{\chi}\right) \vect u \right\rVert_{L^2(Y;\C^{3 \times 3})}.
    \end{equation}
    \begin{equation}
    \label{estimate12}
        \left\lVert \vect u - \fint_Y \vect u\right\rVert_{L^2(Y;\C^3)} \leq C_{\text{fourier}}\left\lVert \left(\simgrad + iX_{\chi}\right) \vect u \right\rVert_{L^2(Y;\C^{3 \times 3})}.
    \end{equation}
\end{proposition}

\begin{proof}
    We begin by taking a function $\vect u \in H_\#^1(Y;\C^3)$ and considering its Fourier series decomposition.
    \begin{equation}
        \vect u = \sum_{k \in \Z^3} a_k e^{2\pi i k \cdot y}, \quad  \nabla \vect u = \sum_{k\in \Z^3} e^{2\pi i k \cdot y} a_k \otimes \left(2 \pi i k \right), \quad \vect u - \fint_Y \vect u= \sum_{k \in \Z^3 \setminus \{0\}} a_k e^{2\pi i k \cdot y}.
    \end{equation}
    Plancherel's formula yields: 
    \begin{equation}
        \left\lVert \vect u\right\rVert_{L^2(Y;\C^3)}^2 = \sum_{k\in \Z^3} |a_k|^2, \quad \left\lVert \nabla \vect u\right\rVert_{L^2(Y;\C^{3 \times 3})}^2 = \sum_{k\in \Z^3} |2 \pi |^2|a_k |^2| k|^2, \quad \left\lVert \vect u- \fint_Y \vect u\right\rVert_{L^2(Y;\C^3)}^2 = \sum_{k\in \Z^3 \setminus \{0\}} |a_k|^2.
    \end{equation}
    Furthermore, we have: 
    \begin{equation}
    \begin{split}
            \nabla \vect u + \vect u \otimes i\chi= \sum_{k\in \Z^3} e^{2\pi i k \cdot y} a_k \otimes \left(2 \pi i k + i\chi \right).
    \end{split}
    \end{equation}
    Now we calculate:
    \begin{equation}
        \begin{split}
            \left\lVert \nabla \vect u + \vect u \otimes (i\chi)\right\rVert_{L^2(Y;\C^{3 \times 3})}^2 = \sum_{k\in \Z^3} |a_k \otimes \left(2 \pi i k + i\chi \right)|^2 & = \sum_{k\in \Z^3 \setminus \{0\}} |a_k \otimes \left(2 \pi i k + i\chi \right)|^2 + |a_0 \otimes  i\chi |^2 \\
            & = \sum_{k\in \Z^3 \setminus \{0\}} |a_k |^2| 2 \pi i k + i\chi |^2 + |a_0 |^2|  \chi |^2.
        \end{split}
    \end{equation}
    Now, since $\chi \in Y' = [-\pi,\pi)^3$, if at least one $(k)_j \geq 1$, it is clear that 
    \begin{equation}
        |2 \pi i k + i \chi|^2 \geq C,
    \end{equation}
    where the constant $C>0$ does not depend on $\chi$ and $k \in \Z^3\setminus \{0\}$. Now, we obtain:

    \begin{equation}\label{estimate1_working}
         \left\lVert \nabla \vect u + \vect u \otimes (i\chi)\right\rVert_{L^2(Y;\C^{3 \times 3})}^2  =  \sum_{k\in \Z^3 \setminus \{0\}} |a_k |^2| 2 \pi i k + i\chi |^2 + |a_0 |^2|  \chi |^2 \geq  \sum_{k\in \Z^3 \setminus \{0\}} C|a_k |^2 = C\left\lVert \vect u- \fint_Y \vect u\right\rVert_{L^2(Y;\C^3)}^2.
    \end{equation}
    Furthermore: 
    \begin{equation}\label{estimate11_working}
         \left\lVert \nabla \vect u + \vect u \otimes (i\chi)\right\rVert_{L^2(Y;\C^{3 \times 3})}^2  =  \sum_{k\in \Z^3 \setminus \{0\}} |a_k |^2| 2 \pi i k + i\chi |^2 + |a_0 |^2|  \chi |^2 \geq  \sum_{k\in \Z^3} C|\chi|^2|a_k |^2 = C|\chi|^2\left\lVert \vect u\right\rVert_{L^2(Y;\C^3)}^2.
    \end{equation}
    Also:
    \begin{equation}\label{estimate12_working}
    \begin{split}
             \left\lVert \nabla \vect u + \vect u \otimes (i\chi)\right\rVert_{L^2(Y;\C^{3 \times 3})}^2     \geq  \sum_{k\in \Z^3 \setminus \{0\}} C|a_k |^2 |2 \pi|^2 |k|^2 = C\left\lVert \nabla \vect u\right\rVert_{L^2(Y;\C^3)}^2 .
    \end{split}
    \end{equation}
    On the other hand, we have
    \begin{alignat*}{3}
   \left\lVert \nabla \vect u + \vect u \otimes (i\chi)\right\rVert_{L^2(Y;\C^{3 \times 3})}^2
    &= \sum_{k\in \Z^3} |a_k \otimes \left(2 \pi i k + i\chi \right)|^2 
    && \text{By Plancherel's formula.}\\
    & \leq C\sum_{k\in \Z^3} |\sym (a_k \otimes \left(2 \pi i k + i\chi \right))|^2 \quad \quad \quad 
    &&\text{By the estimate \eqref{rankonesymformula}.}\\
    &= C \left\lVert \sym \left(\nabla \vect u + \vect u \otimes (i\chi) \right)\right\rVert_{L^2(Y;\C^{3 \times 3})}^2
    &&\parbox{30em}{Again by the Plancherel's formula.}\\
    &= C \left\lVert \simgrad \vect u + iX_\chi \vect u \right\rVert_{L^2(Y;\C^{3 \times 3})}^2.
    &&\text{By the definition of the operator $iX_\chi$.}
\end{alignat*}
This concludes the proof.
\end{proof}

\section{Estimates for Section \ref{sect:the_asymp_method}}\label{sect:routine_estimates}

In this appendix we justify the asymptotic procedure by verifying the claimed estimates in Section \ref{sect:the_asymp_method_cycle1est} (cycle 1) and Section \ref{sect:the_asymp_method_cycle2est} (cycle 2). We remind the reader the following notation from \eqref{eqn:dist_to_spec}
\begin{equation}
    D_{\text{hom}}(z) := \text{dist} (z, \sigma(\tfrac{1}{|\chi|^2} \mathcal{A}_\chi^{\text{hom}} ) ), \qquad
    D(z) := \text{dist} (z, \sigma(\tfrac{1}{|\chi|^2} \mathcal{A}_\chi ) ).
\end{equation}

\subsection*{Cycle 1}

\begin{proof}[Proof of \eqref{bound11}.]
    Since $\vect u_0$ is a constant, $\| \vect u_0 \|_{H^1} = \| \vect u_0 \|_{L^2} = |\vect u_0|$. Since $z \in \rho(\frac{1}{|\chi|^2} \mathcal{A}_\chi^{\text{hom}})$, we have
    \begin{equation*}
        \| \vect u_0 \|_{H^1} \stackrel{\text{\eqref{leadingorderterm}}}{=} \left\| \left( \frac{1}{|\chi|^2} \mathcal{A}_\chi^{\text{hom}} - zI_{\C^3} \right)^{-1} S \vect f \right\|_{L^2}
        \leq \frac{1}{D_\text{hom}(z)} \| \vect f \|_{L^2}.
    \end{equation*}
    This gives us \eqref{bound11}.
\end{proof}

\begin{proof}[Proof of \eqref{bound12}.]
    By testing \eqref{corr2} with $\vect u_1$, applying Assumption \ref{coffassumption}, \eqref{eqn:Xchi_est}, and H\"older's inequality, we have
    \begin{align}
        &\| \simgrad \vect u_1 \|_{L^2}^2 \leq C |\chi| \| \vect u_0 \|_{L^2} \| \simgrad \vect u_1 \|_{L^2} \nonumber\\
        \Longleftrightarrow &\| \simgrad \vect u_1 \|_{L^2} \leq C |\chi| \| \vect u_0 \|_{L^2}, \label{eqn:simgrad_u1}
    \end{align}
    where the constant $C>0$ depends on $\nu$, $\|\mathbb{A}_{jl}^{ik}\|_{L^\infty}$, and $C_{\text{symrk1}}$. Now $\| \vect u_0 \|_{L^2}$ can be bounded above by $\text{dist} (z, \sigma(\frac{1}{|\chi|^2} \mathcal{A}_\chi^{\text{hom}}) )^{-1} \| \vect f \|_{L^2}$ (by \eqref{bound11}), and $\| \simgrad \vect u_1 \|_{L^2}$ can be bounded below by $C \| \vect u_1 \|_{H^1}$ (by \eqref{korninequality33}, using $\int \vect u_1 = 0$.) This gives us the desired inequality, \eqref{bound12}.
\end{proof}

\begin{proof}[Proof of \eqref{bound13}.]
    By testing \eqref{correctoru2} with $\vect u_2$, applying Assumption \ref{coffassumption}, \eqref{eqn:Xchi_est}, and H\"older's inequality, we have
    \begin{align*}
        \nu \| \simgrad \vect u_2 \|_{L^2}^2 &\leq C \bigg[ |\chi| \| \vect u_1 \|_{L^2} \| \simgrad \vect u_2 \|_{L^2} + |\chi| \| \simgrad \vect u_1 \|_{L^2} \| \vect u_2 \|_{L^2} + |\chi|^2 \| \vect u_0 \|_{L^2} \| \vect u_2 \|_{L^2} \bigg] \\
        & \qquad + |\chi|^2 \bigg[ |z| \|\vect u_0 \|_{L^2} + \| \vect f \|_{L^2} \bigg] \| \vect u_2 \|_{L^2}.
    \end{align*}
    On the right-hand side, we apply the inequalities \eqref{bound11} to $\vect u_0$, \eqref{bound12} to $\vect u_1$, \eqref{eqn:simgrad_u1} to $\simgrad \vect u_1$, and \eqref{korninequality33} to $\vect u_2$, giving us
    \begin{align}\label{eqn:simgrad_u2}
        \nu \| \simgrad \vect u_2 \|_{L^2}^{\cancel{2}} 
        &\leq C \left[\frac{\max \{1,|z|\}}{D_\text{hom}(z)} + 1 \right] |\chi|^2 \| \vect f \|_{L^2} \cancel{\| \simgrad \vect u_2 \|_{L^2}.}
    \end{align}
    Finally, we note that the left-hand side can be bounded below by $C \| \vect u_2 \|_{H^1}$ (by \eqref{korninequality33}, using $\int \vect u_2 = 0$.) This gives us \eqref{bound13}.
\end{proof}


\begin{proof}[Proof of \eqref{bounderror1}.]
    By H\"older's inequality, Assumption \ref{coffassumption}, and \eqref{eqn:Xchi_est},
    \begin{align*}
        |\mathcal{R}_\text{err} (\vect v)|
        &\leq C \bigg[ |\chi|^2 \| \vect u_1 \|_{L^2} \| \vect v \|_{L^2}
        + |\chi| \| \simgrad \vect u_2 \|_{L^2} \| \vect v \|_{L^2}
        + |\chi| \| \vect u_2 \|_{L^2} \| \simgrad \vect v \|_{L^2}
        + |\chi|^2 \| \vect u_2 \|_{L^2} \| \vect v \|_{L^2}
        \bigg] \\
        &\qquad + |\chi|^2 \bigg( |z| \| \vect u_1 \|_{L^2} 
        + |z| \| \vect u_2 \|_{L^2} \bigg) \| \vect v \|_{L^2}
    \end{align*}
    Now apply \eqref{bound12} to $\| \vect u_1 \|_{L^2}$, \eqref{bound13} to $\| \vect u_2 \|_{L^2}$, and \eqref{eqn:simgrad_u2} in the above inequality, and we obtain
    \begin{align*}
        |\mathcal{R}_\text{err} (\vect v)| 
        &\leq C \left[ \frac{\max \{1,|z|^2\}}{D_\text{hom}(z)} + \max \{1,|z|\} \right] |\chi|^3 \| \vect f \|_{L^2} \bigg[ \| \vect v \|_{L^2} + \| \simgrad \vect v \|_{L^2} \bigg] \\
        &\leq C \left[ \frac{\max \{1,|z|^2\}}{D_\text{hom}(z)} + \max \{1,|z|\} \right] |\chi|^3 \| \vect f \|_{L^2} \| \vect v \|_{H^1}.  \qedhere
    \end{align*}
\end{proof}

\subsection*{Cycle 2}

\begin{proof}[Proof of \eqref{bound21}.]
    Since $\vect u_0^{(1)}$ is a constant, $\| \vect u_0^{(1)} \|_{H^1} = \| \vect u_0^{(1)} \|_{L^2} = |\vect u_0^{(1)}|$. Note that \eqref{newconstantcorrector} may be written as
    \begin{align}
        \left( \frac{1}{|\chi|^2} \mathcal{A}_\chi^{\text{hom}} - z \right) \vect u_0^{(1)} = -\frac{1}{|\chi|^2} S (iX_\chi)^* \A (\simgrad \vect u_2 + iX_\chi \vect u_1)
    \end{align}
    Since $z \in \rho(\frac{1}{|\chi|^2} \mathcal{A}_\chi^{\text{hom}} )$, we compute
    \begin{alignat*}{2}
        \| \vect u_0^{(1)} \|_{H^1} 
        &\leq \frac{1}{D_\text{hom}(z)} \frac{1}{|\chi|^2} \| (iX_\chi)^* \A (\simgrad \vect u_2 + iX_\chi \vect u_1) \|_{L^2} \qquad\\
        &\leq \frac{C}{D_\text{hom}(z)} \frac{1}{|\chi|} \bigg[ \| \simgrad \vect u_2 \|_{L^2} + |\chi| \| \vect u_1 \|_{L^2} \bigg] \quad
        &&\text{By Assumption \ref{coffassumption} and \eqref{eqn:Xchi_est}.} \\
        & \leq C \left[ \frac{\max \{1,|z|\}}{D_\text{hom}(z)^2} + \frac{1}{D_\text{hom}(z)} \right] |\chi| \| \vect f \|_{L^2}
        &&\parbox{15em}{By applying \eqref{eqn:simgrad_u2} to $\| \simgrad \vect u_2 \|$, \\ and \eqref{bound12} to $\vect u_1$.}
    \end{alignat*}
    This gives us \eqref{bound21}, as required.
\end{proof}

\begin{proof}[Proof of \eqref{bound22}.] 
    The proof proceeds in exactly the same way as that of \eqref{bound12}. Testing \eqref{corr2new} with $\vect u_1^{(1)}$, applying Assumption \ref{coffassumption}, \eqref{eqn:Xchi_est}, H\"older's inequality, we have
    \begin{equation}
        \| \simgrad \vect u_1^{(1)} \|_{L^2} \leq C |\chi| \| \vect u_0^{(1)} \|_{L^2} \stackrel{\text{\eqref{bound21}}}{\leq} C \left[ \frac{\max \{1,|z|\}}{D_\text{hom}(z)^2} + \frac{1}{D_\text{hom}(z)} \right] |\chi|^2 \| \vect f \|_{L^2}. \label{eqn:simgrad_u11}
    \end{equation}
    Now apply \eqref{korninequality33} (as $\int \vect u_1^{(1)} = 0$) to bound $\| \simgrad \vect u_1^{(1)} \|_{L^2}$ below in terms of $\| \vect u_1^{(1)} \|_{H^1}$.
\end{proof}

\begin{proof}[Proof of \eqref{bound23}.]
    By testing \eqref{correctoru3new} with $\vect u_2^{(1)}$, applying Assumption \ref{coffassumption}, \eqref{eqn:Xchi_est}, and H\"older's inequality,
    \begin{align*}
        &\nu \| \simgrad \vect u_2^{(1)} \|_{L^2}^2 
        \leq C \bigg[ 
        |\chi| \left( \| \vect u_1^{(1)} \|_{L^2} + \| \vect u_2 \|_{L^2} \right) \| \simgrad \vect u_2^{(1)} \|_{L^2} \\
        &\qquad + |\chi| \left( \| \simgrad \vect u_1^{(1)} \|_{L^2} + \| \simgrad \vect u_2 \|_{L^2} \right) \| \vect u_2^{(1)} \|_{L^2}
        + |\chi|^2 \left( \| \vect u_0^{(1)} \|_{L^2} + \| \vect u_1 \|_{L^2} \right) \| \vect u_2^{(1)} \|_{L^2} \bigg] \\
        &\qquad + |z| |\chi|^2 \left( \| \vect u_0^{(1)} \|_{L^2} + \| \vect u_1 \|_{L^2}  \right) \| \vect u_2^{(1)} \|_{L^2}.
    \end{align*}
    Now apply to the above, the following four inequalities:
    \begin{alignat*}{2}
        \| \vect u_2^{(1)} \|_{L^2} 
        &\leq C \| \simgrad \vect u_2^{(1)} \|_{L^2}.
        \span\span \qquad\qquad\qquad \text{By \eqref{korninequality33}, since $\int_Y \vect u_2^{(1)} = 0$.} \\
        \| \vect u_1^{(1)} \|_{L^2} + \| \vect u_2 \|_{L^2}
        &\leq C \left[ \frac{\max \{1,|z|\}}{D_\text{hom}(z)^2} + \frac{\max \{1,|z|\}}{D_\text{hom}(z)} + 1 \right] |\chi|^2 \| \vect f \|_{L^2}.
        \qquad &&\text{By \eqref{bound22} and \eqref{bound13}.} \\
        \| \simgrad \vect u_1^{(1)} \|_{L^2} + \| \simgrad \vect u_2 \|_{L^2}
        &\leq C \left[ \frac{\max \{1,|z|\}}{D_\text{hom}(z)^2} + \frac{\max \{1,|z|\}}{D_\text{hom}(z)} + 1 \right] |\chi|^2 \| \vect f \|_{L^2}.
        \quad &&\text{By \eqref{eqn:simgrad_u11} and \eqref{eqn:simgrad_u2}.} \\
        \| \vect u_0^{(1)} \|_{L^2} + \| \vect u_1 \|_{L^2}
        &\leq C \left[ \frac{\max \{1,|z|\}}{D_\text{hom}(z)^2} + \frac{1}{D_\text{hom}(z)} \right] |\chi| \| \vect f \|_{L^2}.
        &&\text{By \eqref{bound21} and \eqref{bound12}.}
    \end{alignat*}
    Altogether, this gives us 
    \begin{align}\label{eqn:simgrad_u21}
        \| \simgrad \vect u_2^{(1)} \|_{L^2} \leq C 
        \left[ \frac{\max \{1,|z|^2\}}{D_\text{hom}(z)^2} 
        + \frac{\max \{1,|z|\}}{D_\text{hom}(z)} + 1 \right] |\chi|^3 \| \vect f \|_{L^2}.
    \end{align}
    Another application of \eqref{korninequality33} to the left-hand side of \eqref{eqn:simgrad_u21} completes the proof.
\end{proof}

\begin{proof}[Proof of \eqref{bounderror2}.]
    By H\"older's inequality, Assumption \ref{coffassumption}, and \eqref{eqn:Xchi_est}, 
    \begin{align*}
        |\mathcal{R}_{\text{err}}^{(1)} (\vect v)| 
        &\leq C \bigg[
        |\chi|^2 \| \vect u_2 \|_{L^2} \| \vect v \|_{L^2}
        + |\chi|^2 \| \vect u_1^{(1)} \|_{L^2} \| \vect v \|_{L^2}
        + |\chi| \| \simgrad \vect u_2^{(1)} \|_{L^2} \| \vect v \|_{L^2}
        + |\chi| \| \vect u_2^{(1)} \|_{L^2} \| \simgrad \vect v \|_{L^2} \\
        &\qquad + |\chi|^2 \| \vect u_2^{(1)} \|_{L^2} \| \vect v \|_{L^2}
        + |z| |\chi|^2 \| \vect u_1^{(1)} \|_{L^2} \| \vect v \|_{L^2}
        + |z| |\chi|^2 \| \vect u_2^{(1)} \|_{L^2} \| \vect v \|_{L^2}
        + |z| |\chi|^2 \| \vect u_2 \|_{L^2} \| \vect v \|_{L^2}
        \bigg].
    \end{align*}
    Now apply to the above, the following three inequalities,
    \begin{alignat*}{2}
        \| \vect u_2 \|_{L^2} + \| \vect u_1^{(1)} \|_{L^2} 
        &\leq C \left[ \frac{\max \{1,|z|\}}{D_\text{hom}(z)^2} 
        + \frac{\max \{1,|z|\}}{D_\text{hom}(z)} + 1 \right] |\chi|^2 \| \vect f \|_{L^2}. \qquad
        &&\text{By \eqref{bound13} and \eqref{bound22}.} \\
        \| \simgrad \vect u_2^{(1)} \|_{L^2} 
        &\leq C \left[ \frac{\max \{1,|z|^2\}}{D_\text{hom}(z)^2} 
        + \frac{\max \{1,|z|\}}{D_\text{hom}(z)} + 1 \right] |\chi|^3 \| \vect f \|_{L^2}.
        &&\text{By \eqref{eqn:simgrad_u21}.} \\
        \| \vect u_2^{(1)} \|_{L^2} 
        &\leq C \left[ \frac{\max \{1,|z|^2\}}{D_\text{hom}(z)^2} 
        + \frac{\max \{1,|z|\}}{D_\text{hom}(z)} + 1 \right] |\chi|^3 \| \vect f \|_{L^2}.
        &&\text{By \eqref{bound23}.}
    \end{alignat*}
    Altogether, this gives us
    \begin{align*}
        |\mathcal{R}_\text{err}^{(1)} (\vect v)| 
        &\leq C \left[ \frac{\max \{1,|z|^3\}}{D_\text{hom}(z)^2} 
        + \frac{\max \{1,|z|^2\}}{D_\text{hom}(z)} 
        + \max \{1,|z|\} \right] 
        |\chi|^4 \| \vect f \|_{L^2} 
        \bigg[ \| \vect v \|_{L^2} + \| \simgrad \vect v \|_{L^2} \bigg] \\
        &\leq C \left[ \frac{\max \{1,|z|^3\}}{D_\text{hom}(z)^2} 
        + \frac{\max \{1,|z|^2\}}{D_\text{hom}(z)} 
        + \max \{1,|z|\} \right]
        |\chi|^4 \| \vect f \|_{L^2} 
        \| \vect v \|_{H^1}.  \qedhere
    \end{align*}
\end{proof}

\section{An identity regarding the smoothing operator}
The following lemma relates the scaled Gelfand transform $\mathcal{G}_\eps$ with the Fourier transform, in the special case where the function, after applying $\mathcal{G}_\eps$, is constant in $y\in Y$:

\begin{lemma}\label{lem:fourier_vs_gelfand}
    For all $\theta \in \eps^{-1} Y'$, the following identity holds
    \begin{align}
        \int_Y (\mathcal{G}_\eps \vect u)(y, \eps\theta) dy = \left(\frac{1}{2\pi\eps}\right)^{3/2}\mathcal{F}(\vect u)\left( \frac{\theta}{2\pi} \right).
    \end{align}
    Here, $\mathcal{F}(\vect u)$ denotes the Fourier transform of $\vect u \in L^2(\R^3;\R^3)$, i.e. $\mathcal{F}(\vect u)(\theta) = \int_{\R^3} \vect u (y) e^{-2\pi i \theta \cdot y} dy$.
\end{lemma}

\begin{proof}
    First, let us express the action of the Gelfand transform (\ref{gelfand}), in terms of the scaled quasimomentum $\theta = \eps^{-1}\chi$:
    \begin{equation}
        (\mathcal{G}_\varepsilon \vect u)(y,\varepsilon\theta):= \left(\frac{\varepsilon}{2\pi}\right)^{3/2} \sum_{n\in \Z^3}e^{-i\eps\theta \cdot (y+n)}\vect u(\varepsilon(y+n)), \quad y \in Y, \quad \theta \in \eps^{-1} Y',
    \end{equation}
    The result now follows from the following computation: 
    \begin{equation}
    \begin{split}
            \int_Y (\mathcal{G}_\eps \vect u)(y, \eps\theta) dy 
            &= \int_Y \left(\frac{\eps}{2\pi}\right)^{3/2} \sum_{n\in \Z^3}e^{-i\eps \theta \cdot (y+n)}\vect u(\eps(y+n))dy \\
            &= \left(\frac{\eps}{2\pi}\right)^{3/2} \sum_{n\in \Z^3}\int_Ye^{-i\theta \cdot \eps(y+n)}\vect u(\eps(y+n))dy\\
            &= \left(\frac{\eps}{2\pi}\right)^{3/2} \frac{1}{\eps^3} \sum_{n\in \Z^3}\int_{\eps(Y+ n)} e^{-i\theta \cdot y}\vect u( y)dy \\
            &= \left(\frac{1}{2\pi\eps}\right)^{3/2} \int_{\R^3}e^{-i\theta \cdot y}\vect u( y)dy 
            = \left(\frac{1}{2\pi\eps}\right)^{3/2}\mathcal{F}(\vect u)\left( \frac{\theta}{2\pi} \right) \quad \forall \theta \in \eps^{-1}Y'. \qedhere
    \end{split}
    \end{equation}
\end{proof}

\section{An estimate for the weak solution of resolvent problem}

In this appendix, we will assume the following setup: $\mathcal{A}$ is a non-negative self-adjoint operator on a Hilbert space $(H,\langle \cdot, \cdot \rangle_H)$. Let $a$ be the form associated with $\mathcal{A}$. Denote by $\mathcal{D}(\mathcal{A})$ the domain of $\mathcal{A}$, and $\mathcal{D}(a) = X$ the domain of $a$. Since $a$ is closed, $(X, \langle \cdot, \cdot \rangle_a)$ is a Hilbert space, where
\begin{equation}
    \langle \vect u, \vect v \rangle_a := a(\vect u, \vect v) + \langle \vect u, \vect v \rangle_H
    = \langle \mathcal{A}^{1/2}\vect u , \mathcal{A}^{1/2}\vect v\rangle_H + \langle \vect u, \vect v \rangle_H, \qquad \forall \vect u, \vect v \in \mathcal{D}(\mathcal{A}^{1/2}) = \mathcal{D}(a) = X.
\end{equation}

On $X$, we also consider the inner product
\begin{equation}
    \langle \vect u, \vect v \rangle_X:= \langle(\mathcal{A}^{1/2} + I)\vect u, (\mathcal{A}^{1/2} + I) \vect v \rangle_{H}, \quad \forall \vect u, \vect v \in X.
\end{equation}

Denote by $\| \cdot \|_{H}$ the norm corresponding to the inner product $\langle \cdot, \cdot \rangle_H$, and similarly for $\| \cdot \|_{a}$ and $\| \cdot \|_{X}$.

\begin{lemma}\label{lem:norm_equiv_X_a}
    The norm $\| \cdot \|_{X}$ is equivalent to $\| \cdot \|_a$. In particular, $(X, \langle \cdot, \cdot \rangle_X)$ is a Hilbert space.
\end{lemma}

\begin{proof}
    Let $\vect u \in X$. Note that $\| \cdot \|_X = \| (\mathcal{A}^{1/2}+I) \cdot \|_H$. Then,
    \begin{equation*}
        \| \vect u \|_{a}^2 = a(\vect u, \vect u) + \| \vect u\|_H^2 
        \leq \langle \mathcal{A} \vect u, \vect u \rangle_H + \langle \vect u, \vect u \rangle_H + 2 \underbrace{\langle \mathcal{A}^{1/2} \vect u, \vect u \rangle_H}_{\geq 0}
        = \| (\mathcal{A}^{1/2}+I) \vect u \|_H^2
        = \| \vect u \|_X^2,
    \end{equation*}
    and
    \begin{equation*}
        \| (\mathcal{A}^{1/2}+I) \vect u \|_H^2 
        \leq C \left( \| \mathcal{A}^{1/2} \vect u \|_H^2 + \| \vect u \|_H^2 \right)
        = C \left( a(\vect u,\vect u) + \| \vect u \|_H^2 \right) = C \| \vect u \|_a^2. \qedhere
    \end{equation*}
\end{proof}

\begin{proposition}
    \label{prop:abstract_ineq}
    Assume that $\lambda \in \rho(\mathcal{A})$, $\mathcal{R}$ is a bounded linear functional on $(X, \langle \cdot, \cdot \rangle_X)$, and $\vect u \in X$ solves the problem
    \begin{equation}
    \label{abstractweakresolventproblem}
        a(\vect u, \vect v) - \lambda \langle \vect u,\vect v \rangle_H = \mathcal{R}(\vect v), \quad \forall \vect v \in X.
    \end{equation}
    Then the following inequality holds for some $C>0$:
    \begin{equation}
    \label{eqn:absract_ineq}
        \lVert \vect u\rVert_X \leq C \max \left\{1, \frac{|\lambda + 1|}{\rm{dist}(\lambda, \sigma(\mathcal{A}))} \right\} \| \mathcal{R} \|_{X^*}.
    \end{equation}
\end{proposition}

\begin{remark}\label{rmk:equiv_h1_a_X}
    Let us explain how Proposition \ref{prop:abstract_ineq} will be used. We are interested in the case $H = L^2(Y;\C^3)$, $X = H^1_\#(Y;\C^3)$, $\mathcal{A} = \frac{1}{|\chi|^2}  \mathcal{A}_\chi$. (Note the scaling factor $\frac{1}{|\chi|^2}$.) In this case, we have
    \begin{equation}\label{eqn:norm_H1_vs_a_claim}
        \| \cdot \|_{H^1} \leq C_\text{kornper} \| \cdot \|_a \text{ for some $C_{\text{kornper}}>0$ independent of $\chi$,} 
        \quad\text{ and }\quad 
        \| \cdot \|_a \stackrel{\text{Lemma \ref{lem:norm_equiv_X_a}}}{\sim} \| \cdot \|_X.
    \end{equation}
    The inequality in \eqref{eqn:norm_H1_vs_a_claim} follows by applying Assumption \ref{coffassumption} and Proposition \ref{prop:coercive_est}. We refer the reader to Lemma \ref{lem:equiv_h1_a} below for its proof. As for the functional $\mathcal{R}$, we are interested in the case $\mathcal{R} = \frac{1}{|\chi|^2} \mathcal{R}_\text{err}$. We will demonstrate that (see \eqref{bounderror1})
    \begin{align*}
        \left| \mathcal{R} (\vect v) \right| \leq C_{\text{op}} \| \vect v \|_{H^1}, \quad \text{for some constant $C_{\text{op}}$ that may depend on $z$ and $\chi$.}
    \end{align*}
    That is, $\mathcal{R}$ is a bounded linear functional on $(H_{\#}^1, \| \cdot \|_{H^1})$, with operator norm not exceeding $C_\text{op}$. But by \eqref{eqn:norm_H1_vs_a_claim}, $\mathcal{R}$ is a bounded linear functional on $(H_{\#}^1, \| \cdot \|_a)$, with norm not exceeding $C_\text{op}C_\text{kornper}$.

    We are now in a position to apply Proposition \ref{prop:abstract_ineq}. Indeed,
    \begin{align*}
        \| \vect u \|_{H^1} \stackrel{\text{\eqref{eqn:norm_H1_vs_a_claim}}}{\leq} 
        C_{\text{kornper}} \| \vect u \|_a \stackrel{\text{Prop \ref{prop:abstract_ineq}}}{\leq} 
        (C_\text{kornper})^2 C \max \left\{ 1, \frac{|\lambda + 1|}{\text{dist}(\lambda, \sigma(\mathcal{A}))} \right\} C_\text{op}.
    \end{align*}
    
    This is the desired inequality \eqref{bounduerr1}. A similar remark applies to $\frac{1}{|\chi|^2} \mathcal{R}_\text{err}^{(1)}$, by using \eqref{bounderror2}.
\end{remark}

\begin{remark}
    We will eventually restrict our choice of $\lambda \in \rho(\mathcal{A})$ to $\lambda \in \Gamma \subset \rho(\mathcal{A})$, where $\Gamma$ is a compact subset that is uniformly bounded away from $\sigma(\mathcal{A})$. (See Lemma \ref{lemma:contour}.) As a consequence we may further bound \eqref{eqn:absract_ineq} by $C \| \mathcal{R} \|_{X^*}$, where $C>0$ depends on the spectral parameter through $\Gamma$.
\end{remark}

\begin{proof}[Proof of Proposition \ref{prop:abstract_ineq}.]
    Since $(X, \langle \cdot, \cdot \rangle_X)$ is a Hilbert space (Lemma \ref{lem:norm_equiv_X_a}), an application of the Riesz representation theorem to $\mathcal{R}$ implies that there exists an $\vect r \in X$ such that
    \begin{equation}
        \mathcal{R}(\vect v) = \langle \vect r, \vect v \rangle_X = \langle(\mathcal{A}^{1/2} + I)\vect r, (\mathcal{A}^{1/2} + I) \vect v \rangle_{H},
    \end{equation}
    where $\| \mathcal{R} \|_{X^*} = \| r \|_X$. Then the problem \eqref{abstractweakresolventproblem} becomes: Find $\vect u \in X$ such that
    \begin{equation}
    \label{abstractweakresolventproblem2}
         \langle \mathcal{A}^{1/2}\vect u , \mathcal{A}^{1/2}\vect v\rangle_H - \lambda \langle \vect u,\vect v \rangle_H = \langle(\mathcal{A}^{1/2} + I)\vect r, (\mathcal{A}^{1/2} + I) \vect v \rangle_{H}, \quad \forall \vect v \in X. 
    \end{equation}
    Furthermore, by \cite[Section VIII.3, Theorem VIII.4 and Proposition 3]{reed_simon1}, there exists a measure space $(M,\mu)$, unitary operator $\mathcal{U}:H \to L^2(M,d\mu)$ and a real-valued function $f \in L^2(M,d\mu)$ such that 
    \begin{itemize}
        \item $\vect \psi \in \mathcal{D}(\mathcal{A}) \iff f(\cdot) \left(\mathcal{U} \vect \psi \right)(\cdot) \in L^2(M, d\mu)$.
        \item For all $\vect \varphi \in \mathcal{U}(\mathcal{D}(\mathcal{A}))$, we have $(\mathcal{U} \mathcal{A} \mathcal{U}^{-1})\vect \varphi (\cdot) = f(\cdot) \vect \varphi(\cdot)$.
    \end{itemize}
    The equation \eqref{abstractweakresolventproblem2} is now equivalent to: 
    \begin{equation}
    \label{abstractweakresolventproblem3}
         \langle \mathcal{U}\mathcal{A}^{1/2}\mathcal{U}^{-1}\widetilde{\vect u} , \mathcal{U}\mathcal{A}^{1/2}\mathcal{U}^{-1}\widetilde{\vect v}\rangle_{L^2(M,d \mu)} - \lambda \langle \widetilde{\vect u},\widetilde{\vect v} \rangle_{L^2(M,d\mu)} = \langle \widetilde{\vect g}, (\mathcal{U}\mathcal{A}^{1/2}\mathcal{U}^{-1} + I_{L^2(M)}) \widetilde{\vect v} \rangle_{L^2(M,d\mu)}, 
    \end{equation}
    where we have used the notation: $\widetilde{\vect u}:= \mathcal{U} \vect u$, $\widetilde{\vect v}:= \mathcal{U} \vect v$, $\widetilde{\vect g}:= (\mathcal{U}\mathcal{A}^{1/2}\mathcal{U}^{-1} + I_{L^2(M)})\mathcal{U}\vect r$. Notice that
    \begin{equation}\label{eqn:nice_norms}
        \lVert \widetilde{\vect g}\rVert_{L^2(M,d\mu)} = \lVert (\mathcal{U}\mathcal{A}^{1/2}\mathcal{U}^{-1} + I)\mathcal{U}\vect r \rVert_{L^2(M,d\mu)}  = \lVert (\mathcal{A}^{1/2} + I) \vect r \rVert_{H} = \lVert \vect r \rVert_{X} =  \lVert \mathcal{R}\rVert_{X^*}.
    \end{equation}
    Also, the following holds:
    \begin{itemize}
        \item $\forall \vect \varphi \in \mathcal{U}(\mathcal{D}(\mathcal{A}^{1/2}))$  we have $(\mathcal{U} \mathcal{A}^{1/2} \mathcal{U}^{-1}) \vect \varphi (\cdot) = \sqrt{f(\cdot)}  \vect \varphi(\cdot).$
    \end{itemize}
    Thus \eqref{abstractweakresolventproblem3} translates to 
    \begin{equation}
    \label{abstractproblem4}
         \langle \sqrt{ f}\widetilde{\vect u} , \sqrt{ f}\widetilde{\vect v}\rangle_{L^2(M,d \mu)} - \lambda \langle \widetilde{\vect u},\widetilde{\vect v} \rangle_{L^2(M,d\mu)} = \langle \widetilde{\vect g}, (\sqrt{ f} + 1) \widetilde{\vect v} \rangle_{L^2(M,d\mu)}.
    \end{equation}
    Notice that 
    \begin{equation}
        \lVert \sqrt{ f} \widetilde{\vect v}\rVert_{L^2(M,d\mu)} = \lVert \mathcal{A}^{1/2} \vect v \rVert_{H} < \infty, \quad \forall \vect v \in X.
    \end{equation}
    Now, by rewriting \eqref{abstractproblem4} we arrive at
    \begin{equation}
        \int_M f \widetilde{\vect u} \cdot \widetilde{\vect v} \, d\mu - \lambda \int_M \widetilde{\vect u} \cdot \widetilde{\vect v} d\mu = \int_M (\sqrt{ f} + 1) \widetilde{\vect g} \cdot \widetilde{\vect v} \, d\mu, \quad \forall \widetilde{\vect v} \in L^2(M,d\mu), \text{ such that } \lVert \sqrt{ f} \widetilde{\vect v}\rVert_{L^2(M,d\mu)}< \infty,
    \end{equation}
    from which it follows that 
    \begin{equation}
        ( f  - \lambda )  \widetilde{\vect u} = (\sqrt{ f} + 1) \widetilde{\vect g}, \quad \text{ a.e. on } M,
    \end{equation}
    and consequently 
    \begin{equation}
        (\sqrt{ f} + 1)  \widetilde{\vect u} = \frac{(\sqrt{ f} + 1)^2}{( f  - \lambda )} \widetilde{\vect g}, \quad \text{ a.e. on } M.
    \end{equation}
    Now we calculate: 
    \begin{equation}
        \begin{split}
            |(\sqrt{ f} + 1)  \widetilde{\vect u}| & = \frac{|\sqrt{ f} + 1 |^2}{| f  - \lambda |} | \widetilde{\vect g} | \leq 2\frac{|f+1 |}{|f-\lambda|}|\widetilde{\vect g}|
            \leq 2\left( \frac{|f-\lambda|}{|f-\lambda|} + \frac{|\lambda +1 |}{|f-\lambda|}\right) |\widetilde{\vect g}| \\
            &\leq C \max \left\{1, \frac{|\lambda +1|}{\rm{dist}(\lambda, \rm{essRange}(f))} \right\} |\widetilde{\vect g}|, 
            \quad \text{a.e. on } M.
        \end{split}
    \end{equation}
    Since $\rm{essRange}(f) = \sigma(\mathcal{A})$, we see that
    \begin{equation}
        \lVert \vect u\rVert_X  
        = \left\lVert (\sqrt{ f} + 1)  \widetilde{\vect u}\right\rVert_{L^2(M,d\mu)} 
        \leq C \max \left\{1, \frac{|\lambda +1|}{\rm{dist}(\lambda, \sigma(\mathcal{A}))} \right\} \lVert \mathcal{\widetilde{\vect g}}\rVert_{L^2(M,d\mu)}.
    \end{equation}
    Finally, we recall from \eqref{eqn:nice_norms} that $\| \widetilde{\vect g} \|_{L^2(M,d\mu)} = \| \mathcal{R} \|_{X^\ast}$, giving us the desired inequality.
\end{proof}

\begin{lemma}\label{lem:equiv_h1_a}
    In the setting of Remark \ref{rmk:equiv_h1_a_X}, there exist a constant $C_\text{kornper}>0$, independent of $\chi$, such that the inequality $\| \vect u \|_{H^1} \leq C_\text{kornper} \| \vect u \|_a$ holds for all $\vect u \in H_{\#}^1$.
\end{lemma}
\begin{proof}
    Let $\vect u \in H_{\#}^1$. Then,
    \begin{align}\label{eqn:norm_H1_vs_a_step1}
        \frac{\nu}{|\chi|^2} \| (\simgrad + iX_\chi ) \vect u \|_{L^2}^2 + \| \vect u \|_{L^2}^2 \stackrel{\text{Assumption \ref{coffassumption}}}{\leq}
        \| \vect u \|_a^2.
    \end{align}
    By applying \eqref{estimate1} to the left-hand side of \eqref{eqn:norm_H1_vs_a_step1}, we have
    \begin{align}\label{eqn:norm_H1_vs_a_l2bit}
        \left( \frac{(C_{\text{fourier}})^2}{\nu} + 1 \right) \| \vect u \|_{L^2}^2 \leq \| \vect u \|_a^2.
    \end{align}
    Meanwhile, we multiply \eqref{eqn:norm_H1_vs_a_step1} throughout by $|\chi|^2$,
    \begin{align}\label{eqn:norm_H1_vs_a_step2}
        \nu \| (\simgrad + iX_\chi ) \vect u \|_{L^2}^2 + |\chi|^2 \| \vect u \|_{L^2}^2 \leq |\chi|^2 \| \vect u \|_a^2 \leq C \| \vect u \|_a^2
    \end{align}
    By applying \eqref{estimate11} to the left-hand side of \eqref{eqn:norm_H1_vs_a_step2}, we have
    \begin{align}
        \frac{(C_\text{fourier})^2}{\nu} \| \nabla \vect u \|_{L^2}^2
        \leq \frac{(C_\text{fourier})^2}{\nu} \| \nabla \vect u \|_{L^2}^2 + |\chi|^2 \| \vect u \|_{L^2}^2
        \stackrel{\text{\eqref{estimate11} and \eqref{eqn:norm_H1_vs_a_step2}}}{\leq} C \| \vect u \|_a^2. \label{eqn:norm_H1_vs_a_h1bit}
    \end{align}
    Absorbing the constants $\frac{(C_\text{fourier})^2}{\nu} + 1$ and $\frac{(C_\text{fourier})^2}{\nu}$ over to the right-hand sides of \eqref{eqn:norm_H1_vs_a_l2bit} and \eqref{eqn:norm_H1_vs_a_h1bit} respectively, we may now combine \eqref{eqn:norm_H1_vs_a_l2bit} and \eqref{eqn:norm_H1_vs_a_h1bit} to obtain
    \begin{align}
        \| \vect u \|_{H^1} \leq C \| \vect u \|_{a},
    \end{align}
    where the constant $C>0$ does not depend on $\chi$, verifying the desired inequality in \eqref{eqn:norm_H1_vs_a_claim}.
\end{proof}

\nocite{*} 
\renewcommand{\bibname}{References} 
\printbibliography
\addcontentsline{toc}{section}{\refname} 

\end{document}